\documentclass[10pt,reqno]{amsart}
\usepackage{amsmath,amsfonts,amsthm,amssymb,amsxtra,dsfont}
\usepackage{amsxtra, amssymb, mathrsfs, pstricks}
\usepackage[english]{babel}
\usepackage{amsmath,amsfonts,amstext,amsbsy,amsopn,amssymb,amsthm,amscd}
\usepackage{amsxtra,latexsym}
\usepackage{exscale}
\usepackage{graphicx,mathrsfs}
\usepackage{psfrag}
\usepackage{paralist,eucal,enumerate}
\usepackage{color,graphics}
\definecolor{dred}{rgb}{.8,0,0}
\definecolor{ddmagenta}{rgb}{0.7,0,0.9}
\definecolor{ddcyan}{rgb}{0,0.2,1.0}
\definecolor{dblue}{rgb}{0,0,0.7}
\definecolor{ddorange}{rgb}{1,0.5,0}
\definecolor{ddgreen}{rgb}{0,0.4,0.4}
\definecolor{cgreen}{rgb}{0,0.6,0.3}
\definecolor{violet}{rgb}{0.4,0,0.9}

\newtheorem{theorem}{Theorem}[section]
\newtheorem{proposition}[theorem]{Proposition}

\newtheorem{lemma}[theorem]{Lemma}
\newtheorem{corollary}[theorem]{Corollary}

\theoremstyle{definition}

\newtheorem{assumption}[theorem]{Assumption}
\newtheorem{definition}[theorem]{Definition}

\newtheorem{notation}[theorem]{Notation}

\theoremstyle{remark}
\newtheorem{remark}[theorem]{\bf Remark}
\numberwithin{equation}{section}

\def\trait #1 #2 #3 {\vrule width #1pt height #2pt depth #3pt}
\def\fin{
    \trait .3 5 0
    \trait 5 .3 0
    \kern-5pt
    \trait 5 5 -4.7
    \trait 0.3 5 0
\medskip}

\newcommand{\Set}[2]{\{\,#1\,:\,#2\,\}}

\newcommand{\Bset}[2]{\Big\{\,#1\,:\,#2\,\Big\}}
\newcommand{\QED}{\mbox{}\hfill\rule{5pt}{5pt}\medskip\par}
  \def\bbC{{\mathbb C}}

 \def\bbN{{\mathbb N}} 
  \def\bbR{{\mathbb R}}

\newcommand{\R}{\bbR}
\newcommand{\N}{\bbN}
\def\calD{{\mathcal D}} \def\calE{{\mathcal E}} 
  \def\calI{{\mathcal I}}
  \def\calL{{\mathcal L}}
  \def\calO{{\mathcal O}}
  \def\calR{{\mathcal R}}
 \def\calT{{\mathcal T}} \def\calU{{\mathcal U}}
  
 \def\calZ{{\mathcal Z}}
   \def\rmC{{\mathrm C}}
 \def\rmD{{\mathrm D}}  
   
 \def\rmJ{{\mathrm J}}

 \def\rmS{{\mathrm S}}

\newcommand{\sd}{3}
\newcommand{\eps}{\varepsilon}
\newcommand{\umin}{u_{\mathrm{min}}}
\newcommand{\foraa}{\text{for a.a.}}
\newcommand{\il}{q}
\newcommand{\teta}{\vartheta}
\newcommand{\dx}{\, \mathrm{d}x}
\newcommand{\dd}{\, \mathrm{d}}
\newcommand{\abs}[1]{|#1|}
\newcommand{\Lin}{\mathrm{Lin}}
\newcommand{\epsi}{\epsilon}
\newcommand{\pairing}[4]{ \sideset{_{ #1 }}{_{ #2 }}  {\mathop{\langle #3 , #4
\rangle}}}
\newcommand{\norm}[1]{\| #1\|}
\newcommand{\Div}{\mathrm{div}}

\newcommand{\Cost}[4]{\Delta_{#1}(#2; #3, #4)}
\newcommand{\Costname}[1]{\Delta_{#1}}
\newcommand{\Costn}[5]{\Delta_{#1}^{#2}(#3; #4, #5)}
\newcommand{\Var}[4]{\mathrm{Var}_{#1}(#2; [#3, #4])}
\newcommand{\Varname}[1]{\mathrm{Var}_{#1}}
\newcommand{\JVar}[4]{\mathrm{Jump}_{#1}(#2; [#3, #4])}
\newcommand{\JVarn}[5]{\mathrm{Jump}_{#1}^{#2}(#3; [#4, #5])}
\newcommand{\Jvarname}[1]{\mathrm{Jump}_{#1}}
\newcommand{\pVar}[4]{\text{\sl Var}_{#1}(#2; [#3, #4])}
\newcommand{\pVarname}[1]{\text{\sl Var}_{#1}}
\newcommand{\pVarn}[5]{\text{\sl Var}_{#1}^{#2}(#3; [#4, #5])}
\newcommand{\pVarnamen}[2]{\text{\sl Var}_{#1}^{#2}}
\newcommand{\mixed}[1]{\mathrm{M}_{#1}}
\newcommand{\AC}{\mathrm{AC}}
\newcommand{\BV}{\mathrm{BV}}
\newcommand{\ene}[2]{\calI(#1,#2)}
\newcommand{\tildene}[2]{\widetilde{\calI}(#1,#2)}
\newcommand{\sft}{\mathsf{t}}
\newcommand{\sfs}{\mathsf{s}}
\newcommand{\sfr}{\mathsf{r}}
\newcommand{\sfz}{\mathsf{z}}
\newcommand{\sfteta}{\mathsf{\theta}}

\newcommand{\pwc}[3]{\overline{#1}_{#2,#3}}
\newcommand{\pwl}[3]{\widehat{#1}_{#2,#3}}

\newcommand{\Argmin}{\mathrm{Argmin}}
\newcommand{\wt}{\widetilde}
\newcommand{\weakto}{\rightharpoonup}
\newcommand{\weaksto}{\rightharpoonup^*}
\newcommand{\Dtau}[2]{\triangle_{#1}^\tau(#2)}
\newcommand{\DD}[3]{\begin{array}[t]{c}#1\vspace*{-1em}\\_{#2}\vspace*{-.5em}\\_{#3}\end{array}}
\newcommand{\DDshort}[2]{\begin{array}[t]{c}#1\vspace*{-1em}\\_{#2}\end{array}}
\newcommand{\ddd}[3]{\DD{\begin{array}[t]{c}\underbrace{#1}\vspace*{.6em}\end{array}}{\text{\footnotesize #2}}{\text{\footnotesize #3}}}
\newcommand{\dddshort}[2]{\DDshort{\begin{array}[t]{c}\underbrace{#1}\vspace*{.6em}\end{array}}{\text{\footnotesize #2}}}
\newcommand{\ti}[3]{{#1}_{#2}^{#3}}
\newcommand{\bn}[1]{\overline{#1}_{\tau,\nu}}
\newcommand{\ubn}[1]{\underline{#1}_{\tau,\nu}}
\newcommand{\hn}[1]{\widehat{#1}_{\tau,\nu}}
\newcommand{\hnp}[1]{\widehat{#1}'_{\tau,\nu}}
\newcommand{\segm}[2]{(1{-}\theta){#1}+\theta{#2}}
\newcommand{\segmh}[2]{\frac{(1{-}\theta)}2{#1}+\frac\theta 2{#2}}

\begin{document}

\title[Balanced Viscosity solutions for damage]
{Balanced Viscosity solutions \\ to a rate-independent  system for damage}

\author{Dorothee Knees}
\address{Dorothee Knees, Institute of Mathematics, University of Kassel,
Heinrich-Plett Str.~40, 34132 Kassel, Germany. Phone:
+49 0561 8044355}
\email{dknees@uni-kassel.de}

\author{Riccarda Rossi}
\address{Riccarda Rossi, Department DIMI,
University of Brescia, Via Branze 38, 25133 Brescia, Italy. Phone: +39 030 3715721}
\email{riccarda.rossi@unibs.it}

\author{Chiara Zanini}
\address{Chiara Zanini, Department of Mathematical Sciences ``G. L. Lagrange'', Politecnico di
Torino, Corso Duca degli Abruzzi 24, 10129 Torino, Italy. Phone: +39 011 0907510}
\email{chiara.zanini@polito.it}

\maketitle

\date{April 26, 2017}

\begin{abstract}
This article is the third one in a series of papers by the authors on 
vanishing-viscosity solutions to rate-independent damage systems. While in the 
first two papers \cite{krz,KRZ2} the assumptions on the spatial domain $\Omega$ 
were kept as general as possible (i.e.\ nonsmooth domain with mixed boundary conditions), we 
assume here that $\partial\Omega$ is smooth and that the type of boundary 
conditions does not change. This smoother setting allows us to derive enhanced regularity spatial properties both for the displacement and damage fields.
Thus, we are in a position to work with a stronger solution notion at the level of the viscous approximating system. The vanishing-viscosity analysis then leads us to obtain the existence of a stronger solution concept for the rate-independent limit system.
Furthermore, in comparison to \cite{krz,KRZ2}, in our vanishing-viscosity analysis we do not 
switch to an artificial arc-length parameterization of the trajectories but we 
stay with the true physical time. The resulting concept of Balanced 
Viscosity solution to the rate-independent damage system thus encodes a more explicit characterization of the system behavior at time discontinuities of the solution.
\end{abstract}

%%%%%%%%%%%%%%%%%%%%%%%%%%%%%%%%%%%%%%

\section{\bf Introduction}

We consider in a three-dimensional spatial domain $\Omega$ the rate-independent system for 
damage evolution 
\begin{subequations}
\label{dndia}
\begin{alignat}{3}
 \label{dndia-a}
&-\mathrm{div}\big(g(z) \bbC \eps(u{+}u_D)\big)= \ell   &\qquad & 
\text{ in } 
\Omega \times (0,T),
\\
\label{dndia-b}
&    \partial \calR_1(z_t) + A_q z + f'(z) + \tfrac12 g'(z)
   \bbC \varepsilon(u+u_D) : \varepsilon(u+u_D) \ni 0 
  & & \text{ in } \Omega \times (0,T),
\end{alignat}
\end{subequations}
with
 $q > \sd$, 
$A_q$ the $q$-Laplacian type operator
\[
A_q z= -
\mathrm{div} ( (1+|\nabla z|^{2})^{(\il/2) -1} \nabla z)\,,
\] and the $1$-homogeneous dissipation potential 
\[
\calR_1 (v) = \int_\Omega \mathrm{R}_1(v) \dd x 
\qquad \text{with }   \mathrm{R}_1(v)  = \begin{cases} 
|v| & \text{ if } v \leq 0,
\\
\infty & \text{ otherwise}.
\end{cases}
\]
Here, $u:[0,T]\times \Omega \to \R^3$ denotes the displacement field and 
$z:[0,T]\times \Omega \to \R$ characterizes the time and space-dependent damage 
state in the body $\Omega\subset\R^3$. The natural state spaces for $u$ and $z$ 
are $\calU=H_0^1(\Omega;\R^3)$ and $\calZ= W^{1,q}(\Omega)$. 
The energy potential is of the form 
\begin{align*}
 \calE(t,u,z)=\int_\Omega g(z) 
\frac12\bbC(x)\varepsilon(u+u_D(t)):\varepsilon(u+u_D(t)) + f(z) + 
\frac{1}{q}(1+\abs{\nabla z}^2)^\frac{q}{2}\dx - \langle \ell(t),u\rangle, 
\end{align*}
where $\varepsilon(w)=\frac12(\nabla w + \nabla w^T)$ ($w\in \calU$) is the 
strain tensor and $u_D$ denotes the Dirichlet datum. 
Since the underlying energy $\calE(t,\cdot,\cdot)$ in general is nonconvex and 
since $\calR_1$ is of linear growth, solutions to \eqref{dndia} 
might be discontinuous in time. In order to select reasonable jump 
discontinuities we adopt here the vanishing-viscosity approach
to the weak solvability of rate-independent systems, pioneered in \cite{efendiev-mielke} 
and developed both  for abstract rate-independent systems, cf.\ e.g.\ 
\cite{MRS10a, mie-cime, MRS16},  and for applied problems in fracture and 
plasticity, see for instance \cite{KnMiZa07?ILMC, DalDesSol11, BabFraMor12, 
CL17}.  In the context of damage, in addition to the previously mentioned 
\cite{krz,KRZ2}  we quote the recent \cite{CL16, Negri16}. Let us 
stress that, in all of these papers the vanishing-viscosity analysis is 
performed by suitably adapting the original reparameterization technique of  
\cite{efendiev-mielke}. In 
\cite{KneNeg17}, a time-incremental alternate minimization scheme for a 
damage model of Ambrosio-Tortoreli type  (without viscous regularization) was 
investigated. It turned out that in the time-continuous limit this procedure 
results in a class of solutions that is closely related (but not identical) to 
those obtained by vanishing viscosity limits. Also here, the   
reparameterization technique of \cite{efendiev-mielke} was applied.

\par
Hence,
we approximate the rate-independent flow 
rule for the damage parameter by  its viscous  regularization, and thus address 
the \emph{rate-dependent} system
\begin{subequations}
\label{dndia-eps}
\begin{alignat}{3}
\label{dndia-eps-a}
&-\mathrm{div}(g(z) \bbC \eps(u{+}u_D)=\ell   &\qquad& \text{ in } 
\Omega 
\times (0,T),
\\
&\label{dndia-eps-b}
    \partial \calR_1(z_t) +\epsilon z_t + A_q z + f'(z) + \tfrac12 g'(z)
   \bbC \varepsilon(u+u_D) : \varepsilon(u+u_D) \ni 0  
   && \text{ in } \Omega \times (0,T),
\end{alignat}
\end{subequations}
where  the 
underlying regularized dissipation potential is given by  
\begin{equation}
\label{Reps-intro}
\calR_\epsi:L^2(\Omega) \to [0,+\infty] \text{ given by } 
\calR_\epsi(v): = \calR_1 (v)  + \frac{\epsi}2 \|v\|_{L^2(\Omega)}^2\,,
\end{equation}
and $\epsi>0$ is the viscosity parameter. 
The goal is to perform the limit passage as $\epsi\downarrow 0$ from \eqref{dndia-eps} to 
\eqref{dndia},  without
switching  to an artificial arc-length reparameterization of the trajectories, 
 but
\emph{staying with the true physical time}. 
 The basics for this approach to the construction 
% reparameterizing time 
of  the resulting concept of \emph{Balanced Viscosity}   ($\BV$)
solutions to the limit rate-independent system were set in \cite{MRS10a, MRS16}
for abstract rate-independent systems in finite-dimensional and 
infinite-dimensional Banach spaces, respectively.  A notable feature of this 
vanishing-viscosity technique  is that it allows for a \emph{direct} limit 
passage from the \emph{time discrete} version of  \eqref{dndia-eps} to 
\eqref{dndia}, as the viscosity parameter $\epsi$ and the time discretization 
step $\tau$ \emph{simultaneously} tend to zero   with $\frac{\epsi}{\tau}
\to \infty$.  This provides a \emph{constructive approach} to Balanced Viscosity 
solutions of system \eqref{dndia} which could also be further advanced  from a 
numerical viewpoint.

While the techniques applied here  have been developed in an abstract context in 
 \cite{MRS16}, 
let us emphasize that the existence and convergence results therein,
%MieRoSav, JEMS 2016 
 (in particular \cite[Thms.\ 3.11 and 3.12]{MRS16}), are not directly 
applicable to the present damage system.  The main point is that, in contrast  
 to \cite{MRS16} in our setting the dissipation potential $\calR_1$ may take the value $+\infty$ to enforce the unidirectionality of the damaging process. This causes additional technical difficulties for the 
derivation of uniform a priori bounds. Moreover, the definition of 
 $\BV$ solution has  to be carefully tailored to accomodate this irreversibility constraint.  
 Further analytical difficulties occur due to the presence of the quadratic term on the right-hand side of the differential inclusion \eqref{dndia-b}, which at a first glance 
belongs to $L^1(\Omega)$, only. This necessitates a careful study of the 
spatial regularity properties of the displacement and the damage fields, which was already initiated in \cite{krz, KRZ2}.

The main results of this paper are the following: 
\begin{description}
\item[
\textbf{Regularity}] 
Thanks to the assumed smoothness of $\partial \Omega$ (made precise in Section 
\ref{ss:2.1}) and the assumption $q>3$ on the $q$-Laplacian regularization in \eqref{dndia-b}, which ensures enough spatial regularity for the coefficient $g(z)$ of the elasticity operator in \eqref{dndia-a}, solutions $u=u(t,z)$ of \eqref{dndia-a} 
 belong to $H^2(\Omega)\cap W^{1,p}(\Omega)$ for every $p\geq 1$ if the external data 
$\ell,u_D$ are smooth enough. We derive explicit bounds for the corresponding 
norms of $u$  in terms of $z$ by adapting  arguments from \cite{BabMil12} 
%BAbadjan millot 
to our situation.  
These results improve the integrability properties of the quadratic 
term in \eqref{dndia-b} and in \eqref{dndia-eps-b} and allow us to test (a 
regularized version of) 
\eqref{dndia-eps-b} by $\partial_t A_q z$, which ultimately guarantees that 
$\rmD_z\calE(t,u(t,z),z)\in L^2(\Omega)$, again with uniform bounds, see 
Section \ref{ss:3.1}. Let us mention that, in the case of the standard  Laplacian regularization (i.e.\ $q=2$), this regularity estimate was first proposed in   \cite{bfl}
for doubly nonlinear differential inclusions in phase change modeling. 
\par
Based on the improved integrability property of $\rmD_z\calE(t,u(t,z),z)$ we 
may consider subdifferentials and convex conjugate functions of the 
dissipation potentials with respect to the $L^2(\Omega)$ duality, instead of 
 the $\calZ-\calZ^*$ duality. Furthermore, based on these results we derive 
a (generalized) $\lambda$-convexity property of the energy functional, 
(cf.\ Corollary \ref{coro-fre}), and a chain rule identity (cf.\ Lemma \ref{l:ch-rule}).  
The latter is essential for the existence proof of  $\BV$ solutions for 
the damage system.
\par
 This chain rule identity was not available in the earlier 
\cite{KRZ2}, which still addressed  the case of a $q$-Laplacian regularization in the damage flow rule, whereas in \cite{krz} some technical difficulties were smeared out by taking as regularizing operator a (less physical) fractional Laplacian. 
 Hence, in \cite{KRZ2} we had to deal with a 
weaker notion of vanishing-viscosity solution compared to the present paper. In 
particular, in \cite{KRZ2} it could  be shown that the vanishing-viscosity limits satisfied an energy-dissipation inequality but, due to the lack of 
an appropriate chain rule this could not be improved to an energy-dissipation 
identity. 
\item[
\textbf{Existence and approximation of $\BV$ solutions}]
The concept of $\BV$ solution to the rate-independent system \eqref{dndia}
consists of 
a (local) stability condition and of an energy-dissipation balance that encodes the (possible) onset of viscous behavior in the jump regime. More precisely, 
let $u(t,z)\in \calU$ be the unique solution of \eqref{dndia-a} and 
$\calI(t,z):= \calE(t,u(t,z),z)$ the reduced energy. We call 
 a curve $z\in L^\infty(0,T;\calZ)\cap \BV([0,T]; 
L^2(\Omega))$ with $\rmD_z\calI(\cdot,z(\cdot)) \in L^\infty(0,T;L^2(\Omega))$  a \emph{Balanced Viscosity} solution to  \eqref{dndia} if $z$ satifies 
 the local stability \eqref{intro-locstab} and the energy-dissipation 
balance \eqref{intro-endi}
\begin{align}
\label{intro-locstab}
 \tag{$\mathrm{S}_\text{\rm loc}$}
 -\rmD_z\calI(t,z(t))\in \partial\calR_1(0)&\quad \text{ for all }t\in 
[0,T]\backslash\rmJ_z, 
 \\
 \label{intro-endi}
 \tag{$\mathrm{ED}$}
 \pVar{\mathfrak{f}}z{0}{t}   + \calI(t,z(t))=\calI(0,z(0)) + 
\int_0^t \partial_t\calI(r,z(r))\dd r&\quad\text{ for all }t\in [0,T],
\end{align}
where  $\rmJ_z$ denotes the countable jump set of 
$z$.  
The quantity $\pVar{\mathfrak{f}}{\cdot} {0}{t}$ is a total variation 
functional that encompasses both the dissipation, %along the solution path 
with respect 
to the $1$-homogeneous potential $\calR_1$, in continuous parts of the solution, as well as the dissipation 
at jump discontinuities. At jump discontinuities it reflects the viscous 
regularization term from \eqref{dndia-eps-b}. While referring to Section~\ref{ss:5.1} for its precise definition   (and to \cite{MRS16} for more comments on it), we may mention here its structure at a jump from $z_-$ to 
$z_+$ for $t\in \rmJ_z$. Indeed, the \emph{jump contribution} $ \Cost{\mathfrak{f}}t{z_-}{z_+}$ to $ \pVar{\mathfrak{f}}z{0}{t}  $ is given by  
\begin{gather}
\label{intro.cost} 
 \Cost{\mathfrak{f}}t{z_-}{z_+}: = \inf_{ \vartheta\in 
\calT_{t}^\varrho(z_-,z_+)}
    \int_{0}^{1}
\mathfrak{f}_t(\teta(r),\teta'(r))\dd r\,,\\
\mathfrak{f}_t(\teta,\teta')= \calR_1(\teta') + \norm{\teta'}_{L^2(\Omega)} 
\inf_{\xi\in \partial\calR_1(0)} \norm{-\rmD_z\calI(t,\teta) - 
\xi}_{L^2(\Omega)}\,,
\end{gather} 
where $\calT_{t}^\varrho(z_-,z_+)$ denotes the set of admissible transition 
curves connecting $z_-$ with $z_+$ and satisfying certain properties. 
%bounds depending on 
%$\varrho>0$. 
\par
The appearance of the term from \eqref{intro.cost}
in the vanishing-viscosity limit of \eqref{dndia-eps} 
can be motivated by a 
comparison with the energy-dissipation balance that is valid for solutions of 
the viscous system \eqref{dndia-eps}. In fact, we will show in Theorem~\ref{thm:exist} that solutions to \eqref{dndia-eps} exist and that 
they satisfy for all $t\in [0,T]$ the relation 
\begin{align}
\label{intro:evisc} 
 \int_0^t \calR_\epsi (\dot z_\epsi) + \calR^*_\epsi(-\rmD_z\calI(r, 
z_\epsi(r)))\dd r + \calI(t, z_\epsi(t))=\calI(0,z(0)) + \int_0^t 
\partial_t\calI(r,z_\epsi(r))\dd r
\end{align}
with $\calR_\epsi^*(\eta)= \frac{1}{2\epsi} \inf_{\xi\in \partial\calR_1(0)} 
\norm{\eta - \xi}_{L^2(\Omega)}^2$ provided that $\eta\in L^2(\Omega)$. It 
turns out that 
\[
 \mathfrak{f}_t(t,z,v)=\inf_{\eps>0}\left(\calR_\epsi(v) 
+\calR_\epsi^*(-\rmD_z\calI(t,z))\right). 
\]
The challenge here is to perform a sharp limit analysis for $\epsi\to 0$ in 
order to show that  
the dissipation integral in \eqref{intro:evisc} tends to 
 $\pVar{\mathfrak{f}}z{0}{t}$  as $\epsi\to 0$.

The \underline{\bf main result of this paper}, Theorem~\ref{thm:van-visc-eps-tau},
states the existence of  Balanced Viscosity solutions to the damage 
system \eqref{dndia} under suitable assumptions on the data $z_0,u_D$ and 
$\ell$. 
They are obtained from a vanishing-viscosity analysis of the 
time 
discretized version of the viscous system \eqref{dndia-eps} as the time step 
size $\tau$, the viscosity parameter $\epsi$ and the ratio $\tau/\epsi$ tend 
to zero. The convergence of discrete solutions of corresponding 
numerical schemes to $\BV$ solutions is an immediate consequence. Let us stress that, with the techniques from \cite{MRS16} we could prove the existence of $\BV$ solutions also by  taking the vanishing-viscosity analysis of the \emph{time-continuous} system in  \eqref{dndia-eps}, as standardly done in works on the  vanishing-viscosity approach to rate-independent systems. Here we have opted for this simultaneous limit passage to highlight the constructive character of this approach. 
\end{description}
The paper is organized as follows: In \underline{Section~\ref{s:2}} we collect and prove 
the basic regularity and differentiability properties of the reduced energy 
$\calI$ and prove the chain rule identity. Some of the arguments are taken  
from the earlier paper \cite{KRZ2} but are adapted to the enhanced smoothness 
assumptions on the boundary $\partial\Omega$. In \underline{Section~\ref{s:3}} we study a 
time-discrete version of the viscous damage system \eqref{dndia-eps}, derive 
the necessary a priori estimates and provide an energy-dissipation inequality 
for suitable interpolants of the time incremental solutions. The main part of 
Section~\ref{s:3} is devoted to proving that $A_q z_k\in L^2(\Omega)$ for time 
incremental  solutions $z_k$. 
In \underline{Section~\ref{s:4}} we shortly address the existence of viscous solutions to 
the system \eqref{dndia-eps}. The main focus of the paper lies on the 
analysis of the  vanishing-viscosity limit as both the viscosity parameter and 
the time step size tend to zero simultaneously (Sections~\ref{s:5} \& \ref{ss:5.3}). The notion 
of $\BV$ solutions is introduced and explained at length 
in \underline{Section~\ref{s:5}}, where also the main 
existence theorem is formulated and where further properties of $\BV$ solutions 
are discussed. The corresponding proofs are collected in \underline{Section~\ref{ss:5.3}}. A short \underline{Appendix} collects some elliptic regularity results that are key for our analysis. 
\medskip

\noindent 
We conclude by fixing some notation that will be used throughout the paper. 
\begin{notation}
\upshape Throughout the paper, for a given Banach space $X$,  we will  by
$\|\cdot \|_X$ denote its norm; in the case of product spaces $X\times \ldots \times X$,  we will often write $\|\cdot\|_X$ in place of $\| \cdot \|_{X\times \ldots \times X}$. We will denote by 
 $\pairing{}{X}{\cdot}{\cdot}$ the duality pairing between $X^*$ and $X$, using the symbol $(\cdot, \cdot)_X$ for the scalar product in $X$, if $X$ is a Hilbert space.
\par
 We will  denote most of the positive constants
occurring in the calculations, and depending on known quantities,  by the 
symbols  $c,\, c', \, C, \, C', \ldots$, 
whose meaning may vary even within the same line.
Furthermore, the symbols $I_i,$
 $i=0,1,\ldots,$
 will be used as  abbreviations  for  several integral terms 
appearing   in the various
 estimates: we warn the reader that we will not be self-consistent with the numbering, so that, for instance,
$I_1$ will appear several times with different meanings.
\end{notation}

\section{\bf Preliminaries  and properties of the reduced energy}  
\label{s:2}
\noindent
 We start by collecting our standing assumptions on the reference domain 
$\Omega$ and on the energy functional $\calE$ in Section~\ref{ss:2.1}. Combining 
these requirements, in Sec.~\ref{ss:2.2} we will obtain two regularity results 
for the Euler-Lagrange equation associated with the minimization of the elastic 
energy. In Sec.~\ref{ss:2.3}, such results will have a pivotal role in deriving 
a series of properties of the reduced energy $\calI$, at the core of our 
subsequent analysis. 
\subsection{ Setup}
\label{ss:2.1}
\noindent
Throughout the paper, we shall suppose that 
\begin{assumption}[Regularity of the  domain] 
\label{ass:domain}
{\sl  $\Omega \subset \R^\sd $  is a bounded $\mathrm{C}^{1,1}$-domain
with Dirichlet boundary $\Gamma_D = \partial\Omega$.} 
%\end{gathered}
%\end{align}
\end{assumption}
\noindent
From now on, we shall denote the state spaces for the variables $u$ and $z$ by
\[
\calU: =  H^1_{0}(\Omega;\R^\sd),  \qquad \calZ : = W^{1,q}(\Omega) \quad \text{with } q>\sd.
\]
 We will denote by 
 \[
\text{ $W^{-1,p}(\Omega)$ the dual space of 
$W^{1,p'}_0(\Omega)$ with $\frac1p + \frac{1}{p'}=1$.}
\]
For later use, we recall here two crucial properties of the elliptic operator 
$A_q$ 
holding for all $z_1,\, z_2,\, w \in \calZ$: 
\begin{align}
&
\label{e2.22}
 \langle A_q z_1 - A_q z_2, z_1-z_2\rangle_{\calZ}
\geq c_q\int_\Omega(1+\abs{\nabla z_1}^2 + \abs{\nabla
z_2}^2)^\frac{q-2}{2}\abs{\nabla (z_1-z_2)}^2\dx,
\\
&
\label{A7-D}
\abs{ \pairing{}{\calZ}{A_q z_1 - A_q z_2}{w}} \leq
 c_q' \int_\Omega(1+\abs{\nabla z_1}^2 + \abs{\nabla
z_2}^2)^{(q{-}2)/2} \abs{\nabla (z_1-z_2)}\abs{\nabla w}\,\dd x.
 \end{align}
 These inequalities rely on the corresponding estimates for the 
 function $G_q: \R^3 \to \R$ defined by $G_q(A): = \tfrac1q  
(1{+}|A|^2)^{q/2}$ and its gradient. In particular the following monotonicity 
estimate is valid 
\begin{align}
\label{eq.monGq}
 (\nabla G_q(A) {-} \nabla G_q(B)) \cdot (A{-}B) \geq c_q (1{+}|A|^2 
{+}|B|^2)^{(q-2)/2} |A{-}B|^2  \qquad \text{for all }A,\, B \in \R^3
 \end{align} 
 with $c_q>0$ the same constant as in  
\eqref{e2.22}, which is in fact a consequence of the estimates 
 provided in \cite[Lemma 8.3]{ana:Giu94}. 
 
\par
The 
 energy functional  $\calE: [0,T]\times \calU \times \calZ \to \R$  consists of two contributions. The first one,
$\calI_1$, 
 only depends on the damage variable. The second one,
 $\calE_2 = \calE_2(t,u,z)$,
is given by the sum of an elastic energy of the type
$\int_\Omega g(z)W(\varepsilon(x,u + u_D(t)))\dx
$ with  
 $u_D$ a Dirichlet datum,   
and of the external loading term.
\begin{assumption}[The energy functional]
\label{assumption:energy}
{\sl We consider 
\[
\calI_1: \calZ \to \R \ \text{ defined by } \
\calI_1(z):=\calI_\il (z)+ \int_\Omega f(z)\dx \  \text{ with }
\calI_\il (z):= \frac 1{\il} \int_\Omega
(1+\abs{\nabla z}^2)^{\frac{\il}{2}}\,\dd x, \  q>\sd,
%\end{equation}
\]
and $f$ fulfilling
\begin{equation}
\label{ass-eff} f\in \rmC^2(\R)  \quad  \text{ and } \quad  \exists\,
K_1,\, K_2 >0 \quad  \forall\, x \in \R : \quad 
\ f(x)\geq K_1\abs{x} - K_2.
\end{equation} 
As for $\calE_2$,
linearly elastic materials are considered with an elastic energy
density 
\[
W(x,\eta)= \frac{1}{2} \bbC(x) \eta : \eta \quad  \text{ for $\eta \in
\R^{\sd\times \sd}_\text{sym}$ and almost every $x\in \Omega$.}
\]
Hereafter, we shall suppose for the elasticity tensor 
that
\begin{subequations}
\label{elast-tensor}
\begin{align}
\label{elast-1} &
 \bbC \in
 \rmC_\mathrm{lip}^0
(\overline{\Omega};  \Lin(\R^{\sd\times \sd}_\text{sym},\R^{\sd\times
    \sd}_\text{sym}))
\text{ with } \bbC(x)\xi_1:\xi_2= \bbC(x)\xi_2:\xi_1 \text{ for
  all }x\in \Omega,\xi_i\in \R^{\sd\times
    \sd}_\text{sym}, 
\\ &
\label{elast-2} \exists\, \gamma_0>0 \quad  \text{for all }\xi\in
\R^{\sd\times \sd}_\text{sym} \text{ and almost all }x\in \Omega:  \ \
\bbC(x)\xi \colon \xi\geq \gamma_0\abs{\xi}^2.
\end{align}
\end{subequations}
 Let $g:\R\rightarrow \R$ be a further constitutive function 
 such that}
  \begin{equation}
 \label{new-g}
 g\in \mathrm{C}^{2}(\R)  \text{  with } g' , \, g{''} 
 \in L^\infty(\R), \ \text{ and } \exists\, \gamma_1, \, \gamma_2>0  \
\forall\, z \in \R\, : \
\gamma_1 \leq g(z) \leq \gamma_2.
 \end{equation}
\end{assumption}
\noindent
 Then, we
 take the
elastic energy
%\begin{equation}
%\label{elastic-energy}
\[
\begin{aligned}
\calE_2: [0,T]\times \calU \times \calZ \to \R  \ \text{ defined by
} \ \calE_2(t,u,z):=  \int_\Omega g(z)W(x,\varepsilon(u + u_D(t)))\dx
-\pairing{}{\calU}{\ell(t)}{u}
\end{aligned}
%\end{equation}
\]
 where
$\varepsilon(u)=\frac{1}{2}(\nabla u + \nabla u^T)$ is the
symmetrized strain tensor and $\ell\in \rmC^0([0,T],\calU^*)$ an external loading. 
 Further requirements on $\ell$ and $u_D$ will be specified in Assumption 
\ref{ass:load} ahead.  
For $u\in \calU$ and $z\in \calZ$ the
stored energy is then defined by
\begin{align}
\label{stored-energy}
\calE(t,u,z) : = \calI_1(z) +\calE_2(t,u,z).
\end{align}
%%%%%%
%%%%%%
%%%%%
\par
Minimizing the functional
{\color{black} $\calE$} with respect to the displacements we obtain
the \emph{reduced energy}
\begin{equation}
\begin{aligned}
\label{reduced-energy}  \calI:[0,T]\times\calZ\rightarrow \R \
\text{ given by } \ \calI(t,z):= \calI_1(z) + \calI_2(t,z)
\text{ with }  \calI_2(t,z):=\inf \{\calE_2(t,v,z)\, : \ v\in
\calU\}.
\end{aligned}
\end{equation}
%%%%	
\subsection{ Preliminary regularity results}
\label{ss:2.2}
 We focus on the regularity properties of the operator
$ L_{g(z)}:H^1_{0}(\Omega;\R^\sd)\rightarrow
W^{-1,2}(\Omega;\R^\sd)$  
associated with the 
following bilinear form describing linear
elasticity, i.e.,
\begin{align}
\langle  L_{g(z)} u ,v\rangle:= \int_\Omega g(z)\bbC \varepsilon(u):
\varepsilon(v)\dx \quad \text{for all } u,v\in H^1_{0}(\Omega;\R^\sd),
\end{align}
where $\bbC$ is  from \eqref{elast-tensor}, $g$ from \eqref{new-g}, and $z $ is 
a fixed element in  $\calZ = W^{1,q}(\Omega)$, with $q>\sd$. 
 Our first result extends 
\cite[Lemma 2.3]{KRZ2} to a wider range of exponents, cf.\ Remark 
\ref{rmk:better-regularity} below.  
\begin{lemma}%[Lemma 2.3, \cite{KRZ2}]
\label{l:reg-KRZ2}
 Under Assumption \ref{ass:domain}, let $\bbC$ and $g$ comply with 
 \eqref{elast-tensor} and  \eqref{new-g}, respectively. Then,
 there  holds
\begin{enumerate}
 \item[(a)] For every $p>1$ and $z\in W^{1,q}(\Omega)$ the  operator $L_{g(z)}: 
W^{1,p}_0(\Omega)\to W^{-1,p}(\Omega) $ is a 
 topological isomorphism.
 \item[(b)] Uniform estimate: For every  $p_*> 2$ there exists
  a constant $c_{q,p_*}>0$ 
such  that for all $z\in W^{1,q}(\Omega)$ and $p\in [p_*',p_*]$ it holds
 \begin{align}
\label{crucial-reg-estimate}
 \norm{L_{g(z)}^{-1}}_{W^{-1,p}(\Omega;\R^\sd) \rightarrow 
W^{1,p}_{0}(\Omega;\R^\sd)}  \leq c_{q,p_*}
(1 + \norm{\nabla
z}_{L^q(\Omega)})^{\hat{k}_*  \frac{p_*\abs{p-2}}{p(p_*-2)}},
\end{align}
where  $\hat{k}_*\in \N$
is the smallest integer with $\hat{k}_*>\frac{\sd q}{2(q-\sd)}$. 
\end{enumerate}
\end{lemma}
\begin{proof}
 The first statement is a consequence of \cite[Theorem 3]{Val78}, see  also 
\cite[Theorem 7.1]{MR03_polyhedral}. The uniform 
 estimate follows along the same lines as in  the proof of \cite[Lemma 2.3]{KRZ2}, relying on a 
 recursion argument developed in \cite{BabMil12}.
\end{proof}
 \begin{remark}
\label{rmk:better-regularity}
\upshape
 Lemma \ref{l:reg-KRZ2}
enhances \cite[Lemma 2.3]{KRZ2} thanks to the stronger regularity  condition  on 
the reference domain $\Omega$, which in \cite{KRZ2} was  
only required to fulfill these properties: 
\begin{itemize}
 \item[(i)] The spaces $W^{1,p}_{\Gamma_D}(\Omega;\R^d) = \{ u \in W^{1,p}(\Omega;\R^d)\, : \ u|_{\Gamma_D} =0 \}$, $p\in (1,\infty)$ (and 
 $\Gamma_D$  with positive Hausdorff measure, but possibly different from $\partial\Omega$, was allowed in \cite{KRZ2}), form an interpolation scale.
\item[(ii)]
There exists $p_*>3$ such that for all $p\in [2,p_*]$ the
operator $L: W^{1,p}_{\Gamma_D}(\Omega;\R^d)\rightarrow
W^{-1,p}_{\Gamma_D}(\Omega; \R^d)$ is an isomorphism.
\end{itemize}
It was for such $p_*>3$, in fact, that 
the  isomorphism property (a)   and the uniform estimate 
\eqref{crucial-reg-estimate} were obtained in \cite[Lemma 2.3]{KRZ2}. Let us highlight that, instead,
in Lemma~\ref{l:reg-KRZ2}
 property (a) is guaranteed for all $p>1$, and \eqref{crucial-reg-estimate} is shown for \emph{every} $p_*>2$. 
 \end{remark} 
%\eqref{crucial-reg-estimate}  
\par
The most relevant consequence of
 Assumption~\ref{ass:domain} for our analysis, though, is given by this second, enhanced, elliptic regularity result,  
 which is to be compared with \cite[Lemma A.1]{BabMil12}, holding for homogeneous Neumann boundary conditions. 

 \begin{lemma}
\label{l:b-m}
 Under Assumption \ref{ass:domain}, let $\bbC$ and $g$ comply with 
 \eqref{elast-tensor} and \eqref{new-g}, respectively. Then,
 for all 
$z\in W^{1,q}(\Omega)$ the operator $ L_{g(z)} :  \calU  
\to  \calU^*  $ 
fulfills
\[
 L_{g(z)}^{-1}(h) \in H^2(\Omega;\R^\sd) \quad \text{for all } h \in 
 L^2(\Omega;\R^\sd)
\]
and there exists $c_0>0$ such that for all $z \in W^{1,q}(\Omega)$ and all  $ h \in L^2(\Omega;\R^\sd)$
 \begin{align}
 \label{est_H2norm}
\norm{u}_{H^2(\Omega)}\leq c_0(1 + \norm{\nabla
  z}_{L^q(\Omega)})^\alpha(\norm{h}_{L^2(\Omega)} + \norm{u}_{H^1(\Omega)}), 
\end{align}
where $u=L_{g(z)}^{-1}(h)$ and  $\alpha\geq 2$ is the smallest integer bigger than or equal to
$q/(q-\sd)$.  
\end{lemma}
\begin{proof}
 The proof of \cite[Lemma A.1]{BabMil12} can be directly transferred to our situation having in mind that for every $p\in (1,\infty)$ 
 the operator 
 \[
  L_\bbC=L_{g(1)}:W_0^{1,p}(\Omega)\cap W^{2,p}(\Omega)\to L^p(\Omega),\quad u\mapsto -\Div\bbC\varepsilon(u)
 \]
is a continuous isomorphism, cf.\ Theorem \ref{app.reg.thm2}.
\end{proof}
\begin{remark}
\label{rmk:exponents}
 Observe that $
\sup_{p \in  [p_*',p_*]}   \frac{p_*\abs{p-2}}{p(p_*-2)} \leq 1,
$
hence  we can estimate from above
the right-hand side of 	\eqref{crucial-reg-estimate} by $(1 + \norm{\nabla
z}_{L^q(\Omega)})^{\hat{k}_*}$.  That is why, in what follows, whenever 
applying estimates 
\eqref{crucial-reg-estimate} and \eqref{est_H2norm}, 
possibly with two different elements $z_1,\,z_2 \in \calZ$,  we will simply use the quantity
\begin{equation}
\label{place-holder-P}
P(z_1,z_2): = (1 + \norm{\nabla
z_1}_{L^q(\Omega)}+ \norm{\nabla
z_2}_{L^q(\Omega)})^{k_*}\,,
\end{equation} 
where $k_*:=\max\{\hat k_*,\alpha\}+1$  
with $\hat{k}_*$ from Lemma~\ref{l:reg-KRZ2} and $\alpha$ from 
\eqref{est_H2norm}.   
 With this, \eqref{est_H2norm} can be rewritten in terms of the quantity 
 $P$  as
 \[
 \norm{u}_{H^2(\Omega)}\leq c_0 P(z,0)(\norm{h}_{L^2(\Omega)} + \norm{u}_{H^1(\Omega)}).
 \]
 \end{remark}  
 
In the sequel we will frequently use the following regularity result 
from \cite[Theorem 2 \& Remark 3.5]{Savare98} for solutions of the 
$q$-Laplace equation:
 \begin{proposition}
  \label{prop:Savare98}
  For every $q>2$ there exists a constant $C_q>0$ such that for all $f\in 
L^{q'}(\Omega)$ it holds: 
 If $z\in W^{1,q}(\Omega)$ satisfies $\langle A_q z,\wt z\rangle =\langle 
f,\wt z\rangle$ for all $\wt z\in W^{1,q}(\Omega)$, then for all $\sigma\in 
(0,\frac{1}{q})$ the function $z$ belongs to $W^{1+\sigma,q}(\Omega)$ and 
\begin{equation}
\label{est-Sav}
\norm{z}_{W^{1+\sigma,q}(\Omega)}\leq C_q (\norm{f}_{L^{q'}(\Omega)} + 
\norm{z}_{L^q(\Omega)}).  
\end{equation} 
 \end{proposition}
Note that on the right-hand side of \eqref{est-Sav} the $L^q$-norm of $z$ appears since $A_q$ is not 
bijective on $W^{1,q}(\Omega)$. 
%%%%
%%%%
%%%%
\subsection{Properties of the reduced energy}
\label{ss:2.3}
Relying on Lemmas \ref{l:reg-KRZ2} and \ref{l:b-m}, we will show that the reduced energy functional $\calI$ enjoys a series of differentiability properties, which in fact improve those obtained in \cite[Sec.\ 2.3]{KRZ2}, under the additional 
\begin{assumption}[The external loadings]
\label{ass:load}
{\sl 
From now on, we will suppose that $\ell $ and $u_D$ comply with the following requirements}
\begin{equation}
\label{more-regular-data}
\begin{aligned}
&
\ell \in L^\infty (0,T; L^2(\Omega;\R^\sd)) \cap \mathrm{C}^{1,1} ([0,T];  
W^{-1,3}(\Omega;\R^\sd) ), 
\\
&
 u_D \in L^\infty (0,T; H^2(\Omega;\R^\sd)) \cap \mathrm{C}^1 ([0,T];W^{1,3}(\Omega;\R^\sd)).
 \end{aligned}
\end{equation}
\end{assumption}
\par
The starting point is the following result, which improves \cite[Lemmas 2.6, 2.7]{KRZ2}. 
\begin{lemma}[Existence of minimizers for $\calE(t,\cdot, z)$ \&  their continuous dependence on the data]
\label{prop:enhnaced-reg}
$\\$
Under Assumptions \ref{ass:domain}, \ref{assumption:energy}, and  \ref{ass:load}, for every $(t,z) \in [0,T]\times \calZ$ 
there exists a unique minimizer $\umin(t,z)  \in \calU $ for  the 
stored energy $\calE(t,\cdot,z )$ \eqref{stored-energy}. In fact, $\umin(t,z) 
\in H^2(\Omega;\R^\sd)$. Moreover, there exist positive constants
 ${c}_1$ and ${c}_2$ such
that for all $(t,z),\, (t_1,z_1),\, (t_2,z_2)  \in [0,T]\times \calZ$
 and for all $p_*>2$ 
\begin{align}
& \label{H2_improved-1} \norm{\umin(t,z)}_{H^{2}(\Omega)} \leq {c}_1 P(z,0)
\left(\norm{\ell(t)}_{L^2(\Omega)} +
\norm{u_D(t)}_{H^{2}(\Omega)} \right);
\\
&
\label{summary}
\begin{aligned}
&
\norm{\umin(t_1,z_1) {-} \umin(t_2,z_2)
}_{W^{1,p}(\Omega)}
\\ & 
\leq c_2 P(z_1,z_2)^2 \left(|t_1{-}t_2|+ \norm{z_1{-}z_2}_{L^{6p/(6{-}p)} (\Omega)} \right)
\left(\norm{\ell}_{\mathrm{C}^1 ([0,T];   W^{-1,p}(\Omega) )}
+ \norm{u_D(t)}_{\mathrm{C}^1 ([0,T];W^{1,p}(\Omega))}
\right) 
\end{aligned}
\end{align}
for all $ p \in [p_*', \min\{ p_*,3\}]$, 
with  $P(\cdot,\cdot)$ defined by \eqref{place-holder-P}.
In particular, there holds
\begin{align}
 \label{H2_improved-2}
\begin{aligned}
&\norm{\umin(t_1,z_1) {-} \umin(t_2,z_2)
}_{W^{1,3}(\Omega)}
\\ &
\leq {c}_2  P(z_1,z_2)^2 \left(|t_1{-}t_2|+ \norm{z_1{-}z_2}_{L^6 (\Omega)} \right)
\left(\norm{\ell}_{\mathrm{C}^1 ([0,T];   W^{-1,3}(\Omega))}
+ \norm{u_D(t)}_{\mathrm{C}^1 ([0,T];W^{1,3}(\Omega))}. 
\right),
\end{aligned}
\end{align}
Finally, the reduced energy $\calI$ from \eqref{reduced-energy} is bounded  from 
below and in particular satisfies the following coercivity estimate: 
 \begin{align}
\label{est_coerc1} 
 \exists\, c_3,\, c_4>0 \quad \forall\, (t,z) \in [0,T]\times \calZ\, : \qquad 
\calI(t,z)&\geq c_3\big(
 \norm{\nabla z}^\il_{L^\il(\Omega)}
 + \norm{z}_{L^1(\Omega)}
 + \norm{\umin(t,z)}_{H^1(\Omega;\R^\sd)}^2
 \big) -c_4.
\end{align}  
\end{lemma}
\begin{proof} 
We refer to \cite[Lemma 2.1]{krz} for the proof of the existence and 
uniqueness of $\umin(t,z)$, as well as for estimate \eqref{est_coerc1}. 
 Clearly, $\umin(t,z)$ satisfies 
$L_{g(z)}\umin(t,z) = - L_{g(z)} u_D(t) - \ell(t)$.  Observe that $L_{g(z)} 
u_D(t) \in L^2(\Omega)$. Indeed, by the assumptions on $g$, $\bbC$  
and since $u_D(t)\in H^2(\Omega)$, we  have $g(z)\Div (\bbC 
\varepsilon(u_D(t))\in L^2(\Omega)$. On the other hand, 
$\bbC\varepsilon(u_D(t)) \nabla_x g(z)= g'(z)\bbC\varepsilon(u_D(t))\nabla z \in 
L^2(\Omega)$,  which follows by H\"older's inequality taking into 
account that    $H^1(\Omega)\subset L^6(\Omega)$ and that $q>3$. Moreover, it 
holds 
$\norm{L_{g(z)} u_D(t)}_{L^2(\Omega)} \leq c (1 + \norm{\nabla 
z}_{L^q(\Omega)})\norm{u_D(t)}_{H^2(\Omega)}$. 
Hence, it follows from  \eqref{est_H2norm}, cf.\ also Remark 
\ref{rmk:exponents}, and \eqref{crucial-reg-estimate} with $p=2$ 
that 
\begin{align*}
 \norm{\umin(t,z) }_{H^{2}(\Omega)} &\leq {c}_0 (1+\norm{\nabla 
z}_{L^q(\Omega)})^\alpha ( \| \ell(t) \|_{L^2(\Omega)} + 
\norm{\Div(g(z)\bbC \varepsilon( u_D(t))}_{L^2(\Omega)} 
+\norm{\umin(t,z)}_{H^1(\Omega)})\\
&\leq c(1+\norm{\nabla z}_{L^q(\Omega)})^\alpha
\big(
\norm{\ell(t)}_{L^2(\Omega)} + (1+ \norm{\nabla 
z}_{L^q(\Omega)})\norm{u_D(t)}_{H^2(\Omega)}
\big)\\
&\leq c_1 P(z,0)\big( \norm{\ell(t)}_{L^2(\Omega)} 
+\norm{u_D(t)}_{H^2(\Omega)}\big).
\end{align*}
 All in all, we conclude \eqref{H2_improved-1}. 
 % in view of \eqref{est_coerc1}.
\par  Finally, in order to show \eqref{summary} we mimic the argument from 
the proofs of  \cite[Lemma 2.2]{krz} \& \cite[Lemma 2.7]{KRZ2}.  
Namely,  for $i=1,2$, let $u_i:=\umin(t_i,z_i)\in H^2(\Omega;\R^\sd)$. From the 
corresponding Euler-Lagrange
equations 
  we obtain that $u_1-u_2$ satisfies for all 
$v\in\calU$
  \begin{equation}
\label{instrumental}
\begin{aligned}
  &   \int_\Omega g(z_1)\bbC \varepsilon(u_1-u_2){\colon}
\varepsilon(v)\dx
  =\int_\Omega
\big(g(z_2)-g(z_1)\big)\bbC\varepsilon(u_2){\colon}\varepsilon(v)\dx
\\
& \quad  \quad \qquad \qquad -\int_\Omega\big(g(z_1)\bbC\varepsilon(u_D(t_1)) -
g(z_2)\bbC\varepsilon(u_D(t_2))\big){\colon}\varepsilon(v)\dx +
\int_\Omega \left(\ell(t_1){-}\ell(t_2) \right){v} \dx.
\end{aligned}
\end{equation}
Taking into account that, for $i,j\in \{1,2\}$ $g(z_i) \varepsilon (u_j) \in L^6(\Omega;\R^{\sd\times \sd})$ in view of \eqref{new-g} and of the fact that $u_j \in H^2(\Omega;\R^\sd)$, giving $\varepsilon (u_j) \in L^6(\Omega;\R^{\sd\times \sd})$, and exploiting condition \eqref{more-regular-data} on $\ell$ and $u_D$, via a density argument we see that \eqref{instrumental} 
  extends to test functions $v \in W^{1,6/5}_{0}(\Omega;\R^\sd)$ .
   Hence, $u_1-u_2$ fulfills for all 
 $v\in W^{1,6/5}_{0}(\Omega;\R^\sd)$ the relation 
\[
\int_\Omega g(z_1)\bbC \varepsilon(u_1-u_2):\varepsilon(v)\dx=
\pairing{}{W^{1,6/5}_{0}(\Omega;\R^\sd)}{\tilde{\ell}_{1,2}}{v},
\]
where $\tilde{\ell}_{1,2} \in  W^{-1,6}(\Omega;\R^\sd)$
subsumes the terms on the right-hand side of \eqref{instrumental}. We now
fix an arbitrary $p_*>2$ and apply estimate 
\eqref{crucial-reg-estimate}  with $p \in [p_*', \min\{ p_*,3\}]$ 
% and $p_*>d=3$ from Lemma \ref{l:reg-KRZ2} 
(indeed, the restriction $p\leq 3$ is in view of conditions \eqref{more-regular-data} on $\ell$ and $u_D$).
We thus obtain  $
\norm{u_1-u_2}_{W^{1,p}(\Omega;\R^\sd)} \leq c_{q,p_*}
P(z_1,0)\norm{\tilde{\ell}_{1,2}}_{ W^{-1,p}(\Omega;\R^\sd)}$. 
Hence,
\begin{multline}
\label{est-inter} \norm{u_1-u_2}_{W^{1,p}(\Omega;\R^\sd)} \leq
 c_{p_*,q}
P(z_1,0)
\big( \norm{\ell(t_1)-\ell(t_2)}_{ W^{-1,p}(\Omega;\R^\sd)}
+
  \norm{(g(z_1)-g(z_2))\bbC \varepsilon(u_2)}_{L^{p}(\Omega;\R^\sd)} \\
+  \norm{g(z_1)\bbC \varepsilon(u_D(t_1)) -
  g(z_2)\bbC \varepsilon(u_D(t_2))}_{L^{p}(\Omega;\R^\sd)}  \big).
\end{multline}
 Now, the Lipschitz continuity of $g$ (with Lipschitz constant $C_g$) and H\"older's inequality
imply that
\begin{equation}
\label{auxiliary}
\begin{aligned}
 \norm{(g(z_1)-g(z_2))\bbC \varepsilon(u_2)}_{L^{p}(\Omega;\R^\sd)}   & \leq
C_g  \norm{z_1-z_2}_{L^{6p/(6{-}p)}(\Omega)} \norm{\varepsilon
(u_2)}_{L^{6}(\Omega;\R^{\sd})}\\ & 
\leq   C P(z_2,0)
\left(\norm{\ell(t)}_{L^2(\Omega)} +
\norm{u_D(t)}_{H^{2}(\Omega)} \right)  
 \norm{z_1-z_2}_{L^{6p/(6{-}p)}(\Omega)}
\end{aligned}
\end{equation}
where the second estimate follows from  \eqref{H2_improved-1} and from the fact that  $\norm{\varepsilon
(u_2)}_{L^{6}(\Omega;\R^\sd)} \leq C \| u_2 \|_{H^2(\Omega;\R^\sd)}$ by Sobolev 
embeddings. 
Moreover,
\[
\begin{aligned}
 & \norm{g(z_1)\bbC \varepsilon(u_D(t_1)) -
  g(z_2)\bbC \varepsilon(u_D(t_2))}_{L^{p}(\Omega)}  \\ & \leq  \norm{g(z_1) ( \bbC \varepsilon(u_D(t_1)) -\bbC \varepsilon(u_D(t_2)))}_{L^{p}(\Omega)} 
  +\norm{
  (g(z_1) - g(z_2))\bbC \varepsilon(u_D(t_2))}_{L^{p}(\Omega)} 
  \\ & \leq \gamma_2 |t_1-t_2|  \norm{u_D(t)}_{\mathrm{C}^1 ([0,T];W^{1,p}(\Omega))}   + C \| u_D \|_{L^\infty (0,T; H^2(\Omega))} \|z_1-z_2\|_{L^{6p/(6{-}p)}(\Omega)},
  \end{aligned}
\]
where the last estimate follows from the fact that $ \| g(z_1) \|_{L^\infty(\Omega)} \leq \gamma_2$ by \eqref{new-g}, 
 as well as the the fact that, for $p \leq 6$, $\norm{\varepsilon(u_D(t_2))}_{L^{p}(\Omega)}  \leq   C \| u_D \|_{L^\infty (0,T; H^2(\Omega))}.$ 
All in all, we 
 conclude \eqref{summary}, whence  \eqref{H2_improved-2} observing that, for 
$p=3$ one has $\tfrac{6p}{6-p}=6$.  
\end{proof} 
%\RRR
\par
Concerning the differentiability in time, we have the following analogue of 
\cite[Lemma 2.9]{KRZ2}, \cite[Lemma 2.3]{krz}, 
\begin{lemma}
\label{l:2.7}
Under Assumptions \ref{ass:domain}, \ref{assumption:energy}, and \ref{ass:load}, 
  for every $z\in \calZ$ the map $t\mapsto
\calI(t,z)$ is in $\rmC^{1}([0,T];\R)$ with
\begin{align}
\label{form-derivative} \partial_t\calI(t,z)= \int_\Omega g(z)\bbC
\varepsilon(\umin(t,z) + u_D(t)){\colon}\varepsilon(\dot u_D(t)) \dx -
\langle \dot\ell(t),\umin(t,z)\rangle_{H^1_{0}(\Omega;\R^\sd)}.
\end{align}
Moreover, there exists a constant $c_5>0$ such that for all $t\in
[0,T]$, $z\in \calZ$ 
we have
\begin{align}
\label{stim3} \abs{\partial_t\calI(t,z)}&\leq c_5
\big(\norm{u_D}^2_{\rmC^1([0,T];H{1}(\Omega;\R^\sd))} +
\norm{\ell}^2_{\rmC^1([0,T];W^{-1,2}(\Omega;\R^\sd))} \big).
\end{align}
Finally, 
 there exists a constant $c_6>0$ depending on
$\norm{\ell}_{\rmC^{1,1}([0,T];W^{-1,3}(\Omega;\R^\sd))}$ and
$\norm{u_D}_{\rmC^{1,1}([0,T];W^{1,3}(\Omega))}$ such that
for all $t_i\in
  [0,T]$ and $z_i\in \calZ$   we have
\begin{align}
\label{stim5} \abs{\partial_t\calI(t_1,z_1) -
\partial_t\calI(t_2,z_2)} \leq
c_6 P(z_1,z_2)^2
\big(\abs{t_1-t_2} +
\norm{z_1-z_2}_{L^{2}(\Omega)}\big).
\end{align}
\end{lemma}
\noindent
Let us stress that  the quantity on the right-hand side of estimate 
\eqref{stim3}, whose proof  is developed in  \cite[Lemma 2.3]{krz}, is 
\emph{independent} of $z \in \calZ$.
 \begin{proof}
We will only develop the proof of \eqref{stim5}, 
referring to the proof of \cite[Lemma 2.3]{krz} for 
\eqref{form-derivative} and \eqref{stim3}. 
We have
\[
\begin{aligned}
&\partial_t\calI(t_1,z_1) - \partial_t\calI(t_2,z_2)
\\ &
= \int_{\Omega} (g(z_1){-}g(z_2)) \mathbb{C} (\varepsilon
(\umin (t_1,z_1) + u_{D}(t_1))) {\colon} \varepsilon (\dot{u}_D(t_1))
\dx
\\ &  \quad +\int_{\Omega} g(z_2) \mathbb{C} (\varepsilon (\umin (t_1,z_1) +
u_{D}(t_1)){-} \varepsilon (\umin (t_2,z_2) + u_{D}(t_2))) {\colon}
\varepsilon (\dot{u}_D(t_1)) \dx
\\ & \quad +
\int_{\Omega} g(z_2) \mathbb{C} (\varepsilon (\umin (t_2,z_2) +
u_{D}(t_2))) {\colon} (\varepsilon (\dot{u}_D(t_1)){-}\varepsilon
(\dot{u}_D(t_2))) \dx
\\ & \quad - \pairing{}{}{\dot{\ell}(t_1) {-}\dot{\ell}(t_2)}{\umin
(t_1,z_1)} +\pairing{}{}{\dot{\ell}(t_2)}{\umin
(t_2,z_2){-}\umin (t_1,z_1)} \doteq I_1+I_2+I_3+I_4+I_5.
\end{aligned}
\]
To estimate $I_1$, and $I_3$ we rely on the fact that $g,\,
g' \in L^\infty (\R)$, and on \eqref{H2_improved-1}.
To estimate $I_2$ we additionally use
 the boundedness of $g$ and 
 H\"older's inequality as
follows
\[
\begin{aligned}
 I_2  & \leq c \norm{\varepsilon (\umin (t_1,z_1) +
u_{D}(t_1)){-} \varepsilon (\umin (t_2,z_2) + u_{D}(t_2))}_{L^{3/2}(\Omega)}
\norm{\varepsilon (\dot{u}_D(t_1))}_{L^3(\Omega)}
\\ &
\leq 
c P(z_1,z_2)^2 \left(|t_1{-}t_2|+ \norm{z_1{-}z_2}_{L^2 (\Omega)} \right)
\left(\norm{\ell}_{\mathrm{C}^1 ([0,T];  W^{-1,3/2}(\Omega))}
+ \norm{u_D(t)}_{\mathrm{C}^1 ([0,T];W^{1,3/2}(\Omega))},
\right)
\end{aligned}
\]
where the second estimate ensues from  \eqref{summary} with $p=3/2$   (which yields  $6p/(6-p)=2$), and from \eqref{more-regular-data}. 
%indeed, $p=3/2$ yields in \eqref{summary}. 
By \eqref{more-regular-data} and \eqref{summary} we also estimate $I_4$ and $I_5$.
 \end{proof}

We now discuss the differentiability of $\calI$ with respect to $z$; 
we shall denote by $\rmD_z\calI(t,\cdot): \calZ\to \calZ^*$ the
G\^ateaux-differential of the functional $\calI(t,\cdot)$. For the proof of the 
following result, we refer to    \cite[Lemma 2.10]{KRZ2},  \cite[Lemma 
2.4]{krz}. 
\begin{lemma}
\label{l:gateaux}
Under Assumptions \ref{ass:domain}, \ref{assumption:energy}, and \ref{ass:load}, 
for all $t\in [0,T]$ the functional
$\calI(t,\cdot):\calZ \rightarrow\R$ is G\^ateaux-differentiable at all
$z\in \calZ$, and for all $\eta\in \calZ$ we have
\begin{align}
\label{form-gateau-derivative} \pairing{}{\calZ}{\rmD_z\calI(t,z)}{\eta}=
\pairing{}{\calZ}{ A_q z}{\eta} +
\int_\Omega f'(z)\eta\dx + \int_\Omega g'(z)\wt W(t,\nabla
\umin(t,z)) \eta\dx,
\end{align}
where we use the abbreviation $\wt W(t,\nabla v)= W(x,\eps(v +
\nabla u_D(t))) =\frac{1}{2} \bbC\varepsilon(v +
u_D(t)){ : }\varepsilon(v + u_D(t))$. In particular, the following
estimate holds with a constant $c_7$ that depends on the data $\ell, u_D$, but
is independent of $t$ and $z$:
\begin{equation}
\label{esti-gateau}  \forall\, (t,z) \in
[0,T]\times \calZ\, :  \ \ \norm{\rmD_z\calI(t,z)}_{\calZ^*} \leq
  c_7 \left( \|z\|_{\calZ}^{q-1} +
\norm{f'(z)}_{L^\infty(\Omega)}
+ 1 \right).
\end{equation}
\end{lemma}
\par
Hereafter, we will use the short-hand notation  
\begin{equation}
\label{itilde}
\wt \calI(t,z):= \calI_2(t,z) + \int_\Omega f(z)\dx \quad \text{for
all $(t,z) \in [0,T]\times \calZ$}
\end{equation}
with $\calI_2$ from \eqref{reduced-energy} as the part of the reduced energy
collecting all lower order terms.
Accordingly, $\rmD_z \calI$ from \eqref{form-gateau-derivative} decomposes as
\begin{equation}
\label{der-decomp}
\rmD_z \calI(t,z) = A_q z + \rmD_z\wt{\calI}(t,z) \quad \text{for all } (t,z) \in [0,T]\times \calZ.
\end{equation}
In view of  \eqref{form-gateau-derivative}, and taking into account the $H^2(\Omega;\R^\sd)$-regularity of $\umin$ from Lemma \ref{l:b-m}, 
 the term
$\rmD_z\wt\calI(t,z)$ can be identified with an element of  $L^2(\Omega)$. In Lemma~\ref{l:diff_contz}
below we will even show
 that the map
 % $(t,z) \mapsto \wt{\calI}(t,z(t))$
 $(t,z) \mapsto \rmD_z\wt{\calI}(t,z)$ is Lipschitz continuous
w.r.t.\
a suitable \emph{Lebesgue} norm.
Therefore,  with the symbol
$\rmD_z\wt\calI$  we shall denote both
the derivative of $\wt\calI$ as an operator, and the corresponding
density in $L^2(\Omega)$. Accordingly,
we shall write
\begin{equation}
\label{abuse-integral}
\text{ for  a given } v \in L^{2}(\Omega) \quad
\int_\Omega  \rmD_z\wt\calI(t,z)
v  \,\dd x \quad \text{ in place of } \quad
 \langle  \rmD_z\wt\calI(t,z),
v\rangle_{L^{2}(\Omega)}.
\end{equation}
\par
 For $h\in \rmC^0(\R)$ and $z_1,z_2\in\calZ$ let
\begin{equation}
\label{C_h}
C_{h}(z_1,z_2)= \max \Set{ |h(s)|}{ \ |s| \leq
 \|z_1\|_{L^\infty(\Omega)} + \|z_2 \|_{L^\infty(\Omega)} }.
\end{equation}
This notation will be used along the proof of the following lemma.
\begin{lemma}%[Local Lipschitz continuity of $\wt \calI$ and $\rmD_z\wt{\calI}$]
\label{l:diff_contz}
Under Assumptions \ref{ass:domain}, \ref{assumption:energy}, and \ref{ass:load}, 
 there exists a constant $c_8>0$
that depends on the norms 
$\norm{\ell}_{\rmC^{1,1}([0,T];  W^{-1,3}(\Omega;\R^\sd) )}$ and
$\norm{u_D}_{\rmC^{1}([0,T];W^{1,3}(\Omega;\R^\sd))}$
such that for all $t_i\in [0,T]$ and all $z_i\in \calZ$ it holds
\begin{equation}
\begin{aligned}
\label{Lip-cont-I}
\left| \wt{\calI} (t_1, z_1) - \wt{\calI} (t_2, z_2) \right|\leq c_8 ( 
1 + C_{f'}(z_1,z_2)+
 P(z_1,z_2)^{3}) \left( \abs{t_1-t_2}+ \norm{z_1-z_2}_{L^{3}(\Omega)} 
  \right),
\end{aligned}
\end{equation}
with $C_{f'}(z_1,z_2)$  as in \eqref{C_h}, corresponding to  $ h=f'$. 
Further, %for every $\mu\in [1,p_*/2)$,
\begin{align}
&
\label{enhanced-stim-7}
\begin{aligned}
&
 \norm{\rmD_z \wt \calI(t_1,z_1) - \rmD_z \wt \calI(t_2,z_2) }_{L^2(\Omega)}
 \\
&\leq c_8
\big( 1  + C_{f'}(z_1,z_2) 
+P(z_1,z_2)^3 \big)
\big( \abs{t_1-t_2}+\norm{z_1 -z_2}_{L^6(\Omega)}    \big),
\end{aligned}
\\
& 
\label{stim-l4}
\begin{aligned}
&
 \norm{\rmD_z \wt \calI(t_1,z_1) - \rmD_z \wt \calI(t_2,z_2) }_{L^{4/3}(\Omega)}
 \\
&\leq c_8
\big( 1  + C_{f'}(z_1,z_2)  
+P(z_1,z_2)^3 \big)
\big( \abs{t_1-t_2}+\norm{z_1 -z_2}_{L^4(\Omega)}    \big),
\end{aligned}
\end{align}
and % for $\mu\in  [1,p_*/2]$,
\begin{equation}
\label{estimate-for-DI}
	 \| \rmD_z \wt
\calI(t,z)  \|_{L^{2}(\Omega)}  \leq
c_8(1 + \norm{f'(z)}_{L^\infty(\Omega)} + P(z,0)^2) 
\qquad \text{for all } (t,z) \in [0,T] \times \calZ. 
\end{equation}
\end{lemma}
\begin{proof}
 Although the proof follows the same lines as that  of \cite[Lemma 
2.12]{KRZ2}, 
let us briefly see how 
 the improved estimates \eqref{H2_improved-1} and \eqref{H2_improved-2} lead to 
\eqref{Lip-cont-I},  \eqref{estimate-for-DI}, and \eqref{stim-l4},
 while we will omit the calculations for 
%\eqref{enhanced-stim-7}. 
 \eqref{estimate-for-DI}.  
 As for \eqref{Lip-cont-I}, we observe that 
\[
\begin{aligned}
\left| \wt\calI(t_1,z_1) - \wt\calI(t_2,z_2) \right| &  \leq
\int_\Omega |f(z_1)-f(z_2)| \,\dd x + \int_\Omega |g(z_1) - g(z_2)| |\wt
W(t_1,\nabla u_1)| \,\dd x
\\ & \quad + \int_\Omega |g(z_2)|  |\wt W(t_1,\nabla u_1)-\wt W(t_2,\nabla u_2)
| \,\dd x
 + |\pairing{}{\calU}{\ell(t_1) - \ell(t_2)}{u_1}|
 \\
 & \quad
+  |\pairing{}{\calU}{\ell(t_2)}{u_1-u_2}| \doteq I_1 +I_2 +I_3+I_4+I_5,
\end{aligned}
\]
where $u_i:=u_\text{min}(t_i,z_i)\in H^{2}(\Omega;\R^\sd)$  and, as above,
 $\wt W(t_i,\nabla  u_i)=\frac{1}{2} \bbC\varepsilon(u_i +
u_D(t_i)){ : }\varepsilon(u_i + u_D(t_i))$ for $i=1,2$. 
We observe that  (cf.\  notation \eqref{C_h})
\[
\begin{aligned}
&
I_1 \leq C_{f'}(z_1,z_2) \|z_1-z_2\|_{L^1(\Omega)},
\\
&
\begin{aligned}
I_2 &  \leq   %C_{g'} (z_1,z_2)
C  \|z_1-z_2\|_{L^2(\Omega)}    \| \eps(u_1+u_D(t_1))\|_{L^3(\Omega)}
 \| \eps(u_1+u_D(t_1))\|_{L^6(\Omega)}  \\ & \leq C'  P(z_1,0)^2  
 % C_{g'} (z_1,z_2) 
   \|z_1-z_2\|_{L^2(\Omega)},
 \end{aligned}
\\
& 
\begin{aligned}
I_3  & \leq C  \| \eps(u_1+u_D(t_1)) +  \eps(u_2+u_D(t_2)) \|_{L^2(\Omega)}  \| \eps(u_1+u_D(t_1)) -  \eps(u_2+u_D(t_2)) \|_{L^2(\Omega)}
\\ &   \leq C  P(z_1,z_2)  P(z_1,z_2)^2 (|t_1-t_2| +  \|z_1-z_2\|_{L^3(\Omega)}),
\end{aligned}
  \\
  & 
  I_4 \leq C |t_1-t_2| \|u_1\|_{H^1(\Omega)} \leq C' |t_1-t_2|,
  \\
  & 
  I_5 \leq C \|u_1-u_2\|_{H^1(\Omega)} \leq C P(z_1,z_2)^2 (|t_1-t_2| +  \|z_1-z_2\|_{L^3(\Omega)})\,,
\end{aligned}
\]
where, in the estimate for $I_2$ we have exploited \eqref{H2_improved-1}, while in the estimates for $I_3$ and $I_5$ we have also resorted to \eqref{summary} with $p=2$. All in all, we conclude  \eqref{Lip-cont-I}.  The estimate for $I_4$ follows from \eqref{more-regular-data}. 
\par
As for %\eqref{estimate-for-DI}, 
 \eqref{enhanced-stim-7}, 
we have that 
\[
\begin{aligned}
  \| \rmD_z \wt
\calI(t_1,z_1) -  \rmD_z \wt
\calI(t_2,z_2)  \|_{L^{2}(\Omega)}   &  \leq \| 
f'(z_1)-f'(z_2)\|_{L^2(\Omega)} + \| (g'(z_1){-}g'(z_2))  \wt W(t_1,\nabla u_1) 
\|_{L^2(\Omega)} 
\\ & \quad +  \| g'(z_2) (\wt W(t_1,\nabla u_1){-} \wt W(t_2,\nabla u_2)) \|_{L^2(\Omega)} \doteq I_6+I_7+I_8\,.
\end{aligned}
\]
We observe that $I_6 \leq C_{f'}(z_1,z_2) \|z_1-z_2\|_{L^2(\Omega)}$, while
\[
\begin{aligned}
&
I_7 \leq C  %C_{g''}(z_1,z_2) 
\|z_1-z_2\|_{L^3(\Omega)}  \| \eps(u_1+u_D(t_1))\|_{L^6(\Omega)}
\leq C' % C_{g'} (z_1,z_2) 
   \|z_1-z_2\|_{L^3(\Omega)}   P(z_1,0), 
 \\
 & 
 \begin{aligned}
 I_8  & \leq  C  \| \eps(u_1+u_D(t_1)) +  \eps(u_2+u_D(t_2)) \|_{L^6(\Omega)}  \| \eps(u_1+u_D(t_1)) -  \eps(u_2+u_D(t_2)) \|_{L^3(\Omega)}
 \\
 & \leq  C'  P(z_1,z_2)^3   (|t_1-t_2| +  \|z_1 - z_2 \|_{L^{6} (\Omega)})\,.
\end{aligned}
\end{aligned}
\]
thanks to estimates \eqref{H2_improved-1} and \eqref{H2_improved-2} and the fact that $g',\, g''\in L^\infty(\R)$.
The proof of \eqref{stim-l4} follows the very same lines: we estimate $\| 
f'(z_1)-f'(z_2)\|_{L^{4/3}(\Omega)} $ by means of $ C_{f'}(z_1,z_2) 
\|z_1-z_2\|_{L^{4/3}(\Omega)}$, while we have with H\"older's inequality  
\[
\begin{aligned}
&
 \| (g'(z_1){-}g'(z_2)) \wt W(t_1,\nabla u_1) \|_{L^{4/3}(\Omega)}  \leq C 
\|z_1-z_2\|_{L^4(\Omega)}    \| \eps(u_1+u_D(t_1))\|_{L^4(\Omega)}^{
2 }  
 \leq C' %( P(z_1,0)^2 +1)
 \|z_1-z_2\|_{L^4(\Omega)},
\end{aligned}
\]
where the last estimate follows from  \eqref{crucial-reg-estimate} with 
 $p=4$.   
Finally, 
\[
\begin{aligned}
&
\| g'(z_2) (\wt W(t_1,\nabla u_1){-} \wt W(t_2,\nabla u_2)) \|_{L^{4/3}(\Omega)} 
\\ &  \leq C   \| \eps(u_1+u_D(t_1)) +  \eps(u_2+u_D(t_2)) \|_{L^4(\Omega)}  \| \eps(u_1+u_D(t_1)) -  \eps(u_2+u_D(t_2)) \|_{L^{2}(\Omega)} \\ & \leq  C'   P(z_1,z_2)^3    (|t_1-t_2| +  \|z_1 - z_2 \|_{L^{3} (\Omega)}). 
\end{aligned}
\]
This concludes the proof. 
\end{proof}
\par
  From all of the above results, and in particular from 
Lemma \ref{l:diff_contz},  we now draw a series of consequences on which our subsequent analysis will rely.
First of all, we observe the Fr\'echet differentiability of the functional $z \in \calZ \mapsto \calI(t,z)$. This is  due to the continuity of the mapping
$z\in \calZ \mapsto \rmD_z \calI(t,z) \in \calZ^*$, in turn due to the continuity of $z\mapsto A_q z$ and of $z \mapsto \rmD_z \wt\calI(t,z) $.
If restricted to bounded sets in $\calZ$, the latter mapping is even 
 continuous with values in  $L^2(\Omega)$ w.r.t.\ to  $L^6(\Omega)$-convergence for $z$,  cf.\ \eqref{enhanced-stim-7} (and the restriction of the power functional $\partial_t \calI$ is continuous w.r.t.\ $L^2(\Omega)$-convergence for $z$). Taking into account that $\calZ \Subset L^6(\Omega)$,  we may then claim the continuity of $ \rmD_z \wt\calI$ and $\partial_t \calI$ w.r.t.\ weak convergence  in $\calZ$. 
\begin{corollary}[Fr\'echet differentiability of $\calI$]
\label{coro-fre}
Under Assumptions \ref{ass:domain}, \ref{assumption:energy}, and \ref{ass:load},
the functional $\calI $ is Fr{\'e}chet
differentiable on $[0,T]\times \calZ$ and
\begin{equation}
\label{strong-continuity} \text{$t_n\rightarrow t$ and
$z_n\to z$ strongly in $\calZ$ implies }   \rmD_z\calI(t_n,z_n) \to
\rmD_z\calI(t,z) \ \text{strongly  in} \  \calZ^*.
\end{equation}
Furthermore,
\begin{equation}
\label{weak-continuity}
\begin{aligned}
&
\text{$t_n\rightarrow t$ and
$z_n\rightharpoonup z$   in $\calZ$
 implies }
\\
&
\liminf_{n \to \infty}\calI(t_n,z_n)\geq  \calI(t,z), \quad
\wt\calI(t_n,z_n)\rightarrow \wt\calI(t,z), \quad
\partial_t\calI(t_n,z_n)\rightarrow \partial_t\calI(t,z), \\
&
 \rmD_z\wt\calI(t_n,z_n) \to
\rmD_z\wt\calI(t,z) \ \text{strongly in} \  L^2(\Omega).
\end{aligned}
\end{equation}
\end{corollary}
We now  observe  a sort of (generalized) $\lambda$-convexity property for 
$\calI(t,\cdot)$,  
 \eqref{unif-l-convexity}  below,
involving the  $H^1(\Omega)$ and the $L^1(\Omega)$-norm,
valid on bounded sets in $\calZ$ (indeed, note that  the constant modulating  
the $L^1(\Omega)$-norm  in \eqref{unif-l-convexity} depends on 
 the radius of a $\calZ$-ball). 
 \begin{corollary}[$\lambda$-convexity of $\calI$] 
 \label{l-convexity-calI}
 Under Assumptions \ref{ass:domain}, \ref{assumption:energy}, and 
\ref{ass:load},   there exists a constant $\alpha>0$ and  for every $M>0$ there 
exists $ \Lambda_M>0$ such that  for every $t\in [0,T]$, $z_1,\,z_2 \in \calZ$
 with $\|z_1\|_\calZ + \|z_2\|_\calZ  \leq M$ and for every 
   $\theta \in [0,1]$ the   functional $\calL$ with   
   \begin{equation}
   \label{pertub-calI}
   \calL(t,z): = \calI(t,z) + \frac12 \|z\|_{L^2(\Omega)}^2
   \end{equation}
    complies with 
 \begin{equation}
 \label{unif-l-convexity} 
 \begin{aligned}
 \calL(t,(1{-}\theta)z_1 +\theta z_2)  \leq  & (1{-}\theta)\calL(t,z_1) + 
\theta \calL(t,z_2) 
 \\
 & \quad 
 -\theta(1{-}\theta) (\alpha \|z_1{-}z_2\|_{H^1(\Omega)}^2 
 - \Lambda_M \|z_1{-}z_2\|_{L^1(\Omega)}^2) .
 \end{aligned}
 \end{equation}
 \end{corollary}
 \begin{proof}
 From \eqref{eq.monGq} 
 it follows that  the mapping $A \in \R^3 \mapsto G_q(A) - 
\tfrac{c_q}2 |A|^2$ is 
convex, which entails that $A \mapsto G_q(A) $ is $c_q$-convex,
  i.e.\  there holds 
  $G_q(\segm{A_1}{A_2}) \leq \segm{G_q(A_1)}{G_q(A_2)} - \tfrac{c_q}2 
\theta(1{-}\theta) |A_1{-}A_2|^2$ for every $A_1,\, A_2 \in \R^3$ and $\theta 
\in [0,1]$.
  As a consequence, we have that 
  \begin{equation}
  \label{property-for-I_q}
  \calI_q(\segm{z_1}{z_2}) \leq \segm{\calI_q(z_1)}{\calI_q(z_2)}   - 
\frac{c_q}2 \theta(1{-}\theta) \int_\Omega |\nabla (z_1{-}z_2)|^2 \dd x\,.
  \end{equation}
  As for $\wt \calI$, 
 with trivial calculations we have that 
\[
\begin{aligned}
 &\wt\calI(t,(1{-}\theta)z_1 +\theta z_2)  -  (1{-}\theta)\wt\calI(t,z_1) -  \theta \wt\calI(t,z_2) 
\\ &  =  (1{-}\theta) \left(  \wt\calI(t,(1{-}\theta)z_1 +\theta z_2) - 
\wt\calI(t,z_1)  \right) + \theta \left(  \wt\calI(t,(1{-}\theta)z_1 +\theta 
z_2) - \wt\calI(t,z_2)  \right)  
 \doteq I_1 +I_2.
 \end{aligned}
\] 
\par
 There holds
 \[
 \begin{aligned}
 I_1  & = (1{-}\theta)\int_0^1  \int_\Omega  \rmD_z  \wt \calI (t,(1{-}s)z_1 + 
s((1{-}\theta)z_1 +\theta z_2))   \theta(z_2{-}z_1) \dd x \dd s  
 \\
  & 
=(1-\theta) \theta \int_0^1\int_\Omega   
\left( \rmD_z  \wt \calI (t,(1{-}s)z_1 + s((1{-}\theta)z_1 +\theta z_2)) {-} \rmD_z \wt \calI(t,z_1) \right)  (z_2{-}z_1) \dd x  \dd s
\\ &\quad   -  (1-\theta) \theta \int_\Omega   \rmD_z \wt \calI(t,z_1)    
(z_1{-}z_2) \dd x 
   \doteq I_{1,1}+ I_{1,2}\,.
   \end{aligned}
 \]
 We now estimate  $I_{1,1}$ by using H\"older's inequality
 and inequality \eqref{stim-l4},
  taking into account that
 $(1{-}s)z_1 + s((1{-}\theta)z_1 +\theta z_2) - z_1 = s\theta (z_2{-}z_1)$. Therefore,
    \[
 \begin{aligned}
\left|  I_{1,1} \right|  &  \leq c_8 \theta(1{-}\theta)   \int_0^1 (1+ C_{f'}(z_1,\zeta_{1,2}) + P(z_1, \zeta_{1,2})^3 ) \| s\theta(z_2{-}z_1)\|_{L^{4}(\Omega)} \|z_2{-}z_1\|_{L^4(\Omega)} \dd s 
\\
& \leq \wt{C}_1(M)  (1-\theta) \theta  \|z_2{-}z_1\|_{L^4(\Omega)}^2, 
 \end{aligned}
 \]
 where we have used the place-holder $\zeta_{1,2}:= (1{-}s)z_1 + s((1{-}\theta)z_1 +\theta z_2)$, and where
   $ \wt{C}_1(M)>0$ depends on the constant $M$ that bounds $\|z_1\|_\calZ $ and $\|z_1\|_\calZ $.
 With analogous calculations one has that 
 \[
 I_2 \leq  \wt{C}_1(M)  (1-\theta) \theta  \|z_2{-}z_1\|_{L^4(\Omega)}^2  +  \dddshort{(1-\theta) \theta \int_\Omega   \rmD_z \wt \calI(t,z_2)    (z_1{-}z_2) \dd x  }{$I_{2,2}$}.
 \]
Therefore,  estimating $I_{1,2} + I_{2,2} \leq  \wt{C}_2(M)  (1-\theta) \theta  \|z_2{-}z_1\|_{L^4(\Omega)}^2$  with the same arguments as above, we conclude that 
\begin{equation}
\label{property-4-tildeI}
\wt\calI(t,(1{-}\theta)z_1 +\theta z_2)  \leq  (1{-}\theta)\wt\calI(t,z_1)  +  
\theta \wt\calI(t,z_2) + \frac{\wt{C}(M)}2  (1-\theta) \theta  
\|z_2{-}z_1\|_{L^4(\Omega)}^2 
\end{equation}
for some $\wt{C}(M)>0$. 
We now combine 
\eqref{property-for-I_q} with \eqref{property-4-tildeI}.  
Adding to this the trivial identity 
\[
\frac12 \| \segm{z_1}{z_2}\|_{L^2(\Omega)}^2 = \segmh{ 
\|z_1\|_{L^2(\Omega)}^2}{ 
\|z_2\|_{L^2(\Omega)}^2} - \frac{(1{-}\theta)\theta}2 
\|z_1{-}z_2\|_{L^2(\Omega)}^2,  
\]
and 
using Ehrling's Lemma, cf.\ e.g.\ \cite[Thm.\ 7.30]{RenardyRogers04}, to estimate 
$\norm{\eta}_{L^4(\Omega)}^2\leq \delta  
\norm{ \eta}^2_{H^1(\Omega)} + C(\delta)\norm{\eta}_{L^1(\Omega)}^2$ for 
arbitrary $\delta >0$, finally results in \eqref{unif-l-convexity}. 
 \end{proof}
A slight generalization of property \eqref{unif-l-convexity} was proposed in  
\cite[Sec.\ 3.4, (3.63)]{MRS16} as a sufficient condition for 
 a sort of ``uniform differentiability'' condition for $\calI(t,\cdot)$,
 cf.\ \eqref{unif-subdiff-ahead} ahead, which was in turn
  introduced in \cite[Sec.\ 2.1, $(\mathrm{E}.3)$]{MRS16}. As we will see,
  \eqref{unif-subdiff-ahead}
   is at the core of  key chain rule properties for viscous solutions to 
\eqref{dndia-eps} and   for Balanced Viscosity solutions to \eqref{dndia}, cf.\ 
   Lemma \ref{l:ch-rule} and Theorem \ref{prop:ch-rule} ahead.
 As a trivial consequence of  \eqref{unif-subdiff-ahead}, we have 
 a   monotonicity property for the Fr\'echet subdifferential $\calD_z\calI$, 
which will  allow us to prove the (crucial, for our analysis)
 uniqueness of solutions for the time-incremental problems giving rise to 
discrete solutions. 
\begin{corollary}
\label{cor:rosenheim}
Under Assumptions \ref{ass:domain}, \ref{assumption:energy}, and \ref{ass:load}, 
 for every $M>0$ 
there exist constants $c_9,\,c_{10}(M)>0$ such that for all $t \in [0,T]$,  $z_i \in \calZ$, $i=1,2$,  with
$\|z_1\|_{\calZ} + \|z_2\|_{\calZ} \leq M$, 
we have
\begin{equation}
\label{unif-subdiff-ahead}
\calL(t,z_2) - \calL(t,z_1) \geq 
\pairing{}{\calZ}{\rmD_z\calL(t,z_1)}{z_2{-}z_1} +\alpha 
\|z_1-z_2\|_{H^1(\Omega)}^2 - \Lambda_M  \|z_1-z_2\|_{L^1(\Omega)}^2 \,. 
\end{equation}
As a consequence, there holds
\begin{equation}
\label{e:strong-monot-sez2}
\norm{z_1{-}z_2}_{L^2(\Omega)}^2+ 
\pairing{}{\calZ}{\rmD_z\calI(t,z_1){-}\rmD_z\calI(t,z_2)}{z_1{-}z_2} \geq c_9 
\norm{z_1{-}z_2}_{H^1(\Omega)}^2 - 
c_{10}(M) \norm{z_1{-}z_2}_{L^{2} (\Omega)}^2\,.
\end{equation}
\end{corollary}
\noindent Note that, in accordance with \eqref{unif-l-convexity}   and 
\eqref{unif-subdiff-ahead},  only the constant $c_{10}$ depends on $M$. 
\begin{proof}
Estimate \eqref{unif-subdiff-ahead} can be deduced from  
\eqref{unif-l-convexity}   by the very same calculations  as in the proof of  
\cite[Lemma 3.26]{MRS16}, while
\eqref{e:strong-monot-sez2} can be obtained by adding \eqref{unif-subdiff-ahead} 
 with the estimate obtained exchanging $z_1$ with $z_2$, and observing that $-  
\|z_1-z_2\|_{L^1(\Omega)}^2 \geq -C   \|z_1-z_2\|_{L^2(\Omega)}^2$. 
\end{proof}
 A key ingredient for the proof of energy identities in the context of 
solutions to the \emph{viscous} damage system \eqref{dndia-eps} (cf.\  
Section~\ref{s:4}), and of $\BV$ solutions to the rate-independent \eqref{dndia} 
(cf.\ Section~\ref{s:5}),   
  %\eqref{en-diss-bal} 
  is the validity of the   
chain rule identity (but, indeed, a chain rule \emph{inequality} would suffice)
  \begin{equation}
  \label{ch-rule-identity}
  \frac{\dd}{\dd t}\calI(t,z(t))  - \partial_t \calI(t,z(t))= \langle 
\rmD_z 
\calI(t,z(t)), z'(t) \rangle_{L^2(\Omega)}  \qquad \foraa\, t \in (0,T), 
  \end{equation}
  along solution curves $z:[0,T]\to \calZ$ with $\rmD_z\calI(t,z(t))\in 
L^2(\Omega)$. Since   $\calI\in \rmC^1([0,T] 
\times \calZ)$,  the validity of \eqref{ch-rule-identity} with  the duality 
pairing 
$\pairing{}{\calZ}{\cdot}{\cdot}$ is guaranteed along any curve $z \in 
\AC([0,T];\calZ)$. The following result extends  \eqref{ch-rule-identity} to 
curves $z $ with weaker regularity and summability properties.
  \begin{lemma}[Chain rule for $\calI$ in $L^2(\Omega)$]
  \label{l:ch-rule}
     Under Assumptions   \ref{ass:domain}, \ref{assumption:energy},  and 
\ref{ass:load},  for every curve
     \begin{equation}
     \label{ch-rule-curve}
     z\in L^\infty(0,T;\calZ) \cap H^1(0,T;L^2(\Omega)), \text{ with } A_q z 
\in 
L^2(0,T;L^2(\Omega)),   
     \end{equation}
      the map $t \mapsto \calI(t,z(t))$ is absolutely continuous on $[0,T]$, 
and 
 \eqref{ch-rule-identity} holds. 
  \end{lemma}
  \begin{remark}
  \label{rmk:alternative-ch-requir}
  \upshape
  Due to  estimate \eqref{estimate-for-DI} for $\rmD_z \wt \calI$, it follows 
from     \eqref{ch-rule-curve} 
   that the function $t\mapsto  \rmD_z \calI(t,z(t)) $ belongs to $L^2 
(0,T;L^2(\Omega))$. Therefore, 
   $\rmD_z \calI (t,z(t)) = A_q(z(t)) + \rmD_z \wt\calI(t,z(t))$ belongs to $
L^2 (0,T;L^2(\Omega))$  as well and the integral on the r.h.s.\ of 
\eqref{ch-rule-identity}    is well defined for almost all $t\in (0,T)$.
\par
In fact, for later use in Sec.\ \ref{s:5}, let us point out that,
in alternative to \eqref{ch-rule-curve}, in Lemma \ref{l:ch-rule} we might as 
well suppose 
\begin{equation}
\label{ch-rule-curve-weak} 
   z\in L^\infty(0,T;\calZ) \cap W^{1,1}(0,T;L^2(\Omega)), \text{ with } A_q z 
\in L^\infty(0,T;L^2(\Omega)). 
\end{equation}
  %In fact, conditions 
  \end{remark}
  \begin{proof}
   First of all, we show the absolute continuity of $t\mapsto \calI(t,z(t))$.  
We will in fact  show that $t\mapsto \calL(t,z(t))$ is absolutely continuous, 
with $\calL$ from 
   \eqref{pertub-calI}. With this aim, for every $0\leq s \leq t \leq T$ we 
estimate
   \[
   \calL(t,z(t)) - \calL(s,z(s))= \calL(t,z(t)) - \calL(s,z(t)) + \calL(s,z(t)) 
- \calL(s,z(s)) \doteq I_1 +I_2\,. 
   \]
  Since $\partial_t \calL = \partial_t \calI$, we  have 
   \begin{equation}
   \label{AC4I1}
   |I_1| \leq \int_s^t \partial_t \calI(r,z(t)) \dd r \stackrel{(1)}{\leq} 
C(t-s)
   \end{equation}
   with (1) due to \eqref{stim3}.
   As for $I_2$, 
  from 
   the uniform differentiability property
   \eqref{unif-subdiff-ahead} we deduce that 
      \begin{equation}
   \label{toMRS13}
I_2 \geq \int_\Omega \rmD_z\calL(t,z(s))(z(t){-}z(s)) \dd 
x %\rangle_{L^2(\Omega)} 
+\alpha \|z(t)-z(s)\|_{H^1(\Omega)}^2 -  \Lambda_M  
\|z(t)-z(s)\|_{L^1(\Omega)}^2
  \end{equation}
  (cf.\ notation 
\eqref{abuse-integral}), 
  where we have used that, by   \eqref{ch-rule-curve}
  and estimate \eqref{estimate-for-DI} for $\rmD_z \wt \calI$
   that the function $s\mapsto  \rmD_z \calI(s,z(s)) $ belongs to $L^2 
(0,T;L^2(\Omega))$, and  so does $s\mapsto \rmD_z\calL(t,z(s))$, with $t\in 
[0,T]$ fixed, due to
   \eqref{enhanced-stim-7}. 
All in all we arrive at
\begin{equation}
\label{heartsuit}
\begin{aligned}
 \abs{\calL(s,z(s)) - \calL(t,z(t))}&\leq 
 2\Lambda_M\norm{z(t) - z(s)}_{L^1(\Omega)}^2 + 2 c\abs{t-s} \\
& + \left( \norm{\rmD_z\calL(t,z(t))}_{L^2(\Omega)} 
 + \norm{\rmD_z\calL(s,z(s))}_{L^2(\Omega)}\right) 
 \norm{z(t) - z(s)}_{L^2(\Omega)}.
\end{aligned}
\end{equation}
Up to a suitable reparameterization, cf.\ \cite[Lemma 1.1.4]{AGS2008}, we 
can suppose that $z\in W^{1,\infty}(0,\wt T;L^2(\Omega))$ with Lipschitz 
constant $1$. With \cite[Lemma 1.2.6]{AGS2008} we finally conclude from \eqref{heartsuit} the absolute 
continuity of $t\mapsto \calL(t,z(t))$, which gives the same property 
for $t\mapsto 
\calI(t,z(t))$. 
For the proof of identity \eqref{ch-rule-identity}, we refer to \cite[Prop.\ 2.4]{mrs2013}.
  \end{proof}
 
 \section{\bf A priori estimates for the time-discrete solutions}
 \label{s:3}
 We construct time-discrete solutions to the Cauchy problem for the  viscous damage system   \eqref{dndia-eps} by solving  the following
 time incremental minimization problems:
  for fixed $\epsilon>0$, we consider a  uniform  partition $\{0=t_0^\tau<\ldots<t_N^\tau=T\}$ of the time
interval $[0,T]$ with fineness $\tau = 
t^\tau_{k+1} - t^\tau_k =T/N$.  % (cf.\ Remark \ref{rmk:fixed-time-step} ahead),
 The elements  $(z^\tau_{k})_{0\leq
k\leq
N}$ are determined through $z_0^\tau:=z_0 \in \calZ$ and
\begin{align}
\label{def_time_incr_min_problem_eps} 
z_{k+1}^\tau \in
\Argmin\Bset{\calI(t_{k+1}^\tau,z) +
\tau\calR_\epsilon\left(\frac{z
    -z_k^\tau}{\tau}\right)}{z\in\calZ}, \qquad k\in \{0, \ldots, N-1\}.
\end{align}
 Our first result, Prop.\ \ref{prop:exist-mini} below, states the 
existence of minimizers 
for problem \eqref{def_time_incr_min_problem_eps}, which is an immediate outcome of 
classical variational arguments, as well as the uniqueness of solutions to the associated Euler-Lagrange equation \eqref{discr-Eul-Lagr} below. This will be a key ingredient in the proof of 
the main result of this section, Proposition \ref{prop:aprio} ahead. Indeed, in 
order to obtain some of the a priori estimates stated therein, we shall have to 
perform calculations on an approximate version of \eqref{discr-Eul-Lagr}. Then, 
the above mentioned uniqueness property  will ensure that those a priori 
estimates also hold for the solutions to  \eqref{discr-Eul-Lagr}, i.e.\ for the  
minimizers from \eqref{def_time_incr_min_problem_eps}.  
\begin{proposition}
\label{prop:exist-mini} Under Assumptions \ref{ass:domain}, 
\ref{assumption:energy}, and \ref{ass:load}, for every $\epsilon,\, \tau>0$  and 
for every $ k\in \{1, \ldots, N-1\}$ 
 the minimum problem
\eqref{def_time_incr_min_problem_eps} admits a  solution $z_{k+1}^\tau$ 
satisfying the Euler-Lagrange equation 
\begin{equation}
\label{discr-Eul-Lagr}
\omega + \epsilon  \frac{z -
z_k^\tau}{\tau} + \rmD_z\calI(t_{k+1}^\tau,z) = 0 \qquad
\text{in } \calZ^*, \qquad \text{with } \omega \in  \partial_{\calZ,\calZ^*} 
\calR_1\left(\frac{z -
z_k^{\tau}}\tau\right),
\end{equation}
where $\partial_{\calZ,\calZ^*} \calR_1: \calZ \rightrightarrows \calZ^*$ is 
the convex analysis subdifferential of $\calR_1$. 
Moreover, 
for every $\epsilon>0$  
 and for every $M>0$,  
there exists $ \tau(\epsi,M)>0$ such that for all $0<\tau \leq 
\tau(\epsilon,M)$  the %associated 
Euler-Lagrange equation 
\eqref{discr-Eul-Lagr} 
admits at most one  solution
 in the closed ball $\overline{B}_M(0)$ of $\calZ$. 
\par
Suppose in addition that $f$ and $g$ comply with the following condition
\begin{equation}
\label{def-gg}
  f(0) \leq f(z), \quad  g(0) \leq g(z) \quad  \text{  for all } z \leq
 0,
\end{equation}
 and that  the initial datum $z_0$ fulfills
$ z_0(  x)\in [0,1]$  for all $x\in \Omega$. 
   Then,
the minimizer $z_{k+1}^\tau$ from \eqref{def_time_incr_min_problem_eps}
   also fulfills
  $z_{k+1}^\tau(x) \in  [0,1]$ for all $x\in \Omega$ and all $k\in \{0,\ldots, N-1\}. $
\end{proposition}
\begin{proof}
The existence of minimizers can be checked via the direct method in
the calculus of variations. Observe that every minimizer fulfills 
\eqref{discr-Eul-Lagr}, where we have used   that the convex analysis 
subdifferential $\partial_{\calZ,\calZ^*} \calR_\epsi : \calZ \rightrightarrows 
\calZ^*$ is given by $\partial_{\calZ,\calZ^*} \calR_\epsi (\eta) = 
\partial_{\calZ,\calZ^*} \calR_1 (\eta) + \epsilon \eta $ for every $\eta \in 
\calZ$ (here and in what follows, for notational simplicity we write $\eta$ in 
place of $J(\eta)$, with $J : \calZ \to \calZ^*$ the Riesz isomorphism). 
\par
  In order to check that the Euler-Lagrange equation \eqref{discr-Eul-Lagr} has a unique solution, let
   $M>0$ and 
   $z_1,\, z_2 \in \calZ$ be solutions to \eqref{discr-Eul-Lagr} such that 
   $\|z_1\|_\calZ + \|z_2\|_\calZ  \leq M$. 
    Subtracting the equation for $z_2$ from that for $z_1$ and testing the 
obtained relation by $z_1-z_2$, we obtain 
\[
\begin{aligned}
0  & = \pairing{}{\calZ}{\omega_1-\omega_2}{z_1-z_2} + \frac\epsilon\tau \| z_1 
{-}z_2\|_{L^2(\Omega)}^2 + \pairing{}{\calZ}{\rmD_z\calI(t_{k+1}^\tau,z_1) {-} 
\rmD_z\calI(t_{k+1}^\tau,z_2)}{z_1 {-}z_2}  \\ & \geq  \left(  \frac\epsilon\tau 
 -   c_{10}(M)  - 1\right)  \| z_1 {-}z_2\|_{L^2(\Omega)}^2 +c_9  \| z_1 
{-}z_2\|_{H^1(\Omega)}^2 \end{aligned}
\]
where $\omega_i \in  \partial\calR_1\left(\frac{z_i -
z_k^\tau}{\tau}\right)$ for $i=1,2$, and the second inequality follows from  
the monotonicity estimate \eqref{e:strong-monot-sez2}.
 Hence,  for $\tau \leq \tau(\epsilon,M): =  \tfrac \epsilon 
{(c_{10}(M){  +  }1)}$, 
we conclude that $ \| z_1 {-}z_2\|_{L^2(\Omega)}^2 \leq 0$, whence $z_1=z_2$. 
\par
For the proof of 
the property $z_k^\tau \in  [0,1]$ in $\Omega$ under \eqref{def-gg}, we refer to \cite[Prop.\ 
4.5]{krz}. \end{proof}
\par
 The following piecewise constant and piecewise
linear interpolation functions will be used: 
\[
\overline z_\tau (t) = z_{k+1}^\tau \,\,\text{for }t\in
(t_k^\tau,t_{k+1}^\tau],\,\,\, 
\underline z_\tau(t)=  z_k^\tau \,\,\text{for }t\in
[t_k^\tau,t_{k+1}^\tau),\,\,\,
\widehat z_\tau(t)=z_k^\tau + \frac{t - t_{k}^\tau}{\tau} (z_{k+1}^\tau
- z_k^\tau)\,\,\text{for }t\in [t_k^\tau,t_{k+1}^\tau].
\]
 Furthermore, we shall use the
notation
\[
\begin{array}{llll}
 \tau(r) &= \tau &\quad&\text{for } r\in (t_k^\tau, t_{k+1}^\tau),
\\
\overline{t}_\tau(r)&= t_{k+1}^\tau &\quad&\text{for } r\in (t_k^\tau,
t_{k+1}^\tau],
\\
\underline{t}_\tau(r)&= t_{k}^\tau &\quad&\text{for } r\in [t_k^\tau,
t_{k+1}^\tau),
\\
\overline u_\tau (r) & = \umin (\overline{t}_\tau(r), \overline
z_\tau (r)) &  \quad&\text{for }r\in
(t_k^\tau,t_{k+1}^\tau],\\
\underline u_\tau (r) & = \umin (\underline{t}_\tau(r), \underline
z_\tau (r))&\quad &\text{for }r\in [t_k^\tau,t_{k+1}^\tau),
\\
\widehat u_\tau(r)&= \underline u_\tau (r) + \frac{r -
\underline{t}_\tau(r)}{\tau} (\overline u_\tau (r) - \underline
u_\tau (r))  &\quad &\text{for }r\in [t_k^\tau,t_{k+1}^\tau].
\end{array}
\]
 Clearly,
\begin{equation}
\label{quoted -later}
 \overline{t}_\tau(t), \, \underline{t}_\tau(t) \to t
 \quad \text{ as
$\tau \to 0$ for all   $t \in (0,T)$, and }   \underline{t}_\tau(0)=0, 
\  \overline{t}_\tau(T) =T. 
\end{equation}
We will also denote by $\overline{\ell}_\tau$ and $\overline{u}_{D,\tau}$ the (left-continuous) piecewise constant interpolants of the values $ (\ell_k^\tau : = \ell ( t_{k}^\tau))_{k=0}^N, \,  (u_{D,k}^\tau : = u_D( t_{k}^\tau))_{k=0}^N$ and, for a given $N$-uple $\{v_\tau^k \}_{k=0}^{N}$, use 
the short-hand notation 
\[
\Dtau k v : = v^\tau_{k+1} - v_k^\tau .
\]
In view of \eqref{discr-Eul-Lagr} and of formula \eqref{der-decomp} for $\rmD_z 
\calI$,  
the above interpolants fulfill  for almost all $t\in (0,T)$
\begin{equation}
\label{cont-reformulation}
\overline\omega_\tau (t) +  \epsilon  \widehat{z}_\tau'(t)    + A_\il \overline{z}_\tau(t) +
   \rmD_z\wt \calI(\overline{t}_\tau(t),
 \overline{z}_\tau(t)) = 0 \quad \text{in } \calZ^*, \quad \text{with }
 \overline\omega_\tau (t)  \in  \partial_{\calZ, \calZ^*} 
\calR_1\left(\widehat{z}_\tau'(t)\right) \,.
\end{equation}
The following  result collects 
 all the a priori estimates on the functions $(\overline{z}_\tau,  
\widehat{z}_\tau, 
\overline{u}_\tau, \widehat{u}_\tau)_\tau$, 
 uniform w.r.t.\ the parameters $\epsilon,\, \tau>0$, that are  at the core of 
the existence of solutions of  
 the viscous system,  cf.\ Theorem \ref{thm:exist} ahead,   and of its 
vanishing-viscosity analysis developed in Section \ref{s:5}.
 In fact, let us mention that the estimates 
 for $(\overline{u}_\tau, \widehat{u}_\tau)_\tau $ have to be understood as side results,  while the really relevant 
 bounds for the limit passage are those for $(\overline{z}_\tau, 
\widehat{z}_\tau)$. 
   We also prove that 
 the Euler-Lagrange equation \eqref{cont-reformulation} holds in $L^2(\Omega)$, with $\partial_{\calZ, \calZ^*} \calR_1$ 
 replaced by the subdifferential operator $\partial_{L^2(\Omega)} \calR_1: L^2(\Omega) \rightrightarrows L^2(\Omega)$. From now on,
 we will denote the latter operator  by $\partial \calR_1$. 
  \begin{proposition}
 \label{prop:aprio}
 Under Assumptions   \ref{ass:domain}, \ref{assumption:energy}, and \ref{ass:load},
 suppose that the initial datum $z_0 \in 
\calZ$ fulfills in addition 
\begin{equation}
 \label{further-reg}
 A_q z_0 
\in L^2(\Omega).
 \end{equation}
Then, for every $\epsilon>0$  there exists $\bar{\tau}_\epsilon>0$, only 
depending
on $\epsilon$ and on the problem data (cf.\ \eqref{bartauepsilon} ahead), such 
that for every   $\tau \in (0,\bar{\tau}_\epsilon)$  
% with $\tau_\epsilon$ from Proposition \ref{prop:exist-mini}  
there 
holds
\begin{equation}
\label{enh-spat-reg}
 A_\il \overline{z}_\tau \in L^\infty (0,T;L^2(\Omega))  \text{ and }  
\overline\omega_\tau  \in L^\infty(0,T;L^2(\Omega) ), 
 \end{equation}
with $\overline\omega_\tau$ a selection in $\partial_{\calZ,\calZ^*} 
\calR_1(\widehat{z}_\tau')$  which fulfills \eqref{cont-reformulation}.  
Therefore,  
 the functions $(\overline{t}_\tau, \overline{z}_\tau, \widehat{z}_\tau)$ satisfy 
 \begin{equation}
 \label{Eul-Lagr-L2}
\partial \calR_1 (\widehat{z}_\tau'(t))+ \epsilon  \widehat{z}_\tau'(t)     + \rmD_z\calI(\overline{t}_\tau(t),
 \overline{z}_\tau(t))  \ni  0 \qquad \text{ in }L^2(\Omega) \ \foraa\, t \in (0,T). 
 \end{equation}
 Furthermore,    there exist  constants $C,\, C(\epsi),\, C(\sigma)>0$, with 
$C(\epsi) \uparrow +\infty$ as $\epsi  \downarrow 0$,  such that for all 
$\epsilon>0$ and $ \tau \in (0,\bar{\tau}_\epsilon) $ the following estimates 
hold: 
 \begin{subequations}
 \label{est-epsi-tau}
 \begin{align}
  & 
   \label{est-epsi-tau-0}
   \sup_{t\in [0,T]} \left| \calI(\overline{t}_\tau(t), \overline{z}_\tau(t)) \right| \leq C,
 \\
  & 
 \label{est-epsi-tau-1}
 \| \overline{z}_\tau \|_{L^\infty (0,T;W^{1,q}(\Omega))} 
 +  \| \widehat{z}_\tau \|_{L^\infty (0,T;W^{1,q}(\Omega))}   \leq C,
 \\ 
   & 
 \label{est-epsi-tau-1reg}
  \| \overline{z}_\tau \|_{L^\infty (0,T;W^{1+\sigma,q}(\Omega))} \leq 
C(\sigma) \text{ for all }0<\sigma<\tfrac1q,  
 \\
 & 
  \label{est-epsi-tau-2}
 %\epsilon^{1/2} 
   \| \widehat{z}_\tau' \|_{L^2 (0,T;H^{1}(\Omega))}  +  % \epsilon^{1/2} 
  \| \widehat{z}_\tau' \|_{L^\infty (0,T;L^{2}(\Omega))} \leq  C(\epsi), 
  \\ & 
   \label{est-epsi-tau-3}
  \| \widehat{z}_\tau \|_{W^{1,1} (0,T;H^{1}(\Omega))} \leq C,
 \\ & 
  \label{est-epsi-tau-4}
 \| A_q (\overline{z}_\tau)\|_{L^\infty (0,T;L^{2}(\Omega))} \leq C,
 \\ &
   \label{est-epsi-tau-5}
  %\epsilon^{1/2}
   \| \overline{\omega}_\tau \|_{L^\infty (0,T;L^{2}(\Omega))} \leq  
C(\epsi), 
 \\ 
 & 
   \label{est-epsi-tau-6}
 \| \overline{u}_\tau \|_{L^\infty (0,T;H^{2}(\Omega))} \leq C,
 \\
  & 
    \label{est-epsi-tau-7}
 %\epsilon^{1/2} 
 \| \widehat{u}_\tau' \|_{L^2(0,T;W^{1,3}(\Omega))} \leq C(\epsi), 
 \\  
  & 
    \label{est-epsi-tau-8}
  \| \widehat{u}_\tau \|_{W^{1,1}(0,T;W^{1,3}(\Omega))} \leq C.
 \end{align}
 Therefore, 
 \begin{align}
  \label{est-epsi-tau-4-bis}
 \|\rmD_z \calI (\overline{t}_\tau,\overline{z}_\tau)\|_{L^\infty 
(0,T;L^{2}(\Omega))} \leq C. 
 \end{align} 
 \end{subequations}
 \end{proposition}
 \par
 Based on Proposition \ref{prop:aprio} we derive a  discrete energy inequality, cf.\  \eqref{discr-enineq} below, 
involving the Fenchel-Moreau conjugate of the functional $\calR_\epsi$ w.r.t.\ the scalar product in $L^2(\Omega)$,
 namely the functional
\begin{equation}
\label{Fenchel-Moreau}
\calR_\epsi^*: L^2(\Omega) \to [0,+\infty) \qquad \text{defined by } 
\calR_\epsi^*(\xi): = \frac1{2\epsilon}
\min_{\eta \in \partial\calR_1(0)} \|\xi -\eta\|_{L^2(\Omega)}^2\,.
\end{equation}
 Observe that we are in a position to work with this Legendre transform of 
$\calR_\epsi$, and not
with the one w.r.t.\ the $(\calZ,\calZ^*)$-duality, 
 relying on the  fact that $\rmD_z
\calI (\overline{t}_{\tau}(t),\overline{z}_{\tau}(t)) \in L^2(\Omega)$ for 
almost all $t\in (0,T)$, thanks to \eqref{enh-spat-reg}.  

Let us mention in advance that \eqref{discr-enineq} will be the starting point of the vanishing-viscosity analysis developed in Sec.\ \ref{ss:5.3}.
We postpone  the   proof of Corollary \ref{cor:discr-enineq}   to the end of this section.
\begin{corollary}
 \label{cor:discr-enineq}
 Under Assumptions   \ref{ass:domain}, \ref{assumption:energy}, and \ref{ass:load},
 suppose that the initial datum $z_0$ fulfills \eqref{further-reg}.
 \par
 Then, there exists $C>0$ such that for every $\epsilon>0$ and $\tau \in (0,\bar{\tau}_\epsilon)$
 the functions $\overline{z}_\tau,\, \widehat{z}_\tau$ comply with the \emph{discrete energy-dissipation inequality} for every $0\leq s \leq t \leq T$
 \begin{equation}
\label{discr-enineq}
\begin{aligned}
&
\int_{\underline{t}_{\tau}(s)}^{\overline{t}_{\tau}(t)}
\left(\calR_\epsi (\widehat{z}'_{\tau}(r))+\calR_\epsi^* (-\rmD_z
\calI (\overline{t}_{\tau}(r),\overline{z}_{\tau}(r)))  \right)
\,\mathrm{d}r +
\calI( t,\widehat{z}_{\tau}(t))
\\ &
\leq \calI( s,\widehat{z}_{\tau}(s)) +
\int_{\underline{t}_{\tau}(s)}^{\overline{t}_{\tau}(t)}
\partial_t \calI (r,\widehat{z}_{\tau}(r)) \, \mathrm{d}r
\\
&+ C \sup_{t\in[0,T]}\|\overline{z}_{\tau}(t)-\widehat{z}_{\tau}(t)\|_{L^{2}(\Omega)}
\int_{\underline{t}_{\tau}(s)}^{\overline{t}_{\tau}(t)}(|\overline{t}_{\tau}(r)-r | + \|\overline{z}_{\tau}(r)-\widehat{z}_{\tau}(r)\|_{L^{6}(\Omega)} )\,\dd r.
\end{aligned}
\end{equation}
Therefore, there exists a constant $C>0$ such that for
every $\epsilon>0$ and $\tau \in (0,\bar{\tau}_\epsilon)$ 
\begin{subequations}
\begin{align}
\label{est-epsi-tau-1-bis}
&
 \sup_{t\in [0,T]} | \calI( t,\widehat{z}_{\tau}(t)) | \leq C, 
 \\
 & 
 \label{est-diss-epsi-tau}
 \int_{0}^{T}
\left(\calR_\epsi (\widehat{z}'_{\tau}(r))+\calR_\epsi^* (-\rmD_z
\calI (\overline{t}_{\tau}(r),\overline{z}_{\tau}(r)))  \right)
\,\mathrm{d}r \leq C\,.
\end{align}
\end{subequations}
 \end{corollary}
 \par
Let us now comment on the proof of Prop.\ \ref{prop:aprio}:  Estimates  
\eqref{est-epsi-tau} (and the related enhanced spatial regularity 
\eqref{enh-spat-reg}, which leads to \eqref{Eul-Lagr-L2}  as a 
subdifferential inclusion in $L^2(\Omega)$)  will be proved by performing on 
equation \eqref{cont-reformulation} the 
following a priori estimates: 
 \begin{description}
 \item[\textbf{Energy estimate}] based on the energy-dissipation inequality
 \begin{equation}
 \label{en-diss}
 \calI(\overline{t}_\tau(t), \overline{z}_\tau(t)) + \int_0^{\overline{t}_\tau(t)} \calR_\epsilon (\widehat{z}_\tau'(s)) \dd s \leq   \calI(0, z_0) + \int_0^{\overline{t}_\tau(t)}\partial_t \calI(s, \underline{z}_\tau(s)) \dd s 
 \end{equation}
 for every $t \in [0,T]$, it leads to the uniform bounds   
\eqref{est-epsi-tau-0}--\eqref{est-epsi-tau-1}.  Observe that the proof of 
 \eqref{en-diss} works for \emph{every} $\tau>0$.
 \end{description}
We then choose
 \begin{equation}
 \label{bartauepsilon} 
 \bar{\tau}_\epsilon : = \tau(\epsilon,M) \text{ according to Proposition 
\ref{prop:exist-mini}, with } M \geq  \sup_{\tau>0} \| \overline{z}_\tau 
\|_{L^\infty (0,T;W^{1,q}(\Omega))} \,. 
 \end{equation} 
\begin{description} 
 \item[\textbf{First regularity estimate}]  In view of estimate 
\eqref{H2_improved-1}, from the estimate for $\overline{z}_\tau $ in 
$L^\infty(0,T;W^{1,q}(\Omega))$ we deduce \eqref{est-epsi-tau-6}. 
 \item[\textbf{Enhanced energy estimate}]  it consists in (formally) 
differentiating \eqref{cont-reformulation} w.r.t.\ time, and testing it by 
$\widehat z_\tau'$. In view of the coercivity property \eqref{e2.22} of the  
elliptic operator $A_q$, this gives estimates \eqref{est-epsi-tau-2} \& 
\eqref{est-epsi-tau-3} for $\widehat{z}'_\tau$. 
  \item[\textbf{Second regularity estimate}] 
 Estimates \eqref{est-epsi-tau-7} \& \eqref{est-epsi-tau-8} for $\widehat{u}_\tau'$ derive from \eqref{est-epsi-tau-2} \& \eqref{est-epsi-tau-3}, respectively, via    the continuous dependence estimate \eqref{H2_improved-2}.
 \item[\textbf{Third regularity estimate}] it consists in testing \eqref{cont-reformulation} by  (the formally written term)  $\partial_t A_q \overline{z}_\tau$. This gives rise to estimate
 \eqref{est-epsi-tau-4}, which induces the spatial regularity 
\eqref{est-epsi-tau-1reg} by applying 
%the results from  \cite{Savare98},
Proposition \ref{prop:Savare98}, and 
 it induces \eqref{est-epsi-tau-5} by a comparison argument in \eqref{cont-reformulation}. 
 \end{description}
 \noindent
 The energy \& the enhanced energy estimates can be rendered rigorously on the 
discrete equation  \eqref{cont-reformulation}, cf.\ Lemma \ref{l:1st-3-est} 
below.   
  In its proof, we shall revisit 
 the  calculations  developed  in \cite[Sec.\ 5]{KRZ2}, 
 relying on the novel estimates provided by Lemmas \ref{l:2.7} and 
\ref{l:diff_contz}. 
 Instead, the third regularity estimate obtained in Lemma  \ref{l:last-est} 
ahead   cannot be performed directly on \eqref{cont-reformulation}. 
 In fact, 
  it would involve testing the subdifferential inclusion 
\eqref{cont-reformulation},  set in $\calZ^*$, by the difference $ \frac1\tau ( 
A_q \overline{z}_\tau(t) {-} A_q \underline{z}_\tau(t))$ which then 
should belong to $\calZ$. This cannot be rigorous, since  $A_q 
\overline{z}_\tau(t)$ is in principle an element of $\calZ^*$, only. Therefore,  
 in the proof of Lemma  \ref{l:last-est}  we shall perform all the calculations 
on an approximate version of \eqref{cont-reformulation}, featuring a 
regularized version of the dissipation potential~$\calR_1$.  
 \par
 \begin{lemma}
 \label{l:1st-3-est}
 Under Assumptions \ref{ass:domain}, \ref{assumption:energy}, and 
\ref{ass:load}, and the condition that the initial datum $z_0\in \calZ$ fulfills \eqref{further-reg}, estimates
  \eqref{est-epsi-tau-0}--\eqref{est-epsi-tau-1}, 
\eqref{est-epsi-tau-2}--\eqref{est-epsi-tau-3}, and  
\eqref{est-epsi-tau-6}--\eqref{est-epsi-tau-8} hold true for every 
$\tau>0$. 
 \end{lemma}
\begin{proof}
The discrete energy-dissipation inequality \eqref{en-diss}  
can be derived by choosing the competitor $z= z_{k}^\tau$ in  the minimum 
problem
 \eqref{def_time_incr_min_problem_eps}, which leads to 
 \[
 \begin{aligned}
\calI(t_{k+1}^\tau, z_{k+1}^\tau) + \tau_k \calR_\epsi
\left(\frac{z_{k+1}^\tau-z_{k}^\tau}{\tau_k} \right)   \leq
\calI(t_{k+1}^\tau, z_{k}^\tau)   = \calI(t_{k}^\tau,
z_{k}^\tau) + \int_{t_k^\tau}^{t_{k+1}^\tau} \partial_t \calI (s,
z_{k}^\tau) \,\dd s.
\end{aligned}
 \]
Then, \eqref{en-diss} follows upon summing over the index $k$. In view of 
 estimate
\eqref{stim3} 
on the power functional $\partial_t \calI$, 
and Assumption  \ref{ass:load}, the right-hand side of \eqref{en-diss} is 
uniformly bounded. 
Since the second term on its l.h.s.\ is non-negative, we immediately conclude 
estimate \eqref{est-epsi-tau-0}.  Then, the coercivity property 
\eqref{est_coerc1}, combined with Poincar\'e's inequality, gives 
\eqref{est-epsi-tau-1} for $\overline{z}_\tau$. The bound for 
$\widehat{z}_\tau$ then trivially follows.    From the bound for $\int_0^T 
\calR_\epsi (\widehat{z}_\tau'(t)) \dd t  $ we also infer 
that $\epsi^{1/2}  \| \widehat{z}_\tau' \|_{L^2 (0,T;L^{2}(\Omega))} \leq C$. 
\par
Thanks to \eqref{H2_improved-1}, we have that 
\[
\norm{\overline{u}_\tau}_{L^\infty(0,T;H^{2}(\Omega))} \leq  c_1 
\sup_{t\in (0,T)}  P(0,\overline{z}_\tau(t)) 
\left(\norm{\overline{\ell}_\tau}_{L^\infty(0,T;L^2(\Omega))} +
\norm{\overline{u}_{D,\tau}}_{L^\infty(0,T;H^{2}(\Omega))} \right)\leq C',
\]
where we have used estimate \eqref{est-epsi-tau-1}, as well as Assumption 
\ref{ass:load}.
Then, \eqref{est-epsi-tau-6} follows. 
\par
In order to derive estimates \eqref{est-epsi-tau-2} and \eqref{est-epsi-tau-3}, we follow the proof of \cite[Lemma 5.3]{KRZ2} and observe that, by the $1$-homogeneity of $\calR_1$, 
 \eqref{cont-reformulation} rewrites as 
\[	
\begin{cases}
\pairing{}{\calZ}{\overline{h}_\tau(\rho) }{\widehat{z}_\tau'(\rho)}  = \calR_1 (\widehat{z}_\tau'(\rho))  \quad  \text{for all } \rho \in (t_k^\tau, t_{k+1}^\tau) 
\\ 
\pairing{}{\calZ}{\overline{h}_\tau(r) }{\widehat{z}_\tau'(\rho)}  \leq \calR_1(\widehat{z}_\tau'(\rho)) \quad \text{for all } r \in [0,T]\setminus \{ t_0^\tau, \ldots, t_N^\tau\},
\\
 \end{cases} 
\]
where we have used the place-holder $\overline{h}_\tau(\rho): =  -  (\epsilon  \widehat{z}_\tau'(\rho)    + A_\il \overline{z}_\tau(\rho) +
   \rmD_z\wt \calI(\overline{t}_\tau(\rho),
 \overline{z}_\tau(\rho) ))$. 
Subtracting the second relation from the first one gives  $\tau^{-1}\pairing{}{\calZ}{\overline{h}_\tau(\rho) - \overline{h}_\tau(r)  }{\widehat{z}_\tau'(\rho)}  \geq 0$ for $\rho
\in (t_k^\tau, t_{k+1}^\tau)$ and $r \in (t_{k-1}^\tau, t_{k}^\tau)$.  Hence,  we get
\begin{multline}
\label{pruni1:e2} \epsilon \tau^{-1} \ddd{ \int_\Omega  (\widehat z_\tau'(\rho){ -}  \widehat z_\tau'(r))
\widehat z_\tau'(\rho) \dx }{$=I_1$}{}   + \ddd{ \tau^{-1}
\pairing{}{\calZ}{A_\il \overline z_\tau(\rho) -
     A_\il \overline
z_\tau(r)} {\widehat z_\tau'(\rho) }}{$=I_2$}{}
\\\leq
\ddd{-\tau^{-1} \int_\Omega (\rmD_z\wt\calI(\overline
t_\tau(\rho), \overline z_\tau(\rho)) {-}
 \rmD_z\wt\calI(\overline t_\tau(r), \overline z_\tau(r)))  \widehat
 z_\tau'(\rho) \dx }{$=I_3$}{}.
\end{multline}
Observe that 
$I_1\geq \tfrac12 \int_\Omega ( |\widehat z_\tau'(\rho)|^2 {-} | \widehat 
z_\tau'(r)|^2 ) \dd x $, whereas
it follows from estimate \eqref{e2.22} that 
\begin{align}
I_2 \geq c_q \int_\Omega \left(1 {+} |\nabla \overline{z}_\tau(\rho)|^2 {+} |\nabla \overline{z}_\tau(r)|^2  \right)^{(q{-}2)/2}  |\nabla \widehat{z}_\tau'(\rho)|^2 
\dd x \geq C_q %\frac {c_q}2
\int_\Omega \left(1  {+} |\nabla \widehat{z}_\tau(\rho)|^2  \right)^{(q{-}2)/2}  |\nabla \widehat{z}_\tau'(\rho)|^2 
\dd x
\label{ddd:e1}
\end{align}
 for some positive constant $C_q$, 
where we have used that $|\nabla \widehat{z}_\tau(\rho)|^2 \leq 2 ( |\nabla \overline{z}_\tau(\rho)|^2 {+} |\nabla \overline{z}_\tau(r)|^2 )$. 
As for $I_3$, 
by the H\"older inequality
\[
    | I_3| \ \leq C\tau^{-1} \norm{\rmD_z\wt\calI(\overline
t_\tau(\rho), \overline z_\tau(\rho)) {-}
 \rmD_z\wt\calI(\overline t_\tau(r), \overline z_\tau(r))}_{L^{2}(\Omega)} 
 \norm{\widehat z_\tau'(\rho)}_{L^{2}(\Omega)}\,.
\]
Relying on \eqref{enhanced-stim-7}, we then find
\begin{equation}
\label{I3}
\begin{aligned}
|I_3| 
\leq C  (1 +  \| \widehat z_\tau'(\rho)
\|_{L^{6}(\Omega)} ) \| \widehat z_\tau'(\rho)
\|_{L^{2}(\Omega)}\,,
\end{aligned}
\end{equation}
where we have also used that $\sup_{t\in[0,T]} C_{f''}(\overline 
z_\tau(\rho),\overline z_\tau(r)
) %+ C_{g''}(\overline z_\tau(\rho),\overline z_\tau(r)
 + P(\overline z_\tau(\rho),\overline z_\tau(r)
)  ) \leq C$  thanks to the previously proved estimate 
\eqref{est-epsi-tau-1}.  
Hence, multiplying \eqref{pruni1:e2} by $\tau$, we infer 
\begin{equation}
\label{very-useful-later}
\begin{aligned}
&
\frac\epsi2 \| \widehat z_\tau'(\rho) \|_{L^2(\Omega)}^2 +% \frac{c_q \tau} 2  
 C_q  \tau  \int_\Omega \left(1  {+} |\nabla \widehat{z}_\tau(\rho)|^2  
\right)^{(q{-}2)/2}  |\nabla \widehat{z}_\tau'(\rho)|^2 
\dd x
\\
&
 \leq 
\frac\epsi2 \| \widehat z_\tau'(r) \|_{L^2(\Omega)}^2 + \tau  C  (1 +  \| \widehat z_\tau'(\rho)
\|_{L^{6}(\Omega)} ) \| \widehat z_\tau'(\rho)
\|_{L^{2}(\Omega)}\,,
\end{aligned}
\end{equation}
which leads, upon summation, to the following estimate on the time 
interval $(t_0,t)$, with $t_0 \in (0,t_1^\tau)$ and $t\in (t_k^\tau, 
t_{k+1}^\tau)$:
\begin{equation}
\label{added-R-1}
\begin{aligned}
\frac{\epsilon}{2}\norm{\widehat
  z_\tau'(t)}^2_{L^2(\Omega)}
&+  C_q      \int_{t_1^\tau}^{\overline t_\tau(t)}
\int_\Omega\big(1 + \abs{\nabla \widehat z_\tau(\rho)}^2 \big)^{(q{-}2)/2}\abs{\nabla \widehat z'_\tau(\rho)}^2\,\dd x\,\dd\rho
\\
&\leq \frac{\epsilon}{2}\norm{\widehat
  z_\tau'(t_0)}^2_{L^2(\Omega)}
+  \frac{C_q}4  \int^{\overline t_\tau(t)}_{ t_1^\tau}
 (1 +  \| \widehat z_\tau'(\rho)
\|_{H^{1}(\Omega)}^2 ) \,\dd\rho +C   \int^{\overline t_\tau(t)}_{ t_1^\tau}
  \| \widehat z_\tau'(\rho)
\|_{L^{2}(\Omega)}^2 \,\dd\rho,
\end{aligned}
\end{equation}
where we have used Young's inequality, and the continuous embedding  
$H^1(\Omega) \subset L^6(\Omega)$, to handle the last term on the r.h.s.\ of 
\eqref{very-useful-later}.
For the first time step with $t_0\in (0,t_1^\tau)$, following  the very same 
calculations as in the proof of \cite[Lemma~5.3]{KRZ2},  we  obtain 
\begin{equation}
\label{added-R-2}
\begin{aligned}
&
\epsi \| \widehat{z}_\tau'(t_0)\|^2_{L^2(\Omega)} 
+ C_q \tau  \int_\Omega(1+\abs{\nabla\widehat z_\tau(t_0)}^2 
)^{(q{-}2)/2}\abs{\nabla\widehat z_\tau'(t_0)}^2 \dx
\\
&\quad \leq \frac{\epsilon}{2}\norm{\widehat
  z_\tau'(t_0)}_{L^2(\Omega)}^2
+ \epsilon^{-1} \norm{\rmD_z\calI(0,z_0)}_{L^{2}(\Omega)}^{2}  + 
\frac{C_q\tau}4  (1 +  \norm{\widehat 
z_\tau'(t_0)}_{H^1(\Omega)}^2) + C \tau  
  \norm{\widehat 
z_\tau'(t_0)}_{L^2(\Omega)}^2 \,.
\end{aligned}
\end{equation}
Summing \eqref{added-R-1} with \eqref{added-R-2},  and adding the term $  
\tfrac{C_q\tau}2  \int_0^{\overline t_\tau(t)}    \| \widehat z_\tau'(\rho)
\|_{L^{2}(\Omega)}^2 \,\dd\rho $ to both sides, we thus end up with the following estimate
%while for $I_1$ and $I_2$ we proceed as in \cite[Lemma~5.3]{KRZ2}.
\begin{align}
 & \frac{\epsilon}{2}\norm{\widehat z_\tau'(t)}^2_{L^2(\Omega)}+
C_q  \int_0^{\overline t_\tau(t)}    \| \widehat z_\tau'(\rho)
\|_{H^{1}(\Omega)}^2 \,\dd\rho 
\nonumber\\
&\leq  C + 
 \epsilon^{-1}
\norm{\rmD_z\calI(0,z_0)}_{L^{2}(\Omega)}^{2}
+ \frac{C_q}4   \int_0^{\overline t_\tau(t)}    \| \widehat 
z_\tau'(\rho)
\|_{H^{1}(\Omega)}^2 \,\dd\rho
+  C  \int_{0}^{\overline
t_\tau(t)} \|\widehat z'_\tau(\rho)\|_{L^{2}(\Omega)}^2 \,\dd \rho\,.
 \label{pruni1:4e}
\end{align}
Applying the discrete Gronwall Lemma  we get estimate \eqref{est-epsi-tau-2}.
\par 
In order to prove \eqref{est-epsi-tau-3}, which is uniform w.r.t.\ $\epsilon$,
we follow the proof of \cite[Lemma 5.5]{KRZ2}.
Since $\widehat z_\tau'$ is not defined in the points $t_k^\tau$, we 
write \eqref{pruni1:e2}  for $\rho=m_k$ and $\sigma = m_{k-1}$, with $m_k:=\frac{1}{2}(t_{k-1}^\tau + t_k^\tau)$,
$k\in\{2,\ldots,N\}$, and set $\widehat z_\tau'(m_0):=0$. For all $k\in \{1,\ldots,N\}$ we have 
(cf.\ \cite[Formula (5.26)]{KRZ2})
\begin{multline}
\label{e:disc_L1_1-joined}
\frac{\epsilon}{\tau} 
\int_\Omega \left( \widehat z_\tau'(m_k) {-}
\widehat z_\tau'(m_{k-1}) \right)  \widehat z_\tau'(m_k)\, \dx  +
\tau^{-1}\langle A_q\overline{z}_\tau (m_k) -
 A_q\underline{z}_\tau (m_{k}),\widehat z_\tau'(m_k)\rangle_{\calZ}
+\norm{\widehat z_\tau'(m_k)}^2_{L^2(\Omega)}
\\
\leq
-\frac{1}{\tau} \int_\Omega \left( \rmD_z\wt\calI(t_k^\tau ,
 \overline{z}_\tau (m_k)){ -}
\rmD_z\wt\calI(t_{k-1}^\tau ,   \underline{z}_\tau (m_{k})) \right)  \widehat
z_\tau'(m_k)\, \dx 
+\norm{\widehat z_\tau'(m_k)}^2_{L^2(\Omega)}
\\
+ \frac{\delta_{1,k}}{\tau}\left|
\int_\Omega  \rmD_z\calI(0,z_0) \widehat z_\tau'(m_1)\, \dx \right|,
\end{multline}
with the Kronecker symbol $\delta_{i,j}$.  
 Arguing as in the proof of \cite[Lemma 5.5]{KRZ2},
by estimate \eqref{e2.22} and the fact that
   $|\nabla\widehat{z}_\tau(m_k)|^2\leq 2  |\nabla\overline{z}_\tau(m_k)|^2 + 2  
|\nabla\overline{z}_\tau(m_{k-1})|^2$, it follows that 
   the left-hand side of \eqref{e:disc_L1_1-joined} can be bounded  from 
below  by
\begin{equation}\label{eq:lhs}
\text{L.H.S. }  \geq \frac{\epsilon}{2\tau} \norm{\widehat
z_\tau'(m_k)}_{L^2(\Omega)} \left(  \norm{\widehat
z_\tau'(m_k)}_{L^2(\Omega)} - \norm{\widehat
z_\tau'(m_{k-1})}_{L^2(\Omega)} \right)  + \mixed{k}^2,
\end{equation}
with the abbreviation
\[
 \mixed{k}^2 :=  C_q \int_{\Omega} (1+ |\nabla\widehat{z}_\tau(m_k)|^2 )^{\frac{q-2}{2}}
|\nabla\widehat z'_\tau(m_k)|^2 \,\dd x + \norm{\widehat z_\tau'(m_k)}^2_{L^2(\Omega)}
\]
  and  $C_q$ from \eqref{ddd:e1}. 
As for the first term of the right-hand side of \eqref{e:disc_L1_1-joined}, 
instead of \eqref{I3} we shall use
\begin{equation}\label{I3_bis}
\left|\frac{1}{\tau} \int_\Omega \left( \rmD_z\wt\calI(t_k^\tau ,
 \overline{z}_\tau (m_k)){ -}
\rmD_z\wt\calI(t_{k-1}^\tau ,   \overline{z}_\tau (m_{k-1})) \right)  \widehat
z_\tau'(m_k)\, \dx \right| \leq 
C(1+ \norm{\widehat z_\tau'(m_k)}_{L^4(\Omega)})\norm{\widehat z_\tau'(m_k)}_{L^4(\Omega)},
\end{equation}
 which derives from estimate \eqref{stim-l4} for $\| \rmD_z\wt\calI(t_k^\tau 
,
 \overline{z}_\tau (m_k)){ -}
\rmD_z\wt\calI(t_{k-1}^\tau ,   \overline{z}_\tau (m_{k-1})) \|_{L^{4/3}(\Omega)}$. 
 We then continue
\eqref{I3_bis} 
by  using the trivial estimate
$C (1+ \norm{\widehat z_\tau'(m_k)}_{L^4(\Omega)})\norm{\widehat z_\tau'(m_k)}_{L^4(\Omega)} \leq C \norm{\widehat z_\tau'(m_k)}_{L^4(\Omega)}^2 + C $, and then
applying the Gagliardo-Nirenberg estimate $\norm{\zeta}_{L^4(\Omega)}^2\leq c \norm{\zeta}_{L^1(\Omega)}^{2(1-\theta)}\norm{\zeta}_{H^1(\Omega)}^{2\theta}$, with $\theta=9/10$, and Young's inequality, so that 
\[
\begin{aligned}
&
\left|\frac{1}{\tau} \int_\Omega \left( \rmD_z\wt\calI(t_k^\tau ,
 \overline{z}_\tau (m_k)){ -}
\rmD_z\wt\calI(t_{k-1}^\tau ,   \overline{z}_\tau (m_{k-1})) \right)  \widehat
z_\tau'(m_k)\, \dx \right|
\\ & 
 \leq %\frac{c_q}4  
 \frac{1}{2}\min\{C_q,1\} \norm{\widehat
z_\tau'(m_k)}_{H^1(\Omega)}^2 +   C  \norm{\widehat z_\tau'(m_k)}_{L^1(\Omega)} \calR_1 (\widehat z_\tau'(m_k)) +  C\,,
\end{aligned}
\]
where we have also used that  $\norm{\widehat z_\tau'(m_k)}_{L^1(\Omega)}^2   \leq \norm{\widehat z_\tau'(m_k)}_{L^1(\Omega)} \calR_1 (\widehat z_\tau'(m_k))$. 
Therefore, 
the right-hand side of \eqref{e:disc_L1_1-joined} 
can be bounded as follows %(see the proof of \cite[Proposition~4.3]{KRZ1})
\begin{align}
\label{rhs}
\text{R.H.S.}
 \leq 
 \frac{1}{2}  \mixed{k}^2
+ C\left(1 + \norm{\widehat z_\tau'(m_k)}_{L^2(\Omega)} \calR_1(\widehat
 z_\tau'(m_k))
+ \delta_{1,k}\tau^{-1}\abs{
\langle \rmD_z\calI(0,z_0),\widehat z_\tau'(m_1)\rangle_\calZ }\right) .
\end{align}
 From \eqref{eq:lhs} and \eqref{rhs}, after some algebra  it results that  (cf.\ \cite[(5.28)]{KRZ2})
\[
\begin{aligned}
&
2  \norm{\widehat
z_\tau'(m_k)}_{L^2(\Omega)} \left(  \norm{\widehat
z_\tau'(m_k)}_{L^2(\Omega)} - \norm{\widehat
z_\tau'(m_{k-1})}_{L^2(\Omega)} \right) +\frac\tau{\epsilon}\norm{ z_\tau'(m_k)}_{L^2(\Omega)}^2 + \frac\tau{\epsilon}  \mixed{k}^2
\\ & \leq \frac{4C\tau}{\epsilon} +  \frac{4C\tau}{\epsilon}  \norm{\widehat z_\tau'(m_k)}_{L^1(\Omega)} \calR_1 (\widehat z_\tau'(m_k)) 
+ 4C \frac{\delta_{1,k}}{ \epsilon \tau} \left|
\int_\Omega  \rmD_z\calI(0,z_0) \widehat z_\tau'(m_1)\, \dx \right|\,. 
\end{aligned}
\]
At this point, 
we apply a suitable discrete version of the Gronwall Lemma (cf.\ \cite[Lemma B.1]{KRZ2}), and conclude 
following the very same lines of the proof of \cite[Lemma 5.5]{KRZ2}. Thus,  
we obtain \eqref{est-epsi-tau-3}.
\par
Finally, we use \eqref{H2_improved-2} and deduce that for almost all $t\in (0,T)$ there holds
\[
\begin{aligned} & 
\norm{\widehat{u}_\tau'(t)
}_{W^{1,3}(\Omega)} = \frac1{\tau} \norm{u_{k+1}^\tau - u_k^\tau
}_{W^{1,3}(\Omega)}
\\ &
\leq \frac{{c}_2}{\tau} P(z_\tau^k,z_\tau^{k+1})^2 \left(\tau+ \norm{z_{k+1}^\tau{-}z_k^\tau}_{L^6 (\Omega)} \right)
\left(\norm{\ell}_{\mathrm{C}^1 ([0,T];  W^{-1,3}(\Omega))}
+ \norm{u_D(t)}_{\mathrm{C}^1 ([0,T];W^{1,3}(\Omega))}
\right) \\ &  \leq C (1+ \norm{\widehat{z}_\tau'(t)
}_{L^{6}(\Omega)}),
\end{aligned}
\]
where the second inequality follows from  \eqref{est-epsi-tau-1} and Assumption \ref{ass:load}. Hence, estimates   \eqref{est-epsi-tau-2} \&  \eqref{est-epsi-tau-3} imply  \eqref{est-epsi-tau-7} \&  \eqref{est-epsi-tau-8}, respectively. 
\end{proof} 
 We postpone to Section \ref{ss:3.1} the proof of the forthcoming Lemma \ref{l:last-est}. 
  \begin{lemma}
 \label{l:last-est}
 Under Assumptions   \ref{ass:domain}, \ref{assumption:energy}, and 
\ref{ass:load}, and, in addition,  \eqref{further-reg}  on the initial 
datum $z_0$, for every $\tau \in (0,\bar{\tau}_\epsi)$    the enhanced 
regularity \eqref{enh-spat-reg}  and estimates 
\eqref{est-epsi-tau-4}--\eqref{est-epsi-tau-5} hold true,  whence 
\eqref{est-epsi-tau-4-bis}.  Furthermore, 
 the subdifferential inclusion  \eqref{Eul-Lagr-L2}  is satisfied. 
 \end{lemma}
The \underline{\textbf{proof of Proposition \ref{prop:aprio}}} now follows from combining Lemmas \ref{l:1st-3-est} \& \ref{l:last-est}. 
\QED
\par 
Let us finally give the \underline{\textbf{proof of Corollary 
\ref{cor:discr-enineq}}}: the very same calculations as in the proof of 
\cite[Lemma 6.1]{KRZ2} 
 (cf.\ also the proof of Thm.\ \ref{thm:exist} ahead),
show that the interpolants $\overline{z}_\tau,\, \widehat{z}_\tau$ fulfill at every $0\leq s \leq t \leq T$ 
\[
\begin{aligned}
&\int_{\underline{t}_{\tau}(s)}^{\overline{t}_{\tau}(t)}
\left(\calR_\epsi (\widehat{z}'_{\tau})(r)+\calR_\epsi^* (-\rmD_z
\calI (\overline{t}_{\tau}(r),\overline{z}_{\tau}(r)))  \right)
\,\mathrm{d}r +
\calI( t,\widehat{z}_{\tau}(t))
\\ &
= \calI(s,\widehat{z}_{\tau}(s)) +
\int_{\underline{t}_{\tau}(s)}^{\overline{t}_{\tau}(t)}
\partial_t \calI (r,\widehat{z}_{\tau}(r)) \, \mathrm{d}r
\\
& \qquad 
\dddshort{- \int_{\underline{t}_{\tau}(s)}^{\overline{t}_{\tau}(t)}\int_\Omega \left( A_q \overline{z}_{\tau}(r) {-} A_q \widehat{z}_{\tau}(r) \right) \widehat{z}'_{\tau}(r) \,\dd r}{$F_1$}
\dddshort{- \int_{\underline{t}_{\tau}(s)}^{\overline{t}_{\tau}(t)}\int_\Omega \left(
\rmD_z \wt{\calI} (\overline{t}_{\tau}(r),  \overline{z}_{\tau}(r)) -  \rmD_z \wt{\calI} (r, \widehat{z}_{\tau}(r))\right)
\widehat{z}'_{\tau}(r)\,\dd r}{$F_2$}\,.
\end{aligned}
\]
 Observe that the terms $F_1$ and $F_2$ feature integrals, instead of 
duality pairings between $\calZ^*$ and $\calZ$, thanks to 
\eqref{estimate-for-DI} and \eqref{enh-spat-reg}. 
By monotonicity we have $F_1\leq 0$, whereas, the very same argument leading to  \eqref{I3} yields
\[
\begin{array}{ll}
|F_2| & \leq C
\int_{\underline{t}_{\tau}(s)}^{\overline{t}_{\tau}(t)}(|(\overline{t}_{\tau}(r)-r
| +
\|\overline{z}_{\tau}(r)-\widehat{z}_{\tau}(r)\|_{L^{6}(\Omega)}
)
\|\overline{z}_{\tau}(r)-\widehat{z}_{\tau}(r)\|_{L^{2}(\Omega)}\,\dd
r
\\
&\leq C \sup_{t\in[0,T]}\|\overline{z}_{\tau}(t)-\widehat{z}_{\tau}(t)\|_{L^{2}(\Omega)}
\int_{\underline{t}_{\tau}(s)}^{\overline{t}_{\tau}(t)}(|(\overline{t}_{\tau}(r)-r | + \|\overline{z}_{\tau}(r)-\widehat{z}_{\tau}(r)\|_{L^{6}(\Omega)} )\,\dd r,
\end{array}
\]
whence \eqref{discr-enineq}.
\par
It follows from  \eqref{discr-enineq} and  \eqref{stim3} that 
\[
\begin{aligned}
& 
 \int_{0}^{\overline{t}_\tau(t)}
\left(\calR_\epsi (\widehat{z}'_{\tau}(r))+\calR_\epsi^* (-\rmD_z
\calI (\overline{t}_{\tau}(r),\overline{z}_{\tau}(r)))  \right)
\,\mathrm{d}r + 
\calI( t,\widehat{z}_{\tau}(t))
\\
&
\leq \calI( 0,z_0) +
 C + C \left( \|\overline{z}_{\tau}\|_{L^\infty (0,T; L^{2}(\Omega))} + \|\widehat{z}_{\tau}\|_{L^\infty (0,T; L^{2}(\Omega))}   \right) \left(1 + 
\int_{0}^{\overline{t}_{\tau}(t)} \|\overline{z}_{\tau}(r)-\widehat{z}_{\tau}(r)\|_{L^{6}(\Omega)} \,\dd r \right) 
\leq C\,,
\end{aligned}
\]
where the very last estimate ensues from \eqref{est-epsi-tau-1} and  \eqref{est-epsi-tau-3}.
Recalling that $\calI$ is bounded from below (cf.\ \eqref{est_coerc1}), we thus infer 
that $\sup_{t\in [0,T]} | \calI( t,\widehat{z}_{\tau}(t)) |\leq C$,  i.e.\ 
\eqref{est-epsi-tau-1-bis},   as well as \eqref{est-diss-epsi-tau}.
\QED 
%

%%%
\subsection{Proof of Lemma  \ref{l:last-est}}
\label{ss:3.1}
 Observe that, once estimate  \eqref{est-epsi-tau-4} is proved,  
 \eqref{est-epsi-tau-4-bis} then follows by observing that $\rmD_z 
\widetilde\calI (\overline{t}_\tau, \overline{z}_\tau)$ is bounded  in 
$L^\infty(0,T;L^2(\Omega))$  in view of 
 estimate \eqref{estimate-for-DI} for $\rmD_z \widetilde\calI $, 
 combined with the previously obtained \eqref{est-epsi-tau-1}. 
\par
Hence, let us now turn to the proof of   \eqref{est-epsi-tau-4}, 
which is a consequence of 
 the  \emph{Third regularity estimate}.
In order  to render it  on the time-discrete level, 
we need to work on an approximate version of  the discrete equation \eqref{cont-reformulation}, where the dissipation
metric  $\mathrm{R}_1$ inducing $\calR_1$ is replaced, for technical reasons that will be apparent in the proof 
of Lemma \ref{l:est-approx-nu} below, by a \emph{twice-differentiable} function. Observe that
the standard Yosida 
approximation of $\mathrm{R}_1$, namely the function
\begin{equation}
\label{Yos-R1}
\mathrm{R}_{1,\nu}: \R \to \R \  \text{ defined by }  \  \mathrm{R}_{1,\nu}(r): = \min_{y\in \R} \left( \frac{|y-r|^2}{2\nu} 
+  \mathrm{R}_1(y) \right)=\begin{cases}
    \frac{1}{2\nu} r^2&\text{if }r>-\nu\kappa\\
    -\kappa r -\frac{\nu\kappa^2}{2}&\text{if }r\leq -\nu\kappa
\end{cases} \, %\quad \text{for all } r \in \R,
\end{equation}
 with $\nu>0$ fixed, does not enjoy this regularity, as it is only differentiable on $\R$,
 cf.\ \cite{Brez73}. 
   \par
 We will thus resort to a regularization of $\mathrm{R}_1$ 
 devised in \cite{GilRoc2007} and 
 defined in this way.  Let $\varrho \in \rmC^\infty(\R)$ satisfy 
 $\mathrm{supp}(\varrho) \subset [-1.1]$ and $\| \varrho \|_{L^1(\R)} = 1$. 
 We then define 
\begin{equation}
\label{more-regu-Yosida}
\overline{\mathrm{R}}_{1,\nu} (r) : = \int_0^r \int_{-\nu^2}^{\nu^2} \mathrm{R}_{1,\nu}^\prime(\sigma - s) 
\varrho_\nu(s) \dd s \dd \sigma
\end{equation}
where $\varrho_\nu(s) = \nu^{-2} \eta (s/\nu^2)$. 
In \cite{GilRoc2007}  
it has been proved that 
\begin{subequations}
\label{propsR1nu}
\begin{equation}
\label{propsR1nu-1}
\overline{\mathrm{R}}_{1,\nu} \in \rmC^\infty(\R) \text{ is convex and satisfies }  -\nu |r| 
\leq  \overline{\mathrm{R}}_{1,\nu}  (r) 
\leq \mathrm{R}_1 (r)+\nu |r| \quad \text{for all } r \in \R.
\end{equation}
 Of course, for $r>0$ the latter estimate is trivially satisfied, since 
in that case, $\mathrm{R}_1(r)=\infty$. 
Inequality \eqref{propsR1nu-1} in fact derives from the estimate 
\begin{equation}
\label{propsR1nu-2}
|   \overline{\mathrm{R}}_{1,\nu}^\prime  (r) -  \mathrm{R}_{1,\nu}^\prime(r) | \leq \nu \quad \text{for all } r \in \R.
\end{equation}
Since $ \mathrm{R}_{1,\nu}^\prime$ is Lipschitz, from  \eqref{propsR1nu-2} we in 
fact deduce that $\overline{\mathrm{R}}_{1,\nu}$ grows at most  quadratically  on $\R$.
The function $\overline{\mathrm{R}}_{1,\nu}$ induces an integral functional 
\begin{equation}
\label{more-regu-Yosida-integ}
\overline{\calR}_{1,\nu}: %L^1(\Omega) 
 L^2(\Omega) \to \R \quad \text{ defined by } \quad 
\overline{\calR}_{1,\nu}(\eta): = \int_\Omega  
\overline{\mathrm{R}}_{1,\nu}(\eta(x)) \dd x \qquad \text{for all } \eta \in
  L^2(\Omega). 
\end{equation}
Observe that $\overline\calR_{1,\nu}$ is  G\^ateaux-differentiable on $L^2(\Omega)$, 
with derivative $\rmD \overline\calR_{1,\nu}(\eta)$ defined by $\rmD\overline \calR_{1,\nu}(\eta)(x): =\overline{\mathrm{R}}_{1,\nu}^\prime(\eta(x))$ for almost all $x\in \Omega$ (in fact, $\overline{\mathrm{R}}_{1,\nu} ^\prime(\eta) \in L^2(\Omega)$  by the linear growth of  $\overline{\mathrm{R}}_{1,\nu}^\prime $).  
Indeed,
 as soon as $\eta \in \calZ$, $\rmD \overline\calR_{1,\nu}(\eta)$ coincides with the G\^ateaux derivative $\rmD_{\calZ,\calZ^*}\overline \calR_{1,\nu}(\eta)$. 
For our purposes, the following closedness property relating  $\rmD \overline\calR_{1,\nu}  : L^2(\Omega) \to  L^2(\Omega)$ to the convex subdifferential $\partial \calR_1: L^2(\Omega)\rightrightarrows L^2(\Omega)$ will have a prominent role:  for any  $(t_0,t_1) \subset (0,T)$ and all sequences  $(\eta_\nu)_\nu,\, \eta, \ \xi  \in L^2(t_0,t_1;  L^2(\Omega))$  there holds
\begin{equation}
\label{propsR1nu-3}
\begin{aligned}
 & 
 \begin{cases}
 \eta_\nu \weakto \eta & \text{ as } \nu \downarrow 0 \ \  \text{ in } L^2(t_0,t_1;  L^2(\Omega)),
 \\
\rmD\overline \calR_{1,\nu}(\eta_\nu) \weakto \xi   &\text{ as } \nu \downarrow 0  \ \  \text{ in } L^2(t_0,t_1;  L^2(\Omega)),
\\
\limsup_{\nu \downarrow 0} \int_{t_0}^{t_1} \int_\Omega \rmD \overline\calR_{1,\nu}(\eta_\nu)  \eta_\nu \dd x  \dd t \leq   \int_{t_0}^{t_1} \int_\Omega \xi \eta \dd x \dd t  & 
\end{cases}
\\
& \Rightarrow \ \xi(t) \in \partial\calR_1(\eta(t)) \qquad \text{ for almost all } t \in (t_0,t_1).
\end{aligned}
\end{equation}
\end{subequations}
We refer to \cite[Prop.\ 3.1]{GilRoc2007}  for the proof of \eqref{propsR1nu-3}. 
\par
For a fixed time step $\tau>0$, given a partition  $\{0=t_0^\tau<\ldots<t_N^\tau=T\}$ of $[0,T]$, 
we now   incrementally solve the minimum problems featuring the regularized functionals $\overline\calR_{1,\nu}$. Namely, starting from 
 $\ti z0{\tau,\nu} :=z_0$, we set
\begin{align}
\label{time-increm-regul}
\ti z{k+1}{\tau,\nu} \in
\Argmin\Bset{\calI(t_{k+1}^\tau,z) +
\tau\overline\calR_{1,\nu} \left(\frac{z
    -z_k^\tau}{\tau}\right) + \frac{\epsi}{\tau}  \left\| \frac{z
    -z_k^\tau}{\tau} \right\|_{L^2(\Omega)}^2 }{z\in\calZ}, \qquad k\in \{1, \ldots, N-1\}.
\end{align}
The analogue of Prop.\  \ref{prop:exist-mini} holds. In particular, the (left- and right-continuous) piecewise constant and linear interpolants $\overline{z}_{\tau,\nu},\, \ubn z$ and $\widehat{z}_{\tau,\nu}$ of the elements 
$(\ti z{k}{\tau,\nu})_{k=0}^N$ satisfy the following approximate version of \eqref{Eul-Lagr-L2}
 \begin{equation}
 \label{Eul-Lagr-L2-approx}
\rmD \overline\calR_{1,\nu} (\widehat{z}_{\tau,\nu}'(t))+ \epsilon  \widehat{z}_{\tau,\nu}'(t)   + A_q  \overline{z}_{\tau,\nu}(t)    + \rmD_z\wt\calI(\overline{t}_\tau(t),
 \overline{z}_{\tau,\nu}(t)) =  0 \quad \text{ in }L^2(\Omega) \quad \foraa\, t \in (0,T),
 \end{equation}
where we have in fact used  that $\rmD_{\calZ,\calZ^*} 
\overline\calR_{1,\nu}(\widehat{z}_{\tau,\nu}') =  \rmD\overline \calR_{1,\nu} 
(\widehat{z}_{\tau,\nu}')$. In particular, observe that, 
 by comparison in \eqref{Eul-Lagr-L2-approx}, there holds 
\begin{equation}
\label{that's-what-we-need}
 A_q  \overline{z}_{\tau,\nu}(t) \in L^2(\Omega) 
 \qquad \text{ for almost all $t \in (0,T)$.}
 \end{equation}
\par
For the functions  $(\overline{z}_{\tau,\nu}, \widehat{z}_{\tau,\nu}, \overline{u}_{\tau,\nu}, \widehat{u}_{\tau,\nu})_{\tau,\nu}$ (with   $\overline{u}_{\tau,\nu},\, \widehat{u}_{\tau,\nu} $ the interpolants of the elements $\umin (t_\tau^k, \ti z{k}{\tau,\nu})$), 
we are now able to derive \emph{rigorously} estimates  \eqref{est-epsi-tau},  in fact \emph{uniformly} w.r.t.\ both   parameters $\tau$ and $\nu$.
\begin{lemma}
\label{l:est-approx-nu}
 Under Assumptions   \ref{ass:domain}, \ref{assumption:energy}, and 
\ref{ass:load},  and under  condition \eqref{further-reg}
 on the initial datum $z_0$, 
   estimates \eqref{est-epsi-tau} hold for the functions
$(\overline{z}_{\tau,\nu}, \widehat{z}_{\tau,\nu}, \overline{u}_{\tau,\nu}, 
\widehat{u}_{\tau,\nu})_{\tau,\nu}$ (in particular,  \eqref{est-epsi-tau-5} for 
$\bn \omega: = \rmD \overline\calR_{1,\nu} (\hnp z)$), 
 with  constants   $C,\, C(\epsi), \, C(\sigma)>0$ \emph{uniform} 
w.r.t.\ $\tau$ and $\nu$.
\end{lemma}
\begin{proof}
Estimates \eqref{est-epsi-tau-0}--\eqref{est-epsi-tau-1}  (and, consequently, 
\eqref{est-epsi-tau-6} for $\overline{u}_{\tau,\nu}$)  can be derived by the 
very same arguments as in the proof of Lemma \ref{l:1st-3-est}.
 Let us point out that we may suppose that 
$\sup_{\tau,\nu} \| \overline{z}_{\tau,\nu} \|_{L^\infty(0,T;W^{1,q}(\Omega)} 
\leq M$, with $M$ the same constant as in \eqref{bartauepsilon}.  
\par
Instead, the calculations for 
 \eqref{est-epsi-tau-2}--\eqref{est-epsi-tau-3} have to be slightly modified, as the ones developed in the proof of 
Lemma \ref{l:1st-3-est} rely on the $1$-homogeneity of $\calR_1$, whereas $\overline\calR_{1,\nu}$  no longer has this property. Therefore, we argue in this way:  keeping the short-hand notation $\hn h(t): = -(\epsilon  \widehat{z}_{\tau,\nu}'(t)   + A_q  \overline{z}_{\tau,\nu}(t)    + \rmD_z\wt\calI(\overline{t}_\tau(t),
 \overline{z}_{\tau,\nu}(t)) )$, and writing $\bn {\omega}(t)$ in place of  $\rmD \overline\calR_{1,\nu} (\widehat{z}_{\tau,\nu}'(t))$,
 \eqref{Eul-Lagr-L2-approx} rephrases as $\bn {\omega}(t)= \hn{h}(t)$. We subtract \eqref{Eul-Lagr-L2-approx}
 at time $r \in (t_{k-1}^\tau, t_k^\tau)$ from  \eqref{Eul-Lagr-L2-approx}
 at time $t \in (t_{k}^\tau, t_{k+1}^\tau)$ and test the resulting relation by $\hnp{z} (t)$. Therefore we obtain
 \begin{equation}
 \label{crucial-step}
 \overline{\calR}_{1,\nu}^*  (\bn {\omega}(t))- 
 \overline{\calR}_{1,\nu}^* (\bn {\omega}(r))  \leq \int_\Omega \left(  \bn {\omega}(t) {-} \bn {\omega}(r)\right) \hnp{z} (t) \dx = \int_\Omega \left(  \hn {h}(t) {-} \hn {h}(r)\right) \hnp{z} (t) \dx,
  \end{equation}
 where $ \overline{\calR}_{1,\nu}^* $ denotes the Fenchel-Moreau convex conjugate of $ 
 \overline\calR_{1,\nu}$, and we have used that
 \begin{equation}
 \label{conjugate-R-nu}
 \widehat{z}_{\tau,\nu}'(t) \in \partial   \overline{\calR}_{1,\nu}^* (\bn {\omega}(t)) \quad \text{for all } t 
 \in (t_{k}^\tau, t_{k+1}^\tau) \text{ and for all } k =0, \ldots, N-1.
 \end{equation}
 From \eqref{crucial-step}
 we then obtain
 the analogue of \eqref{pruni1:e2}, 
  namely
  \begin{multline}
\label{pruni1:e2-bis}     \frac{1}{\tau}\overline{\calR}_{1,\nu}^*  (\bn 
{\omega}(t))+ \frac{\epsilon}{\tau}  \int_\Omega  (\widehat z_{\tau,\nu}'(t){ -} 
 \widehat z_{\tau,\nu}'(r)) 
\widehat z_{\tau,\nu}'(t) \dx  
+\frac{1}{\tau}\int_\Omega (A_\il \overline z_{\tau,\nu}(t) {-}
     A_\il \overline
z_{\tau,\nu}(r)) \widehat z_{\tau,\nu}'(t)  \dx 
\\ \leq  \frac{1}{\tau}\overline{\calR}_{1,\nu}^*  (\bn {\omega}(r))  -
 \frac{1}{\tau}\int_\Omega (\rmD_z\wt\calI(\overline
t_\tau(t), \overline z_{\tau,\nu}(t)) {-}
 \rmD_z\wt\calI(\overline t_\tau(r), \overline z_{\tau,\nu}(r)))  \widehat
 z_{\tau,\nu}'(t) \dx. 
 \end{multline}
Observe that \eqref{pruni1:e2-bis} contains the same terms as in 
\eqref{pruni1:e2}, but with the additional contribution coming from 
$\overline{\calR}_{1,\nu}^*$.
Following the lines of the proof of  Lemma \ref{l:1st-3-est} (see also 
\cite[Lemma 5.3]{KRZ2}) we ``integrate'' over the time interval $(t_0,t)$ with 
$t_0 
\in (0,t_1^\tau)$ and $t\in (t_k^\tau, t_{k+1}^\tau)$ and get
\begin{equation}
\label{added-R-1-bis}
\begin{aligned}
&\overline{\calR}_{1,\nu}^*  (\bn {\omega}(t))+ \frac{\epsilon}{2}\norm{ \widehat
  z_{\tau,\nu}'(t) }^2_{L^2(\Omega)}
+ %\frac{c_q } 2
 C_q      \int_{t_1^\tau}^{\overline t_\tau(t)}
\int_\Omega\big(1 + \abs{\nabla \widehat z_{\tau,\nu}(\rho)}^2 \big)^{(q{-}2)/2} 
\abs{\nabla \widehat z'_{\tau,\nu}(\rho)}^2\,\dd x\,\dd\rho
\\
&\leq  \overline{\calR}_{1,\nu}^*  (\bn {\omega}(t_0)) 
+\frac{\epsilon}{2}\norm{\widehat 
  z_{\tau,\nu}'(t_0)}^2_{L^2(\Omega)}
+C  \int^{\overline t_\tau(t)}_{ t_1^\tau}
 (1 +  \| \widehat z_{\tau,\nu}'(\rho)
\|_{L^{6}(\Omega)} )  \| \widehat z_{\tau,\nu}'(\rho)
\|_{L^{2}(\Omega)}\,\dd\rho,
\end{aligned}
\end{equation}
with $C_q$ from \eqref{ddd:e1}.  
We observe that $\overline{\calR}_{1,\nu}^*  (\bn {\omega}(t))\geq 0$, 
and therefore on the left-hand side we get the exact analogue of the left-hand side of \eqref{added-R-1}. 
For the right-hand side, we have to deal with the ``extra''-term $ \overline{\calR}_{1,\nu}^* (\bn {\omega}(t_0))$. 
 For this, we observe that 
\begin{equation}
\label{at-1st-step}
\begin{aligned}
 \overline{\calR}_{1,\nu}^*  (\bn {\omega}(t_0)) =  &
 %\overline{\calR}_{1,\nu}^* (\ti \omega{1}{\tau,\nu})= 
 \overline{\calR}_{1,\nu}^*  (\bn {\omega} (t_0)) -  
 \overline{\calR}_{1,\nu}^*  (0)   \leq \int_\Omega \left( \frac{\ti z{1}{\tau,\nu}-z_0}{\tau} \right)
 %\ti \omega{1}{\tau,\nu} 
 \bn {\omega} (t_0)
 \dx
 \\ &  = \int_\Omega \frac{(\ti z{1}{\tau,\nu}-z_0)}{\tau}  \left( - \epsilon 
\frac{\ti z{1}{\tau,\nu}-z_0}{\tau} - \rmD_z \calI(t_1^\tau, \ti z{1}{\tau,\nu}) 
 \right) \dx 
 \\
 &
   =-\epsilon \norm{\widehat
  z_{\tau,\nu}'(t_0)}^2_{L^2(\Omega)}  - \int_\Omega \rmD_z
   \calI(t_1^\tau, \ti z{1}{\tau,\nu}) \widehat z_{\tau,\nu}'(t_0) \dx 
\end{aligned}
\end{equation}
and therefore, the right-hand side of \eqref{added-R-1-bis} can be bounded as follows
\begin{equation}\label{R.H.S.}
\text{R.H.S.} \leq - \int_\Omega \rmD_z \calI(t_1^\tau, \ti z{1}{\tau,\nu}) 
\widehat z_{\tau,\nu}'(t_0) \dx 
  - \frac{\epsilon}{2}\norm{\widehat
  z_{\tau,\nu}'(t_0)}^2_{L^2(\Omega)} 
  +C  \int^{\overline t_\tau(t)}_{ t_1^\tau}
 (1 +  \| \widehat z_{\tau,\nu}'(\rho)
\|_{L^{6}(\Omega)} )  \| \widehat z_{\tau,\nu}'(\rho)
\|_{L^{2}(\Omega)}\,\dd\rho .
\end{equation}
Writing  $
\rmD_z\calI( t_1^{\tau}, z_1^{\tau,\nu}) =  A_q 
(z_1^{\tau,\nu})-A_q(z_0) 
+ \rmD_z\wt\calI(t_1^{\tau},z_1^{\tau,\nu})-  
\rmD_z\wt\calI(0,z_0)
+ \rmD_z\calI(0,z_0) 
$  and performing calculations analogous to those developed in the proof of Lemma \ref{l:1st-3-est}, we obtain
\[
\begin{aligned}
 - \int_\Omega \rmD_z \calI(t_1^\tau, \ti z{1}{\tau,\nu}) \widehat z_{\tau,\nu}'(t_0) \dx
 &  \leq - C_q \tau \int_\Omega (1{+} |\nabla\widehat{z}_{\tau,\nu}(t_0)|^2 
)^{(q{-}2)/2} |\nabla  \widehat z_{\tau,\nu}'(t_0) |^2 \,\dd x + \frac{\epsilon}2   
\norm{\widehat
  z_{\tau,\nu}'(t_0)}^2_{L^2(\Omega)}
  \\
  & \quad
    + \epsilon^{-1}  \norm{\rmD_z\calI(0,z_0)}^2_{L^2(\Omega)} +c\tau (1+  \norm{\widehat
  z_{\tau,\nu}'(t_0)}^2_{L^6(\Omega)})  \norm{\widehat
  z_{\tau,\nu}'(t_0)}^2_{L^2(\Omega)}\,.
  \end{aligned}
\]
Combining this with \eqref{R.H.S.},  summing the resulting inequality with \eqref{added-R-1-bis}, 
 and adding $ C_q  \int_0^{\overline t_\tau(t)} \norm{\widehat 
z'_{\tau,\nu}(\rho)}^2_{L^2(\Omega)} \dd \rho$ to both terms of the resulting 
estimate, 
we obtain 
\begin{equation}
\label{sum}
\begin{aligned}
&\frac{\epsilon}{2}\norm{\widehat
  z_{\tau,\nu}'(t)}^2_{L^2(\Omega)}
+  %\frac{c_q } 2
 C_q  \int_0^{\overline t_\tau(t)} \norm{\widehat 
z'_{\tau,\nu}(\rho)}^2_{L^2(\Omega)}\dd \rho 
+ %\frac{c_q } 2
 C_q      \int_{0}^{\overline t_\tau(t)}
\int_\Omega\big(1 + \abs{\nabla \widehat z_{\tau,\nu}(\rho)}^2 \big)^{(q{-}2)/2}\abs{\nabla \widehat z'_{\tau,\nu}(\rho)}^2\,\dd x\,\dd\rho
\\
& \leq
\epsilon^{-1}\norm{\rmD_z\calI(0,z_0)}^2_{L^2(\Omega)}
+   C_q  \int_0^{\overline t_\tau(t)} \norm{\widehat 
z'_{\tau,\nu}(\rho)}^2_{L^2(\Omega)}\dd \rho
+C  \int_0^{\overline t_\tau(t)}
 (1 +  \| \widehat z_{\tau,\nu}'(\rho)
\|_{L^{6}(\Omega)} )  \| \widehat z_{\tau,\nu}'(\rho)
\|_{L^{2}(\Omega)}\,\dd\rho
\\
& \leq
C+ \epsilon^{-1}\norm{\rmD_z\calI(0,z_0)}^2_{L^2(\Omega)}
+  C\int_0^{\overline t_\tau(t)} \norm{\widehat z'_{\tau,.\nu}(\rho)}^2_{L^2(\Omega)}\dd \rho
 +\frac{C_q}4   \int_0^{\overline t_\tau(t)} \| \widehat 
z_{\tau,\nu}'(\rho)\|_{H^{1}(\Omega)}^2 \dd\rho,
\end{aligned}
\end{equation}
where in the last inequality we have used Young's inequality, and the continuous embedding 
$H^1(\Omega)\subset L^6(\Omega)$, for the last term in the r.h.s. of 
\eqref{R.H.S.} exactly as in the
proof of Lemma \ref{l:1st-3-est}.
 Absorbing $\int_0^{\overline t_\tau(t)} \| \widehat 
z_\tau'(\rho)\|_{H^{1}(\Omega)}^2 \dd\rho,$ into the left-hand side, we conclude 
  estimate \eqref{est-epsi-tau-2} for $\hnp z$, uniformly with respect to 
$\tau$ and $\nu$. 

Combining the arguments in the proof of Lemma \ref{l:1st-3-est} with the above arguments related to
$\overline{\calR}_{1,\nu}^*$ we also obtain estimate \eqref{est-epsi-tau-3} for $\hnp z$
 uniformly with respect to $\epsi$, $\tau$ and $\nu$,  and therefore 
also the bounds \eqref{est-epsi-tau-7}--\eqref{est-epsi-tau-8}  for $\hnp u$.
\par
We  are now in a position to carry out the time-discrete analogue of the \emph{Third regularity estimate}. We multiply \eqref{Eul-Lagr-L2-approx}, written at time $\rho \in (\ti t k \tau, \ti t {k+1}\tau)$, by the difference $(A_q \bn{z}(\rho) {- } A_q \bn{z}(r))$, with $r \in (\ti t {k-1} \tau, \ti t {k}\tau)$, and integrate in space.  Observe that this is now a legal test, in view of \eqref{that's-what-we-need}.  We thus obtain 
\begin{equation}
\label{now-legal-1}
\begin{aligned} & 
\ddd{\int_\Omega  \rmD \overline{\calR}_{1,\nu} (\widehat{z}_{\tau,\nu}'(\rho))  (A_q \bn{z}(\rho) {- } A_q \bn{z}(r)) \dd x }{$I_1$}{} + \epsilon \ddd{\int_\Omega \widehat{z}_{\tau,\nu}'(\rho)  (A_q \bn{z}(\rho) {- } A_q \bn{z}(r)) \dd x }{$I_2$}{} \\ & \quad +  \ddd{\int_\Omega A_q \bn{z}(\rho) (A_q \bn{z}(\rho) {- } A_q \bn{z}(r)) \dd x }{$I_3$}{}  = -  \ddd{\int_\Omega \rmD \wt{\calI}(\overline{t}_\tau(\rho), \bn{z}(\rho)) (A_q \bn{z}(\rho) {- } A_q \bn{z}(r)) \dd x }{$I_4$}{}\,.  
\end{aligned}
\end{equation} 
Now, we have that 
\[
\begin{aligned}
I_1 &  = \int_\Omega \nabla \left(  \overline\calR_{1,\nu}^\prime (\widehat{z}_{\tau,\nu}'(\rho))  \right) 
\cdot \left( (1+ |\nabla \bn{z}(\rho)|^2)^{q/2-1} \nabla \bn{z}(\rho) {-}(1+ |\nabla \bn{z}(r)|^2)^{q/2-1} \nabla \bn{z}(r) \right) \dx \\ & = 
 \int_\Omega 
 \overline{\calR}_{1,\nu}^{\prime \prime}(\widehat{z}_{\tau,\nu}'(\rho))   \nabla \widehat{z}_{\tau,\nu}'(\rho)
\cdot \left( (1+ |\nabla \bn{z}(\rho)|^2)^{q/2-1} \nabla \bn{z}(\rho) {-}(1+ |\nabla \bn{z}(r)|^2)^{q/2-1} \nabla \bn{z}(r) \right) \dx\stackrel{(1)}{\geq}0,
\end{aligned}
\]
where for the first equality we have used that $\rmD \overline\calR_{1,\nu} (\widehat{z}_{\tau,\nu}'(\rho))   =  \overline\calR_{1,\nu}^\prime (\widehat{z}_{\tau,\nu}'(\rho))   $ is an element in $W^{1,q}(\Omega)$:  indeed, 
$\widehat{z}_{\tau,\nu}'(\rho)\in W^{1,q}(\Omega)
\subset \rmC^0 (\overline\Omega)$, so that  there exists a constant $M>0$ with $|\widehat{z}_{\tau,\nu}'(\rho)|\leq M$ a.e.\ in $\Omega$; on the other hand $\overline\calR_{1,\nu}^\prime  \in \mathrm{C}^\infty(\R)$, hence its restriction to the ball $\overline{B}_M(0)$ is Lipschitz, and the composition of a Lipschitz function with an element in $W^{1,q}(\Omega)$ belongs to $W^{1,q}(\Omega)$. 	
Estimate (1) follows from the fact that $\overline\calR_{1,\nu}^{\prime\prime}\geq 0$ on $\R$, 
and from the convexity inequality
\[
(A-B)  
\cdot  \left( (1+ |A|^2)^{q/2-1} A {-}(1+ |B)|^2)^{q/2-1} B \right)\geq 0 \quad \text{for all } A,\, B \in \R^\sd,
\] applied with $A = \nabla \bn{z}(\rho)$ and $B = \nabla \bn{z}(r)$.
 Analogously, we have 
\[
I_2 = \int_\Omega \nabla \widehat{z}_{\tau,\nu}^\prime(\rho)
\cdot \left( (1+ |\nabla \bn{z}(\rho)|^2)^{q/2-1} \nabla \bn{z}(\rho) {-}(1+ |\nabla \bn{z}(r)|^2)^{q/2-1} \nabla \bn{z}(r) \right) \dx \geq 0. 
\]
We have 
\[
I_3 \geq \frac12 \| A_q \bn{z}(\rho)\|_{L^2(\Omega)}^2 -  \frac12 \| A_q \bn{z}(r)\|_{L^2(\Omega)}^2\,.
\]
Finally,  
\[
\begin{aligned}
I_4 =   & \int_\Omega \rmD \wt{\calI}(\overline{t}_\tau(\rho), \bn{z}(\rho)) A_q   \bn{z}(\rho)  \dx 
- \int_\Omega \rmD \wt{\calI}(\overline{t}_\tau(r), \bn{z}(r)) A_q   \bn{z}(r)  \dx
\\ &   - \int_\Omega \left( \rmD \wt{\calI}(\overline{t}_\tau(\rho), \bn{z}(\rho))
{-} \rmD \wt{\calI}(\overline{t}_\tau(r), \bn{z}(r))  \right) A_q \bn{z}(r) \dd x\,.
\end{aligned}
\] 
Summing with respect to the index $k$, we thus obtain for any $t \in (\ti t1{\tau}, T)$ and for $\sigma 
\in  (0,\ti t1{\tau})$ (remember that $\bn{z}(r)=\ubn{z}(\rho)$
 and $\overline{t}_\tau(r) = \underline{t}_\tau(\rho)$ 
 for 
$r\in (t_{k-1}^\tau, t_{k}^\tau]$ and $\rho\in [t_k^\tau,t_{k+1}^\tau)$) 
\[
\begin{aligned}
\frac12 \| A_q \bn{z}(t)\|_{L^2(\Omega)}^2  \leq  &  \frac12 \| A_q \bn{z}(
\sigma)\|_{L^2(\Omega)}^2 
+  \int_\Omega \rmD \wt{\calI}(\overline{t}_\tau(\sigma), \bn{z}(\sigma)) A_q   \bn{z}(\sigma)  \dx 
 - \int_\Omega \rmD \wt{\calI}(\overline{t}_\tau(t), \bn{z}(t)) A_q   \bn{z}(t)  \dx 
 \\ & + 
 \int_{t_1^\tau}^{\overline{t}_\tau(t)}  \int_\Omega \frac1{\tau}  \left( \rmD \wt{\calI}(\overline{t}_\tau(\rho), \bn{z}(\rho))
 {-} \rmD \wt{\calI}(\underline{t}_\tau(\rho), \ubn{z}(\rho))  \right) A_q \ubn{z}(\rho) \dd x \dd \rho\, \doteq I_5+I_6+I_7+I_8.
 \end{aligned}
\]
We estimate via H\"older's and Young's inequalities
\[
\begin{aligned}
& 
\left| I_6 \right| \leq   \|   \rmD \wt{\calI}(\overline{t}_\tau(\sigma), \bn{z}(\sigma))  \|_{L^2(\Omega)}^2 + \frac14  \| A_q \bn{z}(\sigma)\|_{L^2(\Omega)}^2\stackrel{(2)}{\leq} C + \frac14  \| A_q \bn{z}(\sigma)\|_{L^2(\Omega)}^2,
 \\
&
\left| I_7 \right| \leq   \|   \rmD \wt{\calI}(\overline{t}_\tau(t), \bn{z}(t))  \|_{L^2(\Omega)}^2 + \frac14  \| A_q \bn{z}(t)\|_{L^2(\Omega)}^2\stackrel{(1)}{\leq} C + \frac14  \| A_q \bn{z}(t)\|_{L^2(\Omega)}^2,
\\
 &
  \left| I_8 \right| \leq   \int_0^{\overline{t}_\tau(t)} \frac1{\tau} \| \rmD \wt{\calI}(\overline{t}_\tau(\rho), \bn{z}(\rho))
 {-} \rmD \wt{\calI}(\underline{t}_\tau(\rho), \ubn{z}(\rho)) \|_{L^2(\Omega)} \|A_q \ubn{z}(\rho) \|_{L^2(\Omega)} \dd \rho 
 \\ & \qquad 
  \stackrel{(3)}{\leq}   C    \int_0^{\overline{t}_\tau(t)} \frac1{\tau} \| \bn{z}(\rho) {-} \ubn{z}(\rho)\|_{L^6(\Omega)}  \|A_q \ubn{z}(\rho) \|_{L^2(\Omega)} \dd \rho.
\end{aligned}
\]
where (1) and (2) follow from \eqref{estimate-for-DI} and from the bound $\| f'(\bn z(t)\|_{L^\infty(\Omega)} + P(\bn z(t), 0) \leq C$, for a constant uniform w.r.t.\ $t \in [0,T]$, thanks to  estimate \eqref{est-epsi-tau-1} for $(\bn z)_{\tau,\nu}$; instead, (3) is due to \eqref{enhanced-stim-7}, again taking into account that $
\sup_{\rho  \in [0,T]}( C_{f''}(\bn{z}(\rho), \ubn{z}(\rho)) + 
P(\bn{z}(\rho), \ubn{z}(\rho))^3)  \leq C$ due to the bound 
\eqref{est-epsi-tau-1}.  All in all, we conclude 
\[
\frac14 \| A_q \bn{z}(t)\|_{L^2(\Omega)}^2  \leq   \frac34 \| A_q \bn{z}(
\sigma)\|_{L^2(\Omega)}^2 +  C  \left(1+  \int_0^{\overline{t}_\tau(t)}  \| \hnp{z}(\rho) \|_{L^6(\Omega)}  \|A_q \ubn{z}(\rho) \|_{L^2(\Omega)} \dd \rho \right),
\]
and, with a version of  Gronwall's Lemma  (cf.\ e.g.\ \cite[Lemme A.5]{Brez73}),  we conclude that 
\begin{equation}
\label{conse-Gronw}
 \| A_q \bn{z}(t)\|_{L^2(\Omega)} \leq  C \left(1+ \| A_q \bn{z}(
\sigma)\|_{L^2(\Omega)} +  \int_0^{\overline{t}_\tau(t)}  \| \hnp{z}(\rho) \|_{L^6(\Omega)}\dd \rho \right).
\end{equation}
It now remains to estimate $ \| A_q \bn{z}(
\sigma)\|_{L^2(\Omega)}  =  \| A_q  \ti z1{\tau,\nu} \|_{L^2(\Omega)} $. For this, we use the Euler-Lagrange equation 
\[
\rmD \overline\calR_{1,\nu} \left(\frac{ \ti z1{\tau,\nu} - z_0 }\tau\right) + \epsilon  \frac{ \ti z1{\tau,\nu} - z_0 }\tau  + A_q  \ti z1{\tau,\nu}    + \rmD_z\wt\calI( \ti t1{\tau,\nu}, \ti z1{\tau,\nu}) =  0 
\]
and test it by $A_q  \ti z1{\tau,\nu} - A_q z_0$. We repeat the same calculations as above and arrive at 
\[
\begin{aligned}
\frac12 \| A_q  \ti z1{\tau,\nu} \|_{L^2(\Omega)}^2  \leq  &  \frac12 \| A_q z_0\|_{L^2(\Omega)}^2 + 
 \int_\Omega \rmD \wt{\calI}(0,z_0) A_q   z_0 \dx - \int_\Omega \rmD \wt{\calI}(\ti t1{\tau,\nu}, \ti z1{\tau,\nu}) A_q   
 \ti z1{\tau,\nu}  \dx 
 \\ & + 
\int_\Omega  \left(  \rmD \wt{\calI}(\ti t1{\tau,\nu}, \ti z1{\tau,\nu}){-} \rmD \wt{\calI}(0,z_0)  \right) A_qz_0 \dd x,
\end{aligned}
\]
whence 
\[
\| A_q  \ti z1{\tau,\nu} \|_{L^2(\Omega)}^2 \leq C \left(1+  \| A_q z_0\|_{L^2(\Omega)}^2 + \|  \ti z1{\tau,\nu} -z_0\|_{L^6(\Omega)}^2 \right) \leq C,
\]
the last inequality due to
\eqref{further-reg} and bound \eqref{est-epsi-tau-1}.  Combining the 
above estimate  with  \eqref{conse-Gronw}, we conclude estimate 
\eqref{est-epsi-tau-4} in view of  the previously proved bound 
\eqref{est-epsi-tau-3} for $\hnp z$.  
 \par
 Finally, estimate  \eqref{est-epsi-tau-5} 
 for $ \bn \omega = \rmD \overline\calR_{1,\nu} (\hnp z)$ 
  follows from a comparison argument in \eqref{Eul-Lagr-L2-approx}, in view of 
estimate
 \eqref{est-epsi-tau-2}, 
  %  \eqref{est-epsi-tau-3},
   and the previously used bound for $  \rmD \wt{\calI}(\overline{t}_\tau(\cdot), \bn{z}(\cdot)) $ in $L^\infty (0,T;L^2(\Omega))$ due to \eqref{estimate-for-DI} and  \eqref{est-epsi-tau-1}. 
\end{proof}
\par
We are now in a position to conclude the 
\underline{\textbf{proof of Lemma \ref{l:last-est}}}: 
For \emph{fixed} positive $\tau$ and $\epsilon$,  let $(\bn z, \hn z)_{\nu}$  a 
family of solutions to \eqref{Eul-Lagr-L2-approx}. It follows from estimates \eqref{est-epsi-tau} proved in
Lemma \ref{l:est-approx-nu} and from Proposition 
\ref{prop:Savare98}, that the sequence $(\bn z)_\nu$ is also 
uniformly bounded in $L^\infty (0,T; W^{1+\sigma, q}(\Omega))$ for all $0<\sigma <\frac1q$, whence estimate \eqref{est-epsi-tau-1reg}. 
Hence, also the 
sequence $(\hn z)_\nu$  is bounded in that space. Therefore, applying the Aubin-Lions type compactness results from 
\cite{simon87}  to  $(\hn z)_\nu$,  we infer that there exists a function $\widehat z$ such that, 
along a (not relabeled) subsequence, as $\nu\downarrow 0$ the following convergences hold
\begin{subequations}
\label{convs-bn}
\begin{equation}
\label{conv-bn-hat}
\begin{aligned}
&
\hn z \weaksto \widehat z && \text{ in } L^\infty (0,T; W^{1+\sigma, q}(\Omega)) \cap H^1(0,T;H^1(\Omega)) \cap W^{1,\infty}(0,T; L^2(\Omega)) \quad \text{for all } 0 <\sigma <\frac1q,
\\
&
\hn z \to \widehat z && \text{ in } \rmC^0 ([0,T]; \calZ), 
\end{aligned}
\end{equation}
where the last convergence follows from the compact embedding $W^{1+\sigma, q}(\Omega) \Subset  \calZ$ 
 for all $\sigma \in (0,\tfrac1q)$. 
 From the estimate for $(\hnp z )_\nu$ in  $L^1(0,T;H^1(\Omega))$
we gather that
\[
\| \bn z\|_{\mathrm{BV}([0,T];H^1(\Omega))} \leq C
\]
for a constant independent of $\nu$ (and $\tau$). Therefore, thanks to  an infinite-dimensional version of Helly's Theorem, see e.g.\ \cite[Thm.\ 6.1]{MieThe04RIHM}, we conclude that there exists $\overline z \in \mathrm{BV}([0,T];H^1(\Omega))$ such that,   up to the further extraction of a subsequence, 
$\bn z(t)\weakto \overline{z}(t)$ in $H^1(\Omega)$, as $\nu \downarrow 0$ for every $t\in [0,T]$. Since  $(\bn z)_\nu$ is bounded in $L^\infty (0,T;W^{1+\sigma, q}(\Omega) )$, 
we ultimately conclude that  $\bn z(t) \weakto \overline{z}(t) $ in $W^{1+\sigma, q}(\Omega)$ for every $t\in [0,T]$.
Thus, we infer
\begin{equation}
\label{conv-bn-bar}
\bn z (t) \to \overline{z}(t) \qquad \text{ in } \calZ \quad \text{for every } t\in [0,T]. 
\end{equation}
Then, a fortiori one has that 
\begin{equation}
\label{conv-bn-bar-1}
\bn z \weaksto \overline z \text{ in } L^\infty(0,T; \calZ), \qquad 
\bn z \to \overline z \text{ in } L^p(0,T;  \calZ)  \text{ for every } 1 \leq p <\infty. 
\end{equation}
Finally, there exists $\overline\omega \in L^\infty (0,T; L^2(\Omega))$ such that, up to a further extraction, 
\begin{equation}
\label{conv-bar-omega}
\bn \omega \weaksto \overline\omega \qquad \text{ in } L^\infty (0,T; L^2(\Omega)).
\end{equation}
\end{subequations}
\par
It follows from \eqref{conv-bn-bar}, combined with the bound \eqref{est-epsi-tau-3},  that 
\[
A_q \bn z(t) \weakto A_q \overline{z}(t) \qquad \text{ in } L^2(\Omega) \quad \text{for every } t \in [0,T].
\]
Also in view of \eqref{conv-bn-bar-1} it is not difficult to deduce that 
\[
A_q \bn z \weaksto A_q \overline z \qquad \text{ in } L^\infty (0,T;L^2(\Omega)). 
\]
Furthermore, combining estimate \eqref{enhanced-stim-7} with \eqref{est-epsi-tau-1} and convergence \eqref{conv-bn-bar} we find that  for every $t\in [0,T]$
\[
\begin{aligned}
 \| \rmD \wt{\calI}(\overline{t}_\tau(t), \bn{z}(t)){-} \rmD \wt{\calI}(\overline{t}_\tau(t), \overline z(t)) \|_{L^2(\Omega)}  &  \leq  C\left(C_f' (\bn{z}(t), \overline z(t)) + P(\bn{z}(t), \overline z(t) )^3 \right)\| \bn{z}(t) - \overline z(t)\|_{L^6(\Omega)} \\ & \leq C  \| \bn{z}(t) - \overline z(t)\|_{L^6(\Omega)} \to 0 
 \end{aligned}
\]
as $\nu\downarrow 0$. Since  $( \rmD \wt{\calI}(\overline{t}_\tau, \bn{z}))_\nu$ is bounded in $L^\infty (0,T;L^2(\Omega))$ by \eqref{estimate-for-DI} and \eqref{est-epsi-tau-1}, we also have 
\[
\begin{aligned}
& 
\rmD \wt{\calI}(\overline{t}_\tau, \bn{z}) \weaksto \rmD \wt{\calI}(\overline{t}_\tau, \overline z ) \text{ in } L^\infty(0,T; L^2(\Omega)), 
\\
& 
\rmD \wt{\calI}(\overline{t}_\tau, \bn{z}) \to \rmD \wt{\calI}(\overline{t}_\tau, \overline z ) \text{ in } L^p(0,T;  L^2(\Omega))  \quad \text{for every } 1 \leq p <\infty. 
\end{aligned}
\]
Therefore, also on account of convergences \eqref{conv-bn-hat} and \eqref{conv-bar-omega} we can pass to the limit as $\nu \downarrow 0$ in 
\eqref{Eul-Lagr-L2-approx} and conclude that the triple $(\overline z, \widehat{z}, \overline\omega)$ satisfies 
\[
\overline\omega(t) + \epsi \widehat{z}'(t) + A_q \overline z(t) + \rmD \wt{\calI}(\overline{t}_\tau(t), \overline z(t)) =0 \quad \text{ in } L^2(\Omega) \quad \foraa\, t \in (\ti t k \tau, \ti t {k+1} \tau)
\]
and for every $k \in \{0,\ldots, N-1\}$. 
We can also prove that 
\[
\limsup_{\nu \downarrow 0} \int_{\ti tk\tau}^{\ti t{k+1}\tau} \int_\Omega \bn \omega \hnp z \dd x \dd t \leq \int_{\ti tk\tau}^{\ti t{k+1}\tau} \int_\Omega \overline \omega \widehat{z}' \dd x \dd t\,.
\]
This follows from multiplying  \eqref{Eul-Lagr-L2-approx} by $\hnp z$ and taking the limit in each of the terms, on account of the convergences so far proved. 
\par
Therefore, thanks to \eqref{propsR1nu-3}, we infer that 
$\overline \omega(t) \in \partial\calR_1 (\widehat{z}'(t))$ for almost all $ t \in  (\ti t k \tau, \ti t {k+1} \tau)$. All in all, the pair  $(\overline z, \widehat z)$ fulfills the differential inclusion
\begin{equation}
\label{almost-eq}
 \partial\calR_1 (\widehat{z}'(t)) + \epsi \widehat{z}'(t) + A_q \overline z(t) + \rmD \wt{\calI}(\overline{t}_\tau(t), \overline z(t))  \ni 0  \text{ in } L^2(\Omega) \quad \foraa\, t \in (\ti t k \tau, \ti t {k+1} \tau) \  \forall\, k \in \{0,\ldots, N-1\}\,.
\end{equation}
A fortiori, since $ \partial\calR_1 (\widehat{z}'(t)) \subset  \partial_{\calZ, \calZ^*}\calR_1 (\widehat{z}'(t))$, we conclude that $(\overline z, \widehat z)$ fulfill 
\[
\partial_{\calZ, \calZ^*}\calR_1 (\widehat{z}'(t))+ \epsi \widehat{z}'(t) + A_q 
\overline z(t) + \rmD \wt{\calI}(\overline{t}_\tau(t), \overline z(t))  \ni 0 
\quad \text{ in }
 \calZ^*  \quad \foraa\, t \in (\ti t k \tau, \ti t {k+1} \tau) \quad 
\forall\, k \in \{0,\ldots, N-1\}\,.
\]
Since  the latter has a unique solution  in the closed ball 
$\overline{B}_M(0)$ of $\calZ$ for $\tau <\bar{\tau}_\epsi$ (cf.\ Prop.\ 
\ref{prop:exist-mini}), and since $\overline{z}$ and $\overline{z}_\tau$ take 
value in that ball, 
we get that 
\[
\overline z(t) = \overline{z}_\tau(t), \qquad \widehat{z}'(t) = \widehat{z}_\tau'(t) \qquad \foraa\, t \in (\ti t k \tau, \ti t {k+1} \tau) \quad \forall\, k \in \{0,\ldots, N-1\}\,,
\]
and, therefore, a.e.\ in $(0,T)$.  In particular, we find that $A_q\overline{z}_\tau \in L^\infty (0,T; L^2(\Omega))$.  Furthermore, since estimates \eqref{est-epsi-tau-4} and  \eqref{est-epsi-tau-5}   are uniform both w.r.t.\ $\nu>0$ and w.r.t.\ $\tau>0$, they are   inherited in the limit as $
\nu \downarrow 0$.  Therefore, 
\[
\| A_q\overline{z}_\tau \|_{ L^\infty (0,T; L^2(\Omega))} +
%\epsi^{1/2}
\|\overline \omega \|_{ L^\infty (0,T; L^2(\Omega))} \leq C
\]
for a constant independent of  $\tau<\bar{\tau}_\epsi $.  We set 
$\overline \omega_\tau: =  \overline \omega$ and ultimately conclude 
\eqref{enh-spat-reg} as well as 
\eqref{est-epsi-tau-4} and  \eqref{est-epsi-tau-5}. 
 Finally, from \eqref{almost-eq} we gather the validity of  
 \eqref{Eul-Lagr-L2}. This concludes the proof of Lemma \ref{l:last-est}. 
\QED
%%%%%%%
%%%%%%%%
%%%%%%%
 \section{\bf Existence of viscous solutions}
\label{s:4}
   In this section, we briefly comment on the existence of solutions to the 
viscous system
  \eqref{dndia-eps}.
   By passing to the limit with $\epsi>0$ fixed in the time discrete scheme \eqref{cont-reformulation},  we are able to prove 
   the existence of a solution to   \eqref{dndia-eps}, formulated as a subdifferential inclusion in $L^2(\Omega)$, namely
   \begin{equation}
   \label{subdiff-incl-L2}
   \omega(t) +\epsi z'(t) +A_q (z(t)) + \rmD_z\wt\calI(t,z(t)) \ni 0 \quad\text{in }L^2(\Omega) \ \foraa\, t \in (0,T),
   \end{equation}
   with $\omega(\cdot) $ a selection in the subdifferential $\partial\calR_1(z'(\cdot)) \subset L^2(\Omega)$. Furthermore,
   along the footsteps of \cite{mrs2013}
   we obtain an energy-dissipation balance featuring the conjugate $\calR_\epsi^*$
    of $\calR_\epsi$, cf.\
     \eqref{Fenchel-Moreau}.
  \begin{theorem}
  \label{thm:exist}
  Let $\epsi>0$ be fixed.
   Under Assumptions   \ref{ass:domain}, \ref{assumption:energy}, and 
\ref{ass:load},  and under   condition \eqref{further-reg}
 on the initial datum $z_0$, 
  there exist  
  \begin{equation}
  \label{z-sol-curve}
  \begin{gathered}
  z \in L^\infty(0,T;W^{1+\sigma,q}(\Omega)) \cap H^1(0,T;H^1(\Omega)) \cap 
W^{1,\infty}(0,T;L^2(\Omega))
\text{ for every
  $\sigma \in (0,\tfrac1q)$, with }
  \\
   A_q z \in L^\infty (0,T;L^2(\Omega))
   \end{gathered}
  \end{equation}
   and $\omega\in L^\infty (0,T;L^2(\Omega))$ fulfilling 
   the subdifferential inclusion
   \eqref{subdiff-incl-L2} and the Cauchy condition $z(0)=z_0$.
   \par
   Furthermore, $z$ complies with the energy-dissipation balance
   \begin{equation}
   \label{en-diss-bal}
   \int_s^t \calR_\epsi(z'(r)) \dd r + \int_s^t \calR_\epsi^*({-}A_q(z(r)) 
{-}\rmD_z \wt\calI(r,z(r))) \dd r + \calI(t,z(t)) =  \calI(s,z(s)) +\int_s^t 
\partial_t \calI(r,z(r)) \dd r 
   \end{equation}
   for every $0\leq s \leq t \leq T$. 
  \end{theorem}
  \par
\begin{proof} 
Let $(\tau_j)_j$ be a null 
  sequence of time steps, and let $(\overline{z}_{\tau_j})_j,\, 
  (\widehat{z}_{\tau_j})_j $ be the approximate solutions to the viscous subdifferential inclusion \eqref{dndia-eps} constructed in Section \ref{s:3}.
  For them, estimates \eqref{est-epsi-tau} hold with a constant uniform w.r.t.\ $j\in \N$ (recall that   $\epsi>0$ is fixed). 
  \par
  Adapting the arguments from the proof of \cite[Prop.\ 6.2]{KRZ2}, combining 
\eqref{est-epsi-tau} 
with Aubin-Lions type compactness results (cf., e.g., \cite[Thm.\ 5, Cor.\ 
4]{simon87}) and arguing in the same way as in the proof of Lemma 
\ref{l:last-est},
  cf.\ also Lemma 
  \ref{l:compactn} ahead,  we may show that there exist a (not relabeled) subsequence and  a curve $z$ as in  \eqref{z-sol-curve} %  and $\omega \in L^\infty(0,T;L^2(\Omega
  such that the following convergences hold
  \[
  \begin{aligned}
  &
  \overline{z}_{\tau_j},\, \widehat{z}_{\tau_j} \to z  && \text{in}  && L^\infty (0,T;\calZ), % \text{ for all } 0 <\sigma<\frac1q,
  \\
  &
   \widehat{z}_{\tau_j} \weaksto z  && \text{in}  && H^1(0,T;H^1(\Omega)) \cap W^{1,\infty}(0,T;L^2(\Omega)),
   &
\\
& \calI(\overline{t}_{\tau_j}(t),  \overline{z}_{\tau_j}(t)),\, \calI(t, \widehat{z}_{\tau_j}(t)) \to \calI(t,z(t)) &&  &&\qquad \qquad \qquad \qquad  \qquad \text{for all } t \in [0,T],
\\
&
\rmD_z\calI(\overline{t}_{\tau_j}(t),  \overline{z}_{\tau_j}(t)) \weaksto \rmD_z \calI(t,z(t))  && \text{in} && L^\infty(0,T;L^2(\Omega)),
\\
&
\rmD_z\calI(\overline{t}_{\tau_j}(t),  \overline{z}_{\tau_j}(t)) \to \rmD_z \calI(t,z(t))  && \text{in} && L^\infty(0,T;\calZ^*).
  \end{aligned}
  \]
   \par
     With the  limit passage arguments from \cite[Thm.\ 3.5]{KRZ2} we deduce that $z$ complies with 
   the variational inequality
\begin{equation}
\label{weak-def-sol}
\begin{aligned}
\calR_\epsi(w) - \calR_\epsi (z'(t)) \geq  &  \pairing{}{\calZ}{-A_{\il} z(t) }{w}  +  \int_\Omega (1+ |\nabla z(t)|^2)^{\frac{q-2}{2}} \nabla z(t) \cdot \nabla z'(t)\, \dd x
\\ & \quad
-  \int_{\Omega}\rmD_z \widetilde{\calI}(t,z(t)) (w-z'(t))\, \dd x
 \quad \text{for all } w \in \calZ \qquad \foraa\, t \in (0,T)\,,
\end{aligned}
\end{equation}
which in fact defined the concept of \emph{weak solution} to the viscous system considered in 
\cite{KRZ2}. 
\par
We now enhance \eqref{weak-def-sol} by relying on the information that $A_q z \in L^\infty(0,T; L^2(\Omega))$. Due to this,
$ \int_\Omega (1+ |\nabla z(t)|^2)^{\frac{q-2}{2}} \nabla z(t) \cdot \nabla z'(t) \dd x = \int_\Omega A_q(z(t)) z'(t) \dd x $, so that 
\eqref{weak-def-sol} reads for almost all $t\in (0,T)$
\[
\begin{aligned}
\calR_\epsi(w) - \calR_\epsi (z'(t)) \geq  &  - \int_\Omega  A_{\il} z(t) (w-z'(t)) \dd x    
-  \int_{\Omega}\rmD_z \widetilde{\calI}(t,z(t)) (w-z'(t))\, \dd x
 \quad \text{for all } w \in \calZ\,.
\end{aligned}
\]
This extends to all $w \in L^2(\Omega)$ by  a density argument, and therefore we conclude that 
\begin{equation}
\label{incl-Fenchel}
-  A_{\il} z(t)  -  \rmD_z \widetilde{\calI}(t,z(t)) \in \partial  \calR_\epsi (z'(t)) \qquad \text{ in } L^2(\Omega)
\end{equation}
for almost all $t\in (0,T)$, namely the validity of \eqref{subdiff-incl-L2}. 
\par
The energy-dissipation balance \eqref{en-diss-bal} ensues from integrating on the generic interval $(s,t)\subset(0,T)$ the following chain of identities
\[
\begin{aligned}
\calR_\epsi(z'(r)) + \calR_\epsi^* ({-}   A_{\il} z(r)  {-}  \rmD_z \widetilde{\calI}(r,z(r)) ) & \stackrel{(1)}{=} \int_\Omega \left({-}   A_{\il} z(r)  {-}  \rmD_z \widetilde{\calI}(r,z(r)) \right) z'(t) \dd x 
\\
 & \stackrel{(2)}{=} -\frac{\dd}{\dd t} \calI(r,z(r)) + \partial_t \calI(r,z(r)) \quad \foraa\, r \in (0,T),
 \end{aligned}
\]
where (1) is a reformulation of 
 \eqref{incl-Fenchel}, while (2) follows from the chain rule \eqref{ch-rule-identity}. 
 % \QED
\end{proof} 
 %%%%%
 \section{\bf Balanced Viscosity solutions to the rate-independent damage system}
 \label{s:5}
  The main result of this section, 
 Theorem  \ref{thm:van-visc-eps-tau} ahead,
 states the convergence  %in the vanishing-viscosity limit,
of the sequences
\begin{equation}
\label{interp-tau-eps}
(\pwc z{\tau}{\epsilon})_{\tau,\epsilon}, \ (\pwl z{\tau}{\epsilon})_{\tau,\epsilon}
\end{equation}
of discrete solutions 
constructed in Section \ref{s:3} to a 
Balanced Viscosity solution of the rate-independent damage system \eqref{dndia}, 
 as  $\epsilon$ and $\tau$ 
\emph{simultaneously} 
tend to zero (that is why, we  stress the dependence  on the parameter $\epsilon$ in the notation \eqref{interp-tau-eps}). The proof of Thm.\ \ref{thm:van-visc-eps-tau} will be 
carried out in Section \ref{ss:5.3}.  
 \par
In Section~\ref{ss:5.1}   we  provide a precise definition of this solution 
concept, 
after revisiting, and suitably modifying, all the preliminary definitions and 
notions
given in \cite[Sec.\ 3.1]{MRS16}. 
Indeed, the latter paper addressed  the case of a \emph{nonsmooth}  
energy functional driving the (abstract) gradient system under 
consideration, and 
 developed the vanishing-viscosity analysis under the sole 
 \emph{basic energy} estimates
for viscous solutions. 
 In the present context, on the one hand we will  work with simpler definitions, tailored to the smoothness properties of $\calI$, and to the enhanced estimates holding for our own damage system. On the other hand, our definitions shall reflect the fact that the dissipation potential 
$\calR_1$ takes the value $+\infty$, whereas the analysis in \cite{MRS16} is 
confined to the case of a \emph{continuous} potential $\calR_1$.
  \par
 In  Sec.\ \ref{ss:5.2} 
 we gain further insight into
  to the properties of  Balanced Viscosity solutions
  and again revisit  and adapt a series of results given in \cite[Secs.\ 3.2, 3.3, 
3.4]{MRS16}. 
%  
%%%%
\subsection{The notion of Balanced Viscosity solution} 
\label{ss:5.1}
 In order to define the notion of Balanced Viscosity solution for the damage 
system \eqref{dndia}, 
  we start by introducing the \emph{vanishing-viscosity contact potential}
  $\mathfrak{p}$
   induced by the viscous dissipation potentials $\calR_\epsilon$
  from \eqref{Reps-intro}. Such functional
  will enter into the Finsler cost  describing the energy dissipated at jumps. 
  We define $\mathfrak{p}: L^2(\Omega) \times  L^2(\Omega) \to [0,+\infty]$ via
 \[
 \begin{aligned}
\mathfrak{p}(v,\xi) &  := 
\inf_{\epsi>0} \left( \calR_\epsi(v) + \calR_\epsi^*(\xi)\right)
\\ &
 =
\calR_1(v) + \|v\|_{L^2(\Omega)} \inf_{z\in \partial\calR_1(0) } \| \xi-z\|_{L^2(\Omega)}\,.
\end{aligned}
\]
From this, one defines the \emph{dissipation functional} 
$\mathfrak{f} : [0,T] \times \calZ \times L^2(\Omega) \to [0,+\infty] $  via  
\[
\mathfrak{f}_t(z,v): = \mathfrak{p}(v,-\rmD_z \calI(t,z)) = \calR_1(v) + 
\|v\|_{L^2(\Omega)} \min_{\zeta\in \partial\calR_1(0) } \|- \rmD_z 
\calI(t,z)-\zeta\|_{L^2(\Omega)}\,, 
\]
where $v$ plays the role of $z'$. 
Observe that for all $z\in \calZ, v\in L^2(\Omega)$ we have 
\[
 \mathfrak{f}_t(z,v)\geq \langle -\rmD_z\calI(t,z),v\rangle_{L^2(\Omega)}
\]
provided that $\rmD_z\calI(t,z)\in L^2(\Omega)$.  
We are now in a position to define 
the Finsler cost associated with $\mathfrak{f}$, obtained by minimizing 
suitable integral quantities
along \emph{admissible curves}.  Let us mention in advance that our 
definition of the class of admissible curves reflects the enhanced 
estimates 
available  in the present setting  for the discrete viscous solutions, cf.\ 
Remark \ref{rmk:comp-MRS13} below for more details. 
\begin{definition}
\label{def:curves+cost} Let $t\in [0,T]$ and $z_0,\, z_1 \in \calZ$  be fixed. 
\begin{enumerate}
\item
We call a curve $\teta: [r_0,r_1]\to \calZ$, for some $r_0<r_1$,  an   
\emph{admissible transition curve} between $z_0$ and $z_1$,
at the  time $t\in [0,T]$,   if 
\begin{enumerate}
\item $\teta \in  L^\infty(r_0,r_1;\calZ) \cap \mathrm{AC}([r_0,r_1];L^2(\Omega))$;
\item $\rmD_z \calI(t,\teta(\cdot)) \in L^\infty(r_0,r_1; L^2(\Omega))$.
\end{enumerate}
We denote by  $\calT_t(z_0,z_1)$ the set of admissible curves connecting  $z_0$ and $z_1$.
\item The
  (possibly asymmetric) Finsler cost induced by $\mathfrak{f}_t$ at the time
  $t$ is given by
  \begin{align}
    \label{eq:69}
%    \begin{aligned}
    \Cost{\mathfrak{f}}t{z_0}{z_1}:=\inf_{ \vartheta\in \calT_{t}(z_0,z_1)}
    \int_{r_0}^{r_1}
\mathfrak{f}_t(\teta(r),\teta'(r))\dd r
      \end{align}
  with the usual convention
  of setting $\Cost{\mathfrak f}t{u_0}{u_1}
  %=\cost{\vvmname}t{u_0}{u_1}
  =+\infty$ if the set   $\calT_t(z_0,z_1)$  of admissible curves connecting $z_0$ and $z_1$  
 is empty.
 \end{enumerate}
 \end{definition}
 \noindent 
  Along the footsteps of  Remark \ref{rmk:alternative-ch-requir}, we observe 
that, since $\teta \in L^\infty(r_0,r_1;\calZ)$,  requiring 
  $\rmD_z \calI(t,\teta(\cdot)) \in L^\infty(r_0,r_1; L^2(\Omega))$ is 
  equivalent to asking for $A_q(\teta(\cdot)) \in L^\infty(r_0,r_1; 
L^2(\Omega))$. 
 \par We trivially have 
 \begin{equation}
\label{comparison-costs}
  \Cost{\mathfrak{f}}t{z_0}{z_1}  \geq \calR_1(z_1{-}z_0) \quad \text{for every } t \in [0,T] \text{ and } z_0,\, z_1 \in \calZ.
 \end{equation}
 Up to a reparameterization, due to the positive homogeneity of the Finsler 
metric $\mathfrak{f}_t(z,\cdot)$,   we can suppose that the admissible 
transition curves are defined on $[0,1]$. %
 For later use we also introduce, for a fixed $\varrho>0$,   the set of 
admissible transition curves lying in a suitable ball of radius $\varrho$, 
i.e.\ 
 \begin{subequations}
 \label{threshold}
 \begin{align}
 &
 \label{threshold-a}
 \calT_t^\varrho(z_0,z_1): = \{ \teta \in \calT_t(z_0,z_1)\, : \
 \| \teta\|_{L^\infty (0,1;\calZ)} +  \| \teta'\|_{L^1(0,1;L^2(\Omega))} +  
  \|\rmD_z \ene t{\teta(\cdot)} \|_{L^\infty(0,1; L^2(\Omega))} \leq \varrho\} 
 \intertext{ and, accordingly, }
 & 
  \label{threshold-b}
 \Costn{\mathfrak{f}}{\varrho}t{z_0}{z_1}: = \inf_{ \vartheta\in \calT_{t}^\varrho(z_0,z_1)}
    \int_{r_0}^{r_1}
\mathfrak{f}_t(\teta(r),\teta'(r))\dd r\,.
 \end{align}
 \end{subequations}
  Since for every $\varrho>0$ there holds 
 $ \calT_{t}^\varrho(z_0,z_1) \subset  \calT_{t}(z_0,z_1)$, one has $ 
\Cost{\mathfrak{f}}t{z_0}{z_1} \leq \Costn{\mathfrak{f}}{\varrho}t{z_0}{z_1}$. 
 Indeed,
 \begin{equation}
 \label{cost-cost-subl}
   \Cost{\mathfrak{f}}t{z_0}{z_1} = \inf_{\varrho>0}  \Costn{\mathfrak{f}}{\varrho}t{z_0}{z_1} \qquad \text{ for every }  t \in [0,T] \text{ and } z_0,\, z_1 \in \calZ.
 \end{equation}
 For later use, we also record the following monotonicity property
 \begin{equation}
 \label{monoton-cost}
   \Costn{\mathfrak{f}}{\bar\varrho}t{z_0}{z_1}  = \inf_{0<\varrho<\bar\varrho}   \Costn{\mathfrak{f}}{\varrho}t{z_0}{z_1}
   =  \sup_{\varrho>\bar\varrho}   \Costn{\mathfrak{f}}{\varrho}t{z_0}{z_1}
   \text{ for every }  t \in [0,T], \  z_0,\, z_1 \in \calZ \text{ and } \bar\varrho>0,
 \end{equation}
 since $ \calT_{t}^{\varrho}(z_0,z_1) \subset  \calT_{t}^{\bar\varrho}(z_0,z_1)$ for every $0<\varrho<\bar\varrho$. 
% \begin{remark}
% \upshape
% \label{non-attainment}
 Observe that, for every
 fixed $\varrho>0$, the 
 $\inf$ in definition \eqref{threshold-b} 
  is attained, cf.\ 
 Proposition  \ref{prop:technical} ahead,
 whereas it need not be attained in the definition of $\Costname{\mathfrak{f}}$. In fact,  the dissipation functional $\mathfrak{f}$ does not control the norms of the spaces where we look for  admissible transition curves.   
 %%%%%
 \begin{remark}
 \label{rmk:comp-MRS13}
 \upshape The most striking difference between the present  definition of 
admissible curve  and the one given in \cite[Def.\ 3.4]{MRS16}
 resides in the fact that, 
 in contrast with conditions   (a) \& (b)   from Definition 
\ref{def:curves+cost}, 
 in \cite{MRS16} it was  only required
\begin{equation}
\label{Gt-teta}
\begin{gathered}
\teta|_{G_t[\teta]} \in \AC(G_t[\teta]; L^2(\Omega)) \quad \text{with the open set }
\\
G_t[\teta]: = \{ r \in [r_0,r_1]\, : \  \min_{\zeta\in \partial \calR_1(0) } \|- \rmD_z \calI(t,z)-\zeta\|_{L^2(\Omega)}>0\}\,.
\end{gathered}
\end{equation}
% As it will be clear from
 %the text illustrating Prop.\ \ref{prop:technical},
 The stronger condition $\teta \in  \mathrm{AC}([r_0.r_1];L^2(\Omega))$ reflects
  the fact that  the discrete viscous solutions $(\overline{z}_\tau)_\tau$ enjoy a (uniform, w.r.t.\ both parameters $\epsi$ and $\tau$) estimate in  $\mathrm{BV}([0,T];L^2(\Omega))$ (even in $ \mathrm{BV}([0,T];H^1(\Omega))$,
 cf.\ \eqref{est-epsi-tau-3}).
  Instead, in the general framework considered in \cite{MRS16} only 
 the \emph{basic} energy estimate 
 \[
 \int_0^T  \mathfrak{p}(\widehat{z}_\tau'(t),{-}\rmD_z \calI(\overline{t}_\tau(t), \overline{z}_\tau(t))) \dd t \leq \int_0^T \left( \calR_\epsilon (\widehat{z}_\tau'(t)) {+} \calR_\epsilon^* ({-}
 \rmD_z \calI(\overline{t}_\tau(t), \overline{z}_\tau(t))) \right) \dd  t \leq C 
 \]
 was available. In accordance with  that, only \eqref{Gt-teta} was required 
on admissible curves. 
 \par
 Condition (b) in Def.\ \ref{def:curves+cost}  reflects the enhanced estimate \eqref{est-epsi-tau-4-bis}.
 It is also
 peculiar of the present framework, and in particular   it is motivated by  
the fact that we impose unidirectionality of damage evolution, thus 
allowing $\calR_1$ to take the value $+\infty$.  In order to explain this, 
let us observe  
  that, in the setting considered in \cite{MRS16}, it was not necessary to 
specify the summability properties of 
  $\rmD_z \calI(t,\teta(\cdot))$ within the definition of admissible curve. 
Indeed, outside  the set $G_t[\teta]$ one had $\rmD_z \calI(t,\teta(\cdot)) \in  
\partial\calR_1(0)$, a bounded subset of $L^2(\Omega)$  since the dissipation 
potential $\calR_1$ was everywhere continuous. Instead,  on  the set 
$G_t[\teta]$  an estimate for the quantity $\min_{\zeta\in \partial \calR_1(0) } 
\|- \rmD_z \calI(t,z)-\zeta\|_{L^2(\Omega)}$ would morally provide a bound for 
$- \rmD_z \calI(t,z)$, as well, by comparison arguments, again thanks to the 
boundedness $ \partial\calR_1(0)$.
  Instead, in the present setting, since the set  $ \partial\calR_1(0)$ is 
unbounded, it is necessary to encompass  a suitable summability condition   on 
$\rmD_z \calI(t,\teta(\cdot)) $ in the definition of admissible curve. 
 \end{remark}
\par
We are now ready to introduce the jump variation induced by $\mathfrak{f}$, 
accounting for the energy dissipated
at  the jumps of a given curve $z \in \BV([0,T]; L^1(\Omega))$, with 
(countable)
jump set
\[ 
 {\mathrm{J}}_z : = \{ t \in [0,T]\, \, : \ z(t_-) \neq z(t) \text{ or } z(t_+) \neq z(t)\}
 \] 
 and $z(t_\pm)$ the right/left limits of $z$ at $t\in [0,T]$.  
  Based on the jump variation associated with $\mathfrak{f}$
 in \eqref{eq:37bis} ahead,
   we  introduce a novel notion of total variation for the curve $z$,
alternative to the 
% Preliminarily, let us recall the standard definitions of 
  total variation   induced by the dissipation potential $\calR_1$.
   We recall that, for a given curve $ z \in \BV([0,T]; L^1(\Omega))$  and $ 
[a,b]\subset [0,T]$, 
  the latter is given by 
\begin{equation}
\label{R1-tot-var}
   \Var{\calR_1}  zab: = 
   \sup\{ \sum_{m=1}^M \calR_1 ( z(t_m){-} z(t_{m-1})) \, : \ a=t_0<t_1<\ldots <t_{M-1}<t_M = b \}.
\end{equation}
 In particular,
the contribution at the  jumps induced by $\calR_1$ is
\[
\JVar{\calR_1}zab : = \calR_1( z(a_+){-}z(a)) +  \calR_1( z(b_-){-}z(b)) +  \sum_{t\in {\mathrm{J}}_z\cap
      (a,b)} \calR_1( z(t_+){-}z(t)) +  \calR_1( z(t_-){-}z(t)).
\]
For later convenience, we also introduce the scalar function
\begin{equation}
\label{variation-functional}
V(t): = \begin{cases}
0 & \text{ if } t \leq 0,
\\
\JVar{\calR_1}z0t & \text{ if } t \in (0,T),
\\
\JVar{\calR_1}z0T & \text{ if } t \geq T
\end{cases} \qquad 
 \text{ with  distributional derivative } \mu = \frac{\dd}{\dd t} V\,. 
\end{equation}
Recall that $\mu$ is a finite Borel measure supported on $[0,T]$, and it  can be decomposed as $\mu = \mu_{\mathrm{d}} + \mu_{\mathrm{J}}$, with $  \mu_{\mathrm{J}}$ the jump part, concentrated on the (countable) jump set $ {\mathrm{J}}_z$, and $ \mu_{\mathrm{d}} $ the diffuse part, given by the sum of the absolutely continuous and of the Cantor parts, so that 
$\mu_{\mathrm{d}}(\{t\})=0$ for every $t\in \R$.
\par
 We are now in a position to give the notion of 
%The forthcoming definition of 
total variation induced by $\mathfrak{f}$.  Let us mention in advance that 
it  is obtained by replacing  the  $\Jvarname{\calR_1}$-contribution to the 
total variation $\Varname{\calR_1}$, 
 with the $\mathfrak{f}$-jump variation, cf.\ \eqref{eq:81} below. 
  \begin{definition}[Jump and total variation induced by $\mathfrak{f}$]
  \label{def:Finsler_jump}
Let
  $z $ in  $  \mathrm{BV}([0,T];L^1(\Omega))$, with $z (t) \in \calZ$ for all $t 
\in [0,T]$,  be a
given curve with jump set $\mathrm{J}_z$. 
 Let $[a,b] \subset [0,T]$:
  \begin{enumerate}
  \item
 The  \emph{jump variation} of $z$
 on
$[a,b]$
  induced by $\mathfrak{f}$ is
\begin{equation}
  \label{eq:37bis}
  \begin{aligned}
    \JVar{\mathfrak{f}}zab &:=
    \Cost{\mathfrak{f}}a{z(a)}{z(a_+)}+\Cost{\mathfrak{f}}b{z(b_-)}{z(b)}
    \\
    & \quad + \sum_{t\in {\mathrm{J}}_z\cap
      (a,b)} \big(\Cost{\mathfrak{f}}t{z(t_-)}{z(t)}+
         \Cost{\mathfrak{f}}t{z(t)}{z(t_+)}\big).
  \end{aligned}
\end{equation}
\item
The  total variation  of $z$ on $[a,b]$ induced by $\mathfrak{f}$  is
\begin{align}
  \label{eq:81}
    \pVar{\mathfrak{f}}zab &:=
    \Var{\calR_1}  zab - \JVar{\calR_1}zab + \JVar{\mathfrak{f}}zab
    \\
     & 
     \label{eq:tvar-new}
     = \mu_{\mathrm{d}}([a,b])  + \JVar{\mathfrak{f}}zab\,.
 \end{align} 
\end{enumerate}
 For a given $\varrho>0$, we 
  use the symbols $ \JVarn{\mathfrak{f}}{\varrho}zab $ and $\pVarnamen {\mathfrak{f}} {\varrho}$
for the total variation induced by the cost $\Delta_{\mathfrak{f}}^{\varrho}$. 
\end{definition}
\noindent 
\noindent As already pointed out in \cite[Rmk.\ 3.5]{MRS12}, $\pVarname {\mathfrak{f}}$ is not a \emph{standard} total variational functional: it is neither induced by any distance on $L^1(\Omega)$,
nor is it lower semicontinuous w.r.t.\ pointwise convergence in $L^1(\Omega)$. Yet, it enjoys the additivity property.
\par
We are finally in a position to give our definition of Balanced  Viscosity 
solution to the rate-independent damage system. Again, we will consider a 
slightly stronger version than that given in \cite[Def.\ 3.10]{MRS16}, where $z 
\in \BV([0,T];L^1(\Omega))$ was only required.  Instead, here we will 
consider curves $z$ in  $\mathrm{BV}([0,T];L^2(\Omega))$ and, for technical 
reasons that will be apparent in the proof of the 
$\mathrm{BV}$-chain rule from Proposition \ref{prop:ch-rule} ahead, we will also 
restrict to curves $z$ 
such that $\rmD_z \calI(\cdot,z(\cdot)) \in L^\infty(0,T;L^2(\Omega))$.
Furthermore, 
unlike what was done in \cite{MRS16}, 
we will claim an energy balance involving a total variation $ 
\pVarn{\mathfrak{f}}{\varrho}z{0}{t}$ with  a threshold $\varrho>0$ such that  
\begin{equation}
\label{admissible-threshold}
\varrho \geq \| z\|_{L^\infty(0,T;\calZ) \cap \mathrm{BV}([0,T];L^2(\Omega))} +  \|\rmD_z  \calI(\cdot,z(\cdot)) \|_{ L^\infty(0,T;L^2(\Omega))}\,.
\end{equation}

\begin{definition}
  \label{def:BV-solution}
  A curve   $z $ in  $  L^\infty(0,T;\calZ) \cap 
\mathrm{BV}([0,T];L^2(\Omega))$, with
%$z (t) \in \calZ$ for all $t \in [0,T]$
\begin{equation}
\label{precise}
z(t)\in \calZ \  \text{ and } \ \rmD_z \calI(t, z(t)) \in L^2(\Omega) \quad \text{ for all $t\in [0,T]$}
\end{equation}
  and $\rmD_z  \calI(\cdot,z(\cdot)) \in L^\infty(0,T;L^2(\Omega))$, 
   is a \emph{Balanced  Viscosity} solution of the
    rate-independent  damage system  \eqref{dndia} if the \emph{local stability}
  \eqref{eq:65bis} and the \eqref{eq:84}-\emph{energy ba\-lan\-ce} hold:
  \begin{equation}
    \label{eq:65bis}
    \tag{$\mathrm{S}_\mathrm{loc}$}
-\rmD_z \calI (t,z(t)) \in \partial\calR_1(0)
    \quad \text{for all}\quad t \in [0,T] \setminus \mathrm{J}_z,
  \end{equation}
  \begin{equation}
    \label{eq:84}
     \pVarn{\mathfrak{f}}{\varrho}z{0}{t} +\calI(t, 
z(t))=\calI({0},{z(0)})+
    \int_{0}^{t} \partial_t \calI(s,{z(s)})\,\mathrm{d}s \quad \text{ for all } t\in (0,T].
    \tag{E$_{\mathfrak{f}}$}
  \end{equation}
\end{definition}
 with $\varrho>0$ fulfilling \eqref{admissible-threshold}. 
\begin{remark}
\upshape
The requirement $z \in  L^\infty(0,T;\calZ) $ in Def.\ \ref{def:BV-solution} is redundant
 and has been added only for the sake of clarity. 
 Indeed, since $\calI(0,z(0) ) \leq C$ as $z(0) \in \calZ$ (cf.\ 
\eqref{H2_improved-1}), and taking into 
account that $t\mapsto \partial_t \calI(t,z(t))$ is in $L^\infty(0,T)$ thanks to 
\eqref{stim3},  from \eqref{eq:84} we deduce that $|\calI(t,z(t))| \leq C$  
(recall that $\calI$ is bounded from below thanks to \eqref{est_coerc1}). In 
turn, this gives $z \in L^\infty(0,T;\calZ)$.
\par
On the other hand, combining the information $z \in L^\infty(0,T;\calZ)$ with 
estimate \eqref{estimate-for-DI} for $\rmD_z \wt\calI$, we conclude that $\rmD_z 
 \wt\calI(\cdot,z(\cdot)) \in L^\infty(0,T;L^2(\Omega))$. Therefore, what we are 
really requiring in Def.\ \ref{def:BV-solution} is that $A_q z \in 
L^\infty(0,T;L^2(\Omega))$, which enhances the regularity of $z$ to the space 
$L^\infty(0,T;W^{1+\sigma,q}(\Omega))$ for every $0<\sigma <\tfrac1q$ by 
Proposition \ref{prop:Savare98}.   
%\footnote{\RRRS \cite{Savare98}? Or another ref.??}
\end{remark} 
\par
 Prior to stating  the \textbf{main result of the paper},
Theorem \ref{thm:van-visc-eps-tau} below, we need to give the following definition, where  $z_-$ and $z_+$ are place-holders for the left and right limits of a curve $z$ at a jump point. 
\begin{definition}
\label{def:OJT} 
Let 
$\varrho>0$, 
$t\in [0,T]$,  and $z_-$, $z_+ \in \calZ$ be such that 
\begin{equation}
\label{stabil-zpm}
-\rmD_z\calI(t,z_-) \in \partial\calR_1(0) \text{ and }
-\rmD_z\calI(t,z_+) \in \partial\calR_1(0) \,.
\end{equation}
We say that an admissible transition curve $\teta \in  \calT_t^\varrho(z_-,z_+)$ 
is an 
\emph{optimal transition} between $z_-$ and $z_+$ if 
\begin{equation}
\label{form-OJT}
\calI(t,z_-) - \calI(t,z_+) = \Costn{\mathfrak{f}}{\varrho}t{z_-}{z_+} = \int_0^1 \mathfrak{f}_t(\teta(r), \teta'(r)) \dd r =  \mathfrak{f}_t(\teta(r), \teta'(r)) \qquad \foraa\, r \in (0,1).
\end{equation}
We will denote by $ \calO_t^\varrho (z_-,z_+)$   the  
collection of such transitions.
\end{definition}
A few comments are in order. First of all, with  \eqref{stabil-zpm}   we are imposing that the points $z_-$ and $z_+$ to be connected fulfill 
the local stability condition. It is not difficult to check that this is verified whenever $z_-$ and $z_+$ are the left and right limits at a jump point of a Balanced Viscosity solution. Secondly, let us gain further insight into \eqref{form-OJT}: with the second equality, we are asking that 
$\teta$ (which we may always suppose to be defined on $[0,1]$) is a minimizer in the definition of $\Costn{\mathfrak{f}}{\varrho}t{z_-}{z_+}$; with the third one, that $\teta$ has constant `$`\mathfrak{f}_t$-velocity", which can be obtained by a rescaling argument. The first equality relates to the jump conditions verified along any Balanced Viscosity solution, cf.\ \eqref{jump-f} ahead. 
\par
We are now in a position  to give Thm.\  \ref{thm:van-visc-eps-tau} , stating 
 the convergence of the discrete solutions of the viscous damage system to a 
Balanced Viscosity solution of the rate-independent damage system, as the 
parameters $\epsilon$ and $\tau$ tend to zero 
\emph{simultaneously}, with  $\tfrac{\eps}{\tau} \uparrow \infty$.  In 
fact, we will retrieve a Balanced Viscosity solution $z$ with enhanced 
properties: 
\begin{itemize}
\item[(i)]
 we have that 
$z \in \BV([0,T];H^1(\Omega))$, which reflects the enhanced discrete $\BV$-estimate \eqref{est-epsi-tau-3};
\item[(ii)] at all jump points $t$  of $z$, the left and right limits
$z(t_-)$ and $z(t_+)$ 
 can be connected by an optimal jump transition in the sense of Definition \ref{def:OJT}, so that the set $\calO_t^{\bar\varrho} (z(t_-),z(t_+))$ 
 is non-empty.
 Additionally, such transition  has  finite $H^1(\Omega)$-length. 
Furthermore, the total $H^1(\Omega)$-length of the connecting paths is finite.
\end{itemize}
 Observe that property (ii) is not encoded in Definition 
\ref{def:BV-solution}, which  gives  $   
    \pVar{\mathfrak{f}}z{0}{T}<\infty$,  since $   
    \pVar{\mathfrak{f}}z{0}{T}$ only controls the ``$\mathfrak{f}$-length'' of 
the optimal jump paths. 
    \par
This  enhanced concept of Balanced Viscosity solution was already  
introduced in the general setting of \cite{MRS16}, cf.\ Section 3.4 therein. 
Along the footsteps of \cite{MRS16}, we will refer to these  solutions as  
\emph{$H^1(\Omega)$-parameterizable Balanced Viscosity solutions}.
%%%
\begin{theorem}
\label{thm:van-visc-eps-tau}
Under Assumptions \ref{ass:domain}, \ref{assumption:energy},  and 
\ref{ass:load}, let $z_0 \in \calZ$, fulfilling \eqref{further-reg}, 
% with $\rmD_z \calI(0,z_0) \in L^2(\Omega)$, 
be approximated by discrete initial data
$(z_{\tau,\epsilon}^0)_{\tau,\epsilon}$ such that
\begin{equation}
\label{cvg-initial}
z_{\tau,\epsilon}^0 \to z_0 \quad \text{ in } \calZ, \qquad   \calI(0,z_{\tau,\epsilon}^0) \to  \calI(0,z_0), \quad \rmD_z \calI(0,z_{\tau,\epsilon}^0) \weakto \rmD_z\calI(0,z_0) \quad \text{ in }L^2(\Omega),
\end{equation}
and let $(\pwc z\tau{\epsilon})_{\tau,\epsilon}$, $(\pwl z\tau{\epsilon})_{\tau,\epsilon}$ be the discrete solutions to the viscous damage system \eqref{dndia-eps} starting from the 
data
$(z_{\tau,\epsilon}^0)_{\tau,\epsilon}$. 
\par
Then, 
there exists $\bar\varrho>0$, only depending on the problem data  (cf.\ 
\eqref{conds-tech}  below)
%\footnote{\RRN clear enough??}
and fulfilling 
 \eqref{admissible-threshold},  such that 
for 
all sequences $(\tau_k,\epsilon_k)_k$ 
satisfying
\begin{equation}
\label{crucial-seq-params}
\lim_{k\to\infty}\epsilon_k=0 \quad \text{and} \quad \lim_{k\to\infty} \frac{\tau_k}{\epsilon_k} =0,
\end{equation}
there exist a (not relabeled) subsequence,  and a Balanced  Viscosity solution 
$z$ to the rate-independent damage system \eqref{dndia}, 
 fulfilling  $z(0)=z_0$, the energy balance \eqref{eq:84} with 
\begin{equation}
\label{threshold-Var}
  \pVarn{\mathfrak{f}}{\bar\varrho}z{0}{t} =  \sup_{\varrho\geq \bar\varrho}\pVarn{\mathfrak{f}}{\varrho}z{0}{t}
  =\inf_{\varrho\geq \bar\varrho}\pVarn{\mathfrak{f}}{\varrho}z{0}{t}
   \quad \text{for every } t\in [0,T] %\ \text{for some } \bar\varrho>0,
\end{equation}
 and  such that  the following convergences hold as $k\to\infty$, at every 
$t\in [0,T]$:
\begin{subequations}
\label{convergences}
\begin{align}
&
\label{cvg-1-b}
\pwc z{\tau_k}{\epsilon_k}(t),\, \pwl z {\tau_k}{\epsilon_k}(t)  \to z(t) \quad \text{ in } \calZ,
\\
&
\label{cvg-2-b}
\calI(t, \pwc z{\tau_k}{\epsilon_k}(t)),\,  \calI(t, \pwl z{\tau_k}{\epsilon_k}(t)) \to \calI(t,z(t)),
\\
&
\label{cvg-3-b}
\int_{0}^{\overline{t}_{\tau}(t)}
\left(\calR_\epsi (\widehat{z}'_{\tau}(r))+\calR_\epsi^* (-\rmD_z
\calI (\overline{t}_{\tau}(r),\overline{z}_{\tau}(r)))  \right)
\,\mathrm{d}r  \to \pVarn{\mathfrak{f}}{\bar\varrho}z{0}{t}\,. 
\end{align}
\end{subequations}
\par
Furthermore, $z$ is a $H^1(\Omega)$-parameterizable Balanced  Viscosity 
solution, namely $z \in \BV ([0,T]; H^1(\Omega))$, and 
\begin{subequations}
\label{H1-param}
\begin{align}
\label{optim-trans-1} 
&
(1) &&\forall\, t \in \mathrm{J}_z  \ \exists\,  \teta_t \in 
\calO_t^{\bar\varrho} (z(t_-),z(t_+)) \text{ s.t. } \teta_t \in 
\AC([0,1];H^1(\Omega));
\\
\label{optim-trans-2} 
& (2) && \sum_{t\in \mathrm{J}_z} \int_0^1 \| \teta_t'(r)\|_{H^1(\Omega)} \dd r <\infty.
\end{align}
\end{subequations}
\end{theorem} 
Observe that \eqref{threshold-Var} is an additional property, cf.\ 
\eqref{monoton-cost}. The constant $\bar\varrho$ will be specified  along the 
proof of Theorem~\ref{thm:van-visc-eps-tau}, postponed to Section~\ref{ss:5.3}. 
Instead, in the forthcoming Sec.~\ref{ss:5.2} we gain further insight into the notion of Balanced Viscosity  
solution for our damage system, in particular
focusing on the description of the behavior of the system at jumps. 
\subsection{Properties of Balanced Viscosity solutions}
\label{ss:5.2}
\noindent
 One of the cornerstones of the proof of Thm.~\ref{thm:van-visc-eps-tau}  
is a characterization of Balanced Viscosity solutions 
in terms of the local stability condition \eqref{eq:65bis}, combined with the
\emph{upper energy estimate} in \eqref{eq:84}. The proof of this 
characterization relies on 
 %We start by giving a
  \emph{chain-rule inequality} for $\calE$, evaluated along a  \emph{locally stable} %(i.e., fulfilling \eqref{eq:65bis})
   curve
   with the regularity and summability properties specified in Definition 
\ref{def:BV-solution}. 
    This inequality involves  the non-standard
total variation functional 
$\pVarname {\mathfrak{f}}$.
\begin{proposition}[$\BV$-chain rule inequality]
\label{prop:ch-rule}
Under Assumptions \ref{ass:domain},  \ref{assumption:energy}, and 
\ref{ass:load}, let   $z\in L^\infty (0,T;\calZ) \cap 
\BV([0,T];L^2(\Omega))$, with $\rmD_z \calI(\cdot, z(\cdot)) \in 
L^\infty(0,T;L^2(\Omega))$, also fulfill \eqref{precise}. 
  Let $\varrho$ fulfill \eqref{admissible-threshold}.
  Suppose that $z$ satisfies the local stability condition 
 \eqref{eq:65bis}, with $\pVarn{\mathfrak{f}}{\varrho}z{0}{T}  
%\pVarn{\mathfrak{f}}z{0}{T} 
<  \infty$. 
  Then, the map $t \mapsto \ene t{u(t)}$ belongs to 
 $\BV([0,T])$ and satisfies the chain rule inequality
 \begin{equation}
 \label{ch-ineq-dam}
 \left| \ene{t_1}{u(t_1)} {-}  \ene{t_0}{u(t_0)}  
 {-} \int_{t_0}^{t_1}\partial_t \ene{t}{z(t)} \dd t \right| \leq 
\pVarn{\mathfrak{f}}{\varrho}z{t_0}{t_1}   \qquad \text{for all } 0 \leq t_0 
\leq t_1 \leq T\,.
 \end{equation}
  %%%
%%% 
\end{proposition}
We postpone its \emph{proof} to Section~\ref{ss:5.3}. 
We now characterize Balanced Viscosity solutions in terms of  the local 
stability  \eqref{eq:65bis}, joint with the
 upper energy estimate in 
 \eqref{eq:84}, which it is sufficient to give on 
the whole time interval $[0,T]$. Namely we have
\begin{corollary}
\label{cor:this-only}
Under Assumptions \ref{ass:domain}, \ref{assumption:energy}, and \ref{ass:load},
a curve  $z\in \BV([0,T];L^2(\Omega))$  is a Balanced Viscosity solution of the 
rate-independent  damage system \eqref{dndia} (in the sense of Definition 
\ref{def:BV-solution})
if and only if it satisfies \eqref{eq:65bis} and 
\begin{equation}
\label{one-sided-0T}
    \pVarn{\mathfrak{f}}{\varrho}z{0}{T}+\calI(T, z(T)) \leq \calI({0},{z(0)})+
    \int_{0}^{T} \partial_t \calI(s,{z(s)})\,\mathrm{d}s
\end{equation}
 for some   $\varrho$ fulfilling \eqref{admissible-threshold}. 
\end{corollary}
For the \emph{proof}, we refer the reader to the argument for \cite[Cor.\ 
3.14]{MRS16}. 
Corollary \ref{cor:this-only} will play a crucial role in the proof of 
Theorem \ref{thm:van-visc-eps-tau}, for it will allow us to focus on the proof 
of 
 \eqref{eq:65bis} and of  
the energy inequality 
\eqref{one-sided-0T}, only,  in place of the balance \eqref{eq:84}. In turn, 
\eqref{one-sided-0T} will be achieved by means of careful lower 
semicontinuity arguments. The second outcome of the characterization provided by 
Cor.\ \ref{cor:this-only} is  
the following Proposition 
\ref{prop:charact-jump-conds},  which was proved in the abstract setting in 
\cite[Thm.\ 3.15]{MRS16}. It shows 
  that a locally stable curve is a   Balanced Viscosity solution of the 
rate-independent system if and only if it fulfills
\begin{itemize}
\item[(i)] an energy-dissipation inequality only featuring the $\calR_1$-total variation functional from \eqref{R1-tot-var}, 
cf.\ \eqref{glob-en-ineq} below, 
and 
\item[(ii)] at each jump point, the jump conditions  \eqref{jump-f} featuring the Finsler cost  $\Delta_{\mathfrak{f}}$ induced by $\mathfrak{f}$.
\end{itemize}
 Concerning (i), let us also mention that it is possible to show  (cf.\ 
\cite[Thm.\ 3.16]{MRS16})
that any Balanced Viscosity solution also satisfies the subdifferential inclusion
\begin{equation}
\label{ri-subdif-L2}
\partial\calR_1(z'(t)) + \rmD_z \calI(t,z(t)) \ni 0 \quad \text{in }L^2(\Omega) 
\end{equation}
 at every $t\in (0,T)$ that is not a  jump point, hence for almost all $t\in 
(0,T)$. The system behavior at jump points is  instead described by the jump 
conditions  \eqref{jump-f} below. 
This further characterization of the  Balanced Viscosity concept
 in terms of (i) and (ii)   highlights how it differs in comparison to 
the standard Global Energetic notion. The latter can be characterized in 
terms of the \emph{global stability} condition,  the energy-dissipation 
inequality \eqref{glob-en-ineq}, and the analogues of  the jump conditions 
 \eqref{jump-f}, with the cost $\Delta_{\mathfrak{f}}(t;\cdot,\cdot)$
 replaced by $\calR_1$.
 Conditions \eqref{jump-f} highlight that
  the  viscous approximation from which Balanced Viscosity solutions originate  
enters into play  in the description of the energetic behavior of the system at 
jumps. 
 \begin{proposition}
 \label{prop:charact-jump-conds}
 A curve $z\in \BV ([0,T]; L^2(\Omega))$ is a Balanced Viscosity solution of the rate-independent damage system \eqref{dndia} if and only if it satisfies   \eqref{eq:65bis}, the 
 $(\calR_1)$-energy dissipation inequality 
 \begin{equation}
 \label{glob-en-ineq}
    \pVar{\calR_1}z{s}{t}+\calI(t, z(t)) \leq \calI({s},{z(s)})+
    \int_{s}^{t} \partial_t \calI(s,{z(s)})\,\mathrm{d}s \quad \text{for all } 0 \leq s \leq t \leq T,
    \end{equation}
and the jump conditions 
  \begin{equation}
    \label{jump-f}
      \begin{aligned}
    \ene{t}{z(t)}-\ene t{z(t_-)}&=-\Costn{\mathfrak{f}}{\varrho}t{z(t_-)}{z(t)},\\
    \ene{t}{z(t_+)}-\ene t{z(t)}&=-\Costn{\mathfrak{f}}{\varrho}t{z(t)}{z(t_+)},\\
    \ene{t}{z(t_+)}-\ene
    t{z(t_-)}&=-\Costn{\mathfrak{f}}{\varrho}t{z(t_-)}{z(t_+)}
    \\&
    =
    -\Big(\Costn{\mathfrak{f}}{\varrho}t{z(t_-)}{z(t)}+
    \Costn{\mathfrak{f}}{\varrho}t{z(t)}{z(t_+)}\Big)
  \end{aligned}
  \end{equation}    
  at every $t\in \mathrm{J}_z$. 
 \end{proposition}
\noindent
The \emph{proof} follows  the very same lines as the argument for \cite[Thm.\ 
3.15]{MRS16}.
\par
 We conclude this section by shedding further light into the
 the fine properties of  optimal jump transitions. %As in \cite{MRS16}, given 
 %$z_0,\, z_1 \in \calZ$ and 
 Following \cite[Sec.\ 3.4]{MRS16}, we say that an optimal transition $\teta\in 
\calO_t^{\varrho} (z_-,z_+)$  is 
of 
\begin{itemize}
\item
 \emph{sliding} type if $-\rmD_z \calI(t,\teta(r)) \in \calR_1(0)$ for every $r\in [0,1]$;
 \item
 \emph{viscous} type if 
 $-\rmD_z \calI(t,\teta(r)) \notin \calR_1(0)$ for every $r\in [0,1]$.
 \end{itemize}
The forthcoming result on sliding and viscous optimal transitions follows from the very same argument as in the proof of 
\cite[Prop.\ 3.19]{MRS16}.
\begin{proposition}
Let 
$\varrho>0$, 
$t\in [0,T]$,  and $z_-$, $z_+ \in \calZ$ fulfilling 
\eqref{stabil-zpm} be given. Let $\teta \in \calO_t^\varrho(z_-,z_+)$.
Then,
\begin{enumerate}
\item
$\teta $ is of sliding type if and only if it satisfies
\[
\partial\calR_1(\teta'(r)) +\rmD_z \calI(t,\teta(r)) \ni 0 \quad \text{in }L^2(\Omega) \ \foraa\, r \in (0,1);
\]
\item
$\teta $ is of viscous type if and only if 
there exists a map $\epsilon :(0,1) \to (0,+\infty)$ such that $\teta$ and $\epsilon$ satisfy
\[
\partial\calR_1(\teta'(r)) +\epsilon(r) \teta'(r) +\rmD_z \calI(t,\teta(r)) \ni 0 \quad \text{in }L^2(\Omega) \ \foraa\, r \in (0,1);
\]
\item
Every optimal transition $\teta$ can be decomposed in a canonical way into an (at most) countable collection of optimal 
\emph{sliding} and \emph{viscous} transitions. 
\end{enumerate}
\end{proposition}

\section{Proofs}
\label{ss:5.3}

We start by giving Proposition \ref{prop:technical}, which  is  the 
counterpart to \cite[Thm.\ 3.7]{MRS16}. A comparison between the latter result 
and Proposition \ref{prop:technical}   below reflects the major differences 
between the present context  and that of \cite{MRS16}: The transition curves by 
means of which the Finsler cost $\Costname{\mathfrak{f}}$ from \eqref{eq:69} is 
defined have better properties than their analogues in \cite{MRS16}, cf.\ also 
Remark \ref{rmk:comp-MRS13}.
This is also apparent from item (3) of the ensuing  statement, yielding the 
existence of a transition path $\teta$ in the space $ 
W^{1,\infty}(0,1;H^1(\Omega))$, even, in accordance with the uniform  bound 
\eqref{est-epsi-tau-3} for the discrete solutions. 
 % Let us also mention in advance that 
\begin{proposition}
\label{prop:technical}
Let $t\in [0,T]$ and  $z_0,\, z_1 \in \calZ$ be fixed.
Then:
\begin{enumerate}
\item For every $\varrho>0$  such that $\max_{i=0,1 }( \|z_i \|_{\calZ} + 
\|\rmD_z \calI(t,z_i)\|_{L^2(\Omega)}) \leq \varrho$  and   
$\Costn{\mathfrak{f}}{\varrho}{t}{z_0}{z_1}<+\infty$, there exists an optimal 
transition path $\teta \in \mathcal{T}_t^\varrho(z_0,z_1)$ attaining the $\inf$ 
in the definition of $\Costn{\mathfrak{f}}{\varrho}{t}{z_0}{z_1}$, cf.\ 
\eqref{threshold};
\item
 Let $(z_0^n)_n,\, (z_1^n)_n \subset \calZ$ fulfill
\[
z_0^n \to z_0, \quad z_1^n\to z_1 \quad \text{in } \calZ.
\]
Then,
\begin{equation}
\label{lsc-costs}
\liminf_{n\to\infty} \Costn{\mathfrak{f}}{\varrho}{t}{z_0^n}{z_1^n} \geq  
\Costn{\mathfrak{f}}{\varrho}{t}{z_0}{z_1}
\end{equation}
 for every $\varrho \geq \sup_{i=1,2,n \in \N}( \|z_i \|_{\calZ} + \|\rmD_z 
\calI(t,z_i)\|_{L^2(\Omega)})$.  
\item
Let the  sequences  $(\alpha_k)_k,\, (\beta_k)_k \subset [0,T]$, 
$(\widehat{z}_k)_k \subset L^\infty(\alpha_k,\beta_k;\calZ) \cap 
\mathrm{AC}([\alpha_k,\beta_k];H^1(\Omega))  $,  $(\overline{z}_k)_k \subset 
L^\infty(\alpha_k,\beta_k;\calZ)$,
 fulfill
\begin{equation}
\label{conds-tech}
\begin{aligned}
&
\lim_{k\to\infty} \alpha_k =t = \lim_{k\to\infty} \beta_k,
\quad \overline{z}_k(\alpha_k) \to z_0 \text{ in } \calZ, \quad  
\overline{z}_k(\beta_k) \to z_1 \text{ in } \calZ,
\\
&
\lim_{k\to\infty} \sup_{r\in [\alpha_k,\beta_k]} \|\overline{z}_k(r) - 
\widehat{z}_k(r)\|_{H^1(\Omega)} =0,
\\
&
\exists\, \bar{\varrho}>0  \quad  \forall\, k \in \N \, : 
\\
& \qquad  
\|  \widehat{z}_k \|_{L^\infty(\alpha_k,\beta_k;\calZ)  \cap W^{1,1} 
(\alpha_k,\beta_k; H^1(\Omega))}
%\\
%& \qquad \qquad 
+ \| \overline{z}_k\|_{L^\infty (\alpha_k,\beta_k;\calZ)} 
+ \|   \rmD_z \ene{\overline{t}_{\tau_k}}{\overline{z}_k}  
\|_{L^\infty(\alpha_k,\beta_k;L^2(\Omega))}  \leq  \bar{\varrho}\,.
\end{aligned}
\end{equation}
Then, 
there exists a (not relabeled) increasing subsequence of $(k)$, increasing and 
surjective time rescalings $\mathsf{t}_{k} \subset \AC ([0,1]; 
[\alpha_{k},\beta_{k}])$ and an admissible transition  $\teta \in 
\mathcal{T}_t^{\bar{\varrho}} (z_0,z_1)$  such that 
\begin{subequations}
\label{thesis:technical}
\begin{align}
&
\label{thesis:technical-1}
\lim_{k\to\infty} \sup_{s \in [0,1]}\| \overline{z}_{k} \circ \mathsf{t}_{k}(s) 
{-} \teta(s)\|_{H^1(\Omega)} = \lim_{k\to\infty} \sup_{s \in [0,1]}\| 
\widehat{z}_{k} \circ \mathsf{t}_{k}(s) {-} \teta(s)\|_{H^1(\Omega)} =0,
\\
& 
\label{thesis:technical-2}
\text{in addition, } \teta \text{ is in }  %L^\infty (0,1;\calZ) \cap
 W^{1,\infty}(0,1;H^1(\Omega)),  \text{ and }
\\
%&
&
\label{thesis:technical-4}
\Costn{\mathfrak{f}}{\bar\varrho}{t}{z_0}{z_1} \leq 
\int_0^1 \mathfrak{f}_t [\teta(s),\teta'(s)] \dd s \leq 
 \liminf_{k\to\infty}  \int_{\alpha_{k}}^{\beta_{k}} 
\left(
\calR_{\epsi_{k}}(\widehat{z}_{k}'(r)) {+} \calR_{\epsi_{k}}^* ({-} \rmD_z 
\ene{\overline{t}_{\tau_{k}}(r)}{\overline{z}_{k}(r)}) \right) \dd r \,.
\end{align}
\end{subequations}
\end{enumerate}
\end{proposition}
\begin{proof}
%As for \textbf{(1)}, we  refer to the proof of \cite[Thm.\ 3.7]{MRS16}.
We start by addressing  the proof of \textbf{(2)}: Along the footsteps of 
the proof of \cite[Thm.\ 3.7]{MRS16}, we consider    a sequence of admissible 
transitions $\teta_n \in \mathcal{T}_t^{\varrho}(z_0^n,z_1^n)$  such that 
\[
\int_0^1 \mathfrak{f}_t(\teta_n(r),\teta_n'(r)) \dd r \leq 
\Costn{\mathfrak{f}}{\varrho}t{z_0^n}{z_1^n} +\eta_n \qquad \text{with } \eta_n 
\geq 0 \text{ and } \lim_{n\to\infty}\eta_n = \eta\geq 0\,.
\]
We perform the change of variable
\begin{equation}
\label{change-var}
\sfs_n(r): = c_n \left( r {+} \int_0^r \|\teta_n'(\sigma)\|_{L^2(\Omega)} \dd 
\sigma \right), \qquad \sfr_n: = \sfs_n^{-1}: [0, \mathsf{S}] \to [0,1], \qquad 
\sfteta_n: =\teta_n \circ \sfr_n: [0,\mathsf{S}]\to \calZ,
\end{equation}
with $c_n$ a normalization constant such that $\mathsf{S} = \sfs_n(1)$ is 
independent of $n\in\N$. In view of  the estimate $\| 
\teta_n'\|_{L^1(0,1;L^2(\Omega))} \leq \varrho$ encoded in the definition of 
$\Costname{\mathfrak{f}}^{\varrho}$, we  have that $c_n \geq \bar c>0$ for all 
$n\in \N$.  The curves $(\sfr_n,\sfteta_n)_n$ fulfill the normalization 
condition
\begin{subequations}
\label{est-curves}
\begin{equation}
\label{est-curves-1}
\sfr_n'(s) + \|\sfteta_n'(s)\|_{L^2(\Omega)} =\frac1{c_n} \leq \frac1{\bar c} 
\qquad \foraa\, s \in (0,\mathsf{S})
\end{equation}
and, moreover, 
\begin{equation}
\label{est-curves-2}
  \| \sfteta_n \|_{L^\infty(0,\mathsf{S}; \calZ)} +   \| \sfteta_n' 
\|_{L^1(0,\mathsf{S}; L^2(\Omega))} +  \| \rmD_z \ene t{\sfteta_n(\cdot)} 
\|_{L^\infty (0,\mathsf{S}; L^2(\Omega))} \leq \varrho.
\end{equation}
It follows from the first bound in \eqref{est-curves-2} and from 
\eqref{estimate-for-DI} that 
$ \| \rmD_z \widetilde{\calI}(t,{\sfteta_n(\cdot)}) \|_{L^\infty (0,\mathsf{S}; 
L^2(\Omega))}\leq C$. Therefore we deduce that $\| A_q (\sfteta_n)\|_{L^\infty 
(0,\mathsf{S}; L^2(\Omega))}\leq C$, which yields, in view of the aforementioned 
regularity results from 
Proposition \ref{prop:Savare98}, a bound 
for $(\sfteta_n)_n$ in $L^\infty (0,\mathsf{S}; 
W^{1+\sigma,q}(\Omega)))$ for all $0<\sigma<\frac1q$.
\end{subequations}
In view of  \eqref{est-curves-1},  there exists $\sfr \in 
W^{1,\infty}(0,\mathsf{S})$ such that, up to a not relabeled subsequence, 
$\sfr_n \to r$ uniformly in $[0,\mathsf{S}]$ and weakly$^*$ in 
$W^{1,\infty}(0,\mathsf{S})$.
Furthermore, by Aubin-Lions type compactness results (cf., e.g.\   
\cite[Thm.\ 5, Cor.\ 4]{simon87}), there exists a curve  $\sfteta \in  
L^\infty(0,\mathsf{S}; W^{1+\sigma,q}(\Omega)) \cap \rmC^0 
([0,\mathsf{S}];\calZ)  \cap W^{1,\infty}(0,\mathsf{S}; L^2(\Omega))$ for all 
$0<\sigma<\frac1q,$ with $ \rmD_z \ene t{\sfteta(\cdot)} \in L^\infty 
(0,\mathsf{S}; L^2(\Omega))$, such that
\begin{equation}
\label{needed-technical}
\begin{aligned}
&
\sfteta_n \weaksto \sfteta && \text{ in } L^\infty(0,\mathsf{S}; 
W^{1+\sigma,q}(\Omega)) \cap W^{1,\infty}(0,\mathsf{S}; L^2(\Omega)) \quad 
\text{for all } 0<\sigma<\frac1q,
\\
&\sfteta_n \to \sfteta && \text{ in } \rmC^0 ([0,\mathsf{S}];\calZ)\,,
\\
&
\rmD_z \ene t{\sfteta_n} \weaksto \rmD_z \ene t{\sfteta} && \text{ in } L^\infty 
(0,\mathsf{S}; L^2(\Omega))
\end{aligned}
\end{equation}
(the latter convergence property following from the fact that 
$\rmD_z \ene t{\sfteta_n} = A_q (\sfteta_n) + \rmD_z \tildene t{\sfteta_n} $ 
converges strongly to $ \rmD_z \ene t{\sfteta}$ in 
$L^\infty(0,\mathsf{S};\calZ^*)$ in view of  the second of 
\eqref{needed-technical}, combined with \eqref{weak-continuity}).
Therefore,  
\[
  \| \sfteta \|_{L^\infty(0,\mathsf{S}; \calZ)} +   \| \sfteta' 
\|_{L^1(0,\mathsf{S}; L^2(\Omega))}+
\| \rmD_z \ene t{\sfteta(\cdot)} \|_{L^\infty (0,\mathsf{S}; L^2(\Omega))} \leq 
\varrho.
\]
We thus conclude that $\sfteta \in \mathcal{T}_t^{\varrho}(z_0,z_1)$; up to 
a reparameterization, we may suppose $\sfteta$ to be  defined on $[0,1]$.
Arguing  in the very same way as in the proof of \cite[Thm.\ 5.1]{krz}, 
\cite[Thm.\ 7.4]{KRZ2},     
%\footnote{should we give more details here?} 
we see that 
\[
\begin{aligned}
\eta+\liminf_{n\to\infty} \Costn{\mathfrak{f}}{\varrho}{t}{z_0^n}{z_1^n} &  \geq 
\liminf_{n\to\infty}  \int_0^1 \mathfrak{f}_t(\teta_n(r),\teta_n'(r)) \dd r = 
\liminf_{n\to\infty} \int_0^{\mathsf{S}}  
\mathfrak{f}_t(\sfteta_n(s),\sfteta_n'(s)) \dd s 
\\
&  
\geq  \int_0^{\mathsf{S}}  \mathfrak{f}_t(\sfteta(s),\sfteta'(s)) \dd s 
\geq \Costn{\mathfrak f}{\varrho}t
{z_0}{z_1}\,.
\end{aligned}
\]
Observe that the last inequality follows from the fact that $\sfteta$ is an 
admissible curve between $z_0$ and $z_1$.  Since $\eta\geq 0$ is arbitrary,   
this concludes the proof of \textbf{(2)}; a slight modification of this argument 
yields  \textbf{(1)}, as well.  
\par
In order to prove  \textbf{(3)}, we can confine the discussion to the case 
$z_0\neq z_1$, so that 
\[
\lim_{k\to\infty}   \int_{\alpha_k}^{\beta_k} 
\left(
\calR_{\epsi_k}(\widehat{z}_k'(r)) {+} \calR_{\epsi_k}^* ({-} \rmD_z 
\ene{\overline{t}_{\tau_k}(r)}{\overline{z}_k(r)}) \right) \dd r  =: L \geq 
\calR_1(z_1{-}z_0)>0\,.
\]
In analogy with 
\eqref{change-var}, 
 but taking now into account that $ (\widehat{z}_k)_k$ is bounded in 
$W^{1,1}(\alpha_k,\beta_k;H^1(\Omega))$ by \eqref{conds-tech}, 
we define 
\[
\sfs_k(r): = c_k \left( r {+} \int_0^r   \| 
\widehat{z}_k'(\sigma)\|_{H^1(\Omega)} \dd \sigma \right)
\quad \text{for all } r \in [0,\beta_k -\alpha_k]
\]
where the normalization constant $c_k$ is  now chosen in such a way as to have 
$\sfs_k(\beta_k-\alpha_k)=1$. Thus, we set
\[
  \sft_k: = \sfs_k^{-1}: [0, 1] \to [\alpha_k,\beta_k], \qquad 
\overline{\sfz}_k: =\overline{z}_k \circ \sft_k,\,  \widehat{\sfz}_k: 
=\widehat{z}_k \circ \sft_k : [0,1]\to \calZ,
  \]
and observe that  the following estimates hold
\begin{subequations}
\label{other-ests-palla}
\begin{align}
\label{palla1}
&
\| \sft_k\|_{W^{1,\infty}(0,1)} + 
 \|  \widehat{\sfz}_k \|_{W^{1,\infty} (0,1;H^1(\Omega))} 
 \leq C,
 \\
 & 
 \label{palla1/2}
 \| \overline{\sfz}_k\|_{L^\infty (0,1;\calZ)} 
 +
 \| \widehat{\sfz}_k\|_{L^\infty (0,1;\calZ)}
 +   \| \widehat{\sfz}_k'\|_{L^1 (0,1;H^1(\Omega))} 
  + \|   \rmD_z \ene{\overline{t}_{\tau_k} \circ \sft_k}{\overline{\sfz}_k}  
\|_{L^\infty(0,1;L^2(\Omega))}   \leq \bar\varrho\,,
\end{align}
where \eqref{palla1} is due to the analogue of the normalization condition 
\eqref{est-curves-1}, while \eqref{palla1/2} derives from \eqref{conds-tech}. 
From the  bound for $ \| \rmD_z \ene{\overline{t}_{\tau_k} \circ 
\sft_k}{\overline{\sfz}_k} \|_{L^\infty(0,1;L^2(\Omega))}$, taking into account 
that 
$
\|   \rmD_z \widetilde{\calI}(\overline{t}_{\tau_k} \circ \sft_k, \overline{\sfz}_k) 
 \|_{L^\infty(0,1;L^2(\Omega))}   \leq C $ in view of \eqref{estimate-for-DI} 
and the estimate
$ \| \overline{\sfz}_k\|_{L^\infty (0,1;\calZ)} \leq C$, 
we also deduce 
\begin{equation}
\label{palla2}
\|   A_q(\overline{\sfz}_k)  \|_{L^\infty(0,1;L^2(\Omega))}   \leq C\,.
\end{equation}
\end{subequations}
Combining  estimates \eqref{other-ests-palla} with, again, the compactness 
results 
\cite[Thm.\ 5, Cor.\ 4]{simon87}, and taking into account that 
$(\overline{\sfz}_k)$ and $(\widehat{\sfz}_k)_k$ converge to the same limit in 
view of the second of \eqref{conds-tech}, 
with the very same arguments as in the proof of \textbf{(2)}
we conclude that 
 there exists $\teta  $   such that
\begin{subequations}
\label{convs-palla}
\begin{align}
&
\label{convs-palla-1}
\widehat{\sfz}_k\weaksto \teta  && \text{ in }   L^\infty (0,1; \calZ) \cap 
W^{1,\infty}(0,1; H^1(\Omega)),
\\
&
\label{convs-palla-2}
\overline{\sfz}_k\weaksto \teta  && \text{ in }   L^\infty (0,1;  
W^{1+\sigma,q}(\Omega)) \quad \text{for all } 0 <\sigma<\frac1q,
\\
&
\label{convs-palla-3}
\overline{\sfz}_k\to \teta && \text{ in } L^\infty (0,1; \calZ),
\\
& 
\label{convs-palla-4}
\widehat{\sfz}_k\to \teta  && \text{ in }  \rmC^0([0,1],H^1(\Omega))\,,
\end{align}
whence \eqref{thesis:technical-1} and \eqref{thesis:technical-2}. 
Furthermore, observe that   $A_q(\overline{\sfz}_k) \weaksto A_q(\teta) $ 
in $L^\infty(0,1;L^2(\Omega))$   and that, as $k\to\infty$,  
\begin{equation}
\label{convs-palla-5}
\begin{aligned}
  \|   \rmD_z \tildene{\overline{t}_{\tau_k} \circ \sft_k}{\overline{\sfz}_k}  
- 
 \rmD_z \tildene{t}{\teta}
  \|_{L^\infty(0,1;L^2(\Omega))} 
 \stackrel{(1)}{  \leq} 
C\sup_{s\in [0,1]} \left( |\overline{t}_{\tau_k} ( \sft_k(s)) - t| + \| 
\overline{\sfz}_k(s)- \teta(s)\|_{L^6(\Omega)} \right) \stackrel{(2)}{\to} 
0
\end{aligned}
\end{equation}
\end{subequations}
with (1) due to \eqref{enhanced-stim-7}, and convergence (2) due to 
\eqref{convs-palla-3}, joint with the fact that $\sup_{s\in [0,1]} | 
\sft_k(s))-t|\to 0$ as $\sft_k$ takes values in the interval 
$[\alpha_k,\beta_k]$ which shrinks to $\{t\}$. 
 All in all, $  \rmD_z \ene{\overline{t}_{\tau_k} \circ 
\sft_k}{\overline{\sfz}_k} \weaksto  \rmD_z \ene{t}{\teta}$ in 
$L^\infty(0,1;L^2(\Omega))$. 
It follows from  estimates \eqref{palla1/2} and convergences \eqref{convs-palla} 
that $\teta \in \mathcal{T}_t^{\bar{\varrho}} (z_0,z_1)$. 
It remains to conclude \eqref{thesis:technical-4}.  For this limit passage, we 
rely on convergences \eqref{convs-palla} and refer the reader to the proof of 
\cite[Prop.\ 7.1]{MRS16}, cf.\ also \cite[Thm.\ 5.1]{krz}, \cite[Thm.\ 
7.4]{KRZ2}. %\footnote{should we give more details here???}
\par
This finishes the proof of Proposition \ref{prop:technical}. 
\end{proof}
\medskip

 We continue this section by carrying out the 
\underline{\bf proof of Proposition \ref{prop:ch-rule}}, by suitably adapting the argument for  
the chain-rule result \cite[Thm.\ 3.13]{MRS16}.  
 From now on, we will suppose that $t_0=0$ and $t_1=T$ for the sake of 
simplicity. 
 Let $\varrho>0$ fulfill \eqref{admissible-threshold}.
\par
First of all,  for any $z \in \BV([0,T];L^2(\Omega))$  fulfilling 
the conditions of the statement we construct 
a \emph{parameterized} curve $(\mathsf{t},\mathsf{z}): [0,\mathsf{S}] \to 
[0,T]\times \calZ$ with the following properties: 
\[
z(t) \in \{ \mathsf{z}(s)\, : \mathsf{t}(s) =t\}
\]
and
\begin{itemize}
\item[-] $\mathsf{t}$ is non-decreasing, surjective, Lipschitz,
\item[-] $\mathsf{z} \in L^\infty (0,\mathsf{S}; \calZ) \cap \AC 
([0,\mathsf{S}]; L^2(\Omega))$ and $\rmD_z\calI (\cdot,\mathsf{z}(\cdot)) \in 
L^\infty(0,\mathsf{S};L^2(\Omega))$.
\end{itemize}
The integrability and regularity requirements on $\mathsf{z}$ coincide with  
those on admissible transition curves, cf.\ Definition \ref{def:curves+cost}. 
Hence, we will call $(\mathsf{t},\mathsf{z})$  \emph{admissible parameterized 
curve}. 
We borrow the  construction of $(\sft,\sfz)$, starting from the $\BV$-curve $z$, 
from the proof of \cite[Prop.\ 4.7]{MRS16}: first, we introduce the 
parameterization 
\[
\sfs(t): =t + \pVar{L^2(\Omega)}z{0}{t}, \quad 
\mathsf{S}: = \sfs(T).
\] 
We define 
\[
\sft: = \sfs^{-1}: [0,\mathsf{S}]\setminus I \to [0,T], \qquad \sfz:= z \circ \sft,
\]
where the set $I$ is given by $I= \cup_n I_n$,
with $I_n =(\sfs(t_{n-}), \sfs(t_{n+}))$ and the points $(t_n)_n$ constitute 
the countable jump set of $z$, which in fact coincides with the jump set of 
$\sfs$. 
%\CCC 
We extend $\sft$ and $\sfz$ to $I$ by setting
\[
\sft(s): = t_n, \quad \sfz(s): = \teta_n(\mathsf{r}_n(s)) \quad \text{if } s \in I_n,
\]
with $\mathsf{r}_n: \overline{I_n} \to [0,1]$ the unique affine and strictly 
increasing  function from $I_n$ to $[0,1]$ and 
$\teta_n \in \mathcal{T}_{t_n}^\varrho (z(t_{n-}), z(t_{n+}))$ an admissible 
transition curve satisfying $\teta_n(\mathsf{r}_n(\sfs(t_n)))=z(t_n)$ and the 
optimality condition % (cf.\ Proposition \ref{prop:technical})
\[
\int_0^1 \mathfrak{f}_{t_n}(\teta_n(r),\teta_n'(r)) \dd r =  
\Costn{\mathfrak{f}}{\varrho}{t_n}{z(t_{n-})}{z(t_n)} + 
\Costn{\mathfrak{f}}{\varrho}{t_n}{z(t_{n})}{z(t_{n+})}\,.
\]
The existence of such an optimal transition follows from Proposition 
\ref{prop:technical}(1). Indeed, let $t_*\in \rmJ_z$.
Observe  that in $(t_*, z(t_{*\pm}))$ the assumptions of the proposition 
 are satisfied, which can be seen as follows.
First of all, $\Costn{\mathfrak{f}}{\varrho}{t_n}{z(t_{n-})}{z(t_n)}<\infty $ and 
$\Costn{\mathfrak{f}}{\varrho}{t_n}{z(t_{n})}{z(t_{n+})}<\infty$ since $\pVarn{\mathfrak{f}}{\varrho}z0T<+\infty$.  Moreover, choose a 
sequence $s_k\to t_{*-}$ for $k\to\infty$ such that the assumptions of Prop.\ 
\ref{prop:technical}(1) are 
satisfied along this 
sequence and such that $z(s_k)\rightharpoonup z(t_{*-})$ in $\calZ$. 
Consequently, by Corollary \ref{coro-fre}, $\rmD_z\wt\calI(s_k,z(s_k))\to   
\rmD_z\wt\calI(t_*,z(t_{*-}))$ and $\norm{A_q(z(s_k))}_{L^2(\Omega)}\leq C$, 
which translates into a uniform bound of the sequence $(z(s_k))_k$ in 
$W^{1+\sigma,q}(\Omega)$ for $0<\sigma<\frac1q$, cf.\
Proposition \ref{prop:Savare98}. Thus, we finally 
conclude that $\rmD_z\calI(t_*,z(t_{*-}))\in L^2(\Omega)$ and that 
$\norm{z(t_{*-})}_\calZ + \norm{\rmD_z\calI(t_*,z(t_{*-}))}_{L^2(\Omega)}\leq 
\varrho$.  A similar argument applies to $t_{*+}$. 

By construction, $\sfz \in W^{1,\infty}(0,\rmS;L^2(\Omega))$. Indeed, let 
$s_1<s_2\in [0,\rmS]$ and $\sigma_i:=\sft(s_i)$. Hence, 
$s_i=\sigma_i +\pVar{L^2(\Omega)}z{0}{\sigma_i}$. This implies that 
\[
 \norm{\sfz(s_1) - \sfz(s_2)}_{L^2(\Omega)}\leq 
 \abs{\sigma_2 +\pVar{L^2(\Omega)}z{0}{\sigma_2} -(\sigma_1 
+\pVar{L^2(\Omega)}z{0}{\sigma_1})  }=\abs{s_2-s_1}\,.
\]
Hence, altogether 
 $(\sft,\sfz)$ is an admissible parameterized curve. 

By repeating the very same calculations as in the proof of  \cite[Prop.\ 
4.7]{MRS16}, we may show that 
\begin{equation}
\label{ad-cr1}
 \pVarn{\mathfrak{f}}{\varrho}z{0}{T} = \int_0^{\mathsf{S}} \mathfrak{f}_{\sft(s)}(\sfz(s),\sfz'(s)) \dd s\,.
\end{equation}
\par
Secondly, we observe that the chain rule from Lemma \ref{l:ch-rule} (cf.\ also Remark \ref{rmk:alternative-ch-requir}) extends to the admissible parameterized curve $(\sft,\sfz)$, yielding
\[
\frac{\dd}{\dd s} \calI(\sft(s),\sfz(s)) -\partial_t  \calI(\sft(s),\sfz(s)) \sft'(s)  = \int_\Omega \rmD_z  \calI(\sft(s),\sfz(s))  \sfz'(s) \dd x \qquad \foraa\, s \in (0,\mathsf{S})\,.
\]
Therefore, with a simple calculation (cf.\ also the proof of \cite[Thm.\ 
4.4]{MRS16}) we infer that
\begin{equation}
\label{ad-cr2}
\left| \frac{\dd}{\dd s} \calI(\sft(s),\sfz(s)) -\partial_t  \calI(\sft(s),\sfz(s)) \sft'(s) \right| \leq  \mathfrak{f}_{\sft(s)}(\sfz(s),\sfz'(s))    \qquad \foraa\, s \in (0,\mathsf{S})\,.
\end{equation}
\par
Combining \eqref{ad-cr1} \& \eqref{ad-cr2} we obtain the desired chain-rule inequality \eqref{ch-ineq-dam}.  
\QED 
%  $\mathsf{t}: [0,\mathsf{S}] \to [0,T]$ 

We are now in a position to give the \underline{\textbf{proof of Theorem \ref{thm:van-visc-eps-tau}}}. 
We will split the proof in several steps and give some intermediate results. Let us mention in advance that, in their statements, we will always tacitly suppose that  Assumptions \ref{ass:domain}, \ref{assumption:energy}, and \ref{ass:load}, as well as condition \eqref{cvg-initial}, from  Theorem \ref{thm:van-visc-eps-tau} hold. 
More precisely,
\begin{itemize}
\item[-]
 we start by
 fixing the compactness properties of the sequences $(\pwc z{\tau_k}{\epsi_k})_k,\, (\pwl z{\tau_k}{\epsi_k})_k$ in Lemma \ref{l:compactn} below. 
 \item[-]
Throughout Steps $1$--$3$ we show that any limit curve $z$ of $(\pwc z{\tau_k}{\epsi_k})_k,\, (\pwl z{\tau_k}{\epsi_k})_k $ complies with the local stability \eqref{eq:65bis} and with the energy-dissipation inequality \eqref{one-sided-0T},
obtained by passing to the limit in its discrete counterpart \eqref{discr-enineq}. By virtue of Corollary \ref{cor:this-only}  we 
  thus conclude that $z$ is a Balanced Viscosity solution to the rate-independent system \eqref{dndia}.
  \item[-]
Steps 4 \& 5 are devoted to finalizing the proof of convergences \eqref{convergences}, and to showing that $z$ is a $H^1(\Omega)$-parameterizable solution, cf.\
 \eqref{H1-param}.
 \end{itemize}
 \medskip
 
\paragraph{\bf Step $0$: Compactness.} 
We prove the following
\begin{lemma}
\label{l:compactn}
Let $(\tau_k,\epsi_k)_k$ be null sequences. There holds
\begin{equation}
\label{stability-est}
\exists\, C>0 \ \forall\, k \in \N \, : \  \sup_{t\in [0,T]}\| \pwc z{\tau_k}{\epsi_k} (t) {-} \pwl z{\tau_k}{\epsi_k} (t)\|_{H^1(\Omega)} \leq C\left( \frac{\tau_k}{\epsi_k}\right)^{1/2}\,.
\end{equation}
Suppose in addition   \eqref{crucial-seq-params}. Then, there exists a curve $z \in L^\infty (0,T;\calZ) \cap \BV ([0,T];H^1(\Omega))$ such that, up to a (not relabeled) subsequence, the following convergences hold:
\begin{subequations}
\label{cvgs-e-t-zero}
\begin{align}
&
\label{cvg-1}
 \pwc z{\tau_k}{\epsi_k},  \,  \pwl z{\tau_k}{\epsi_k}   \weaksto z && 
\text{ in } L^\infty (0,T;\calZ), 
 \\
&
\label{cvg-3}
 \pwc z{\tau_k}{\epsi_k}(t) ,\, \pwl z{\tau_k}{\epsi_k} (t) \to z(t) && 
\text{ in } \calZ \quad \text{ for all } t \in [0,T], 
\\
&
\label{cvg-4} 
\rmD_z \ene{\overline{t}_{\tau_k}(t)}{\pwc z{\tau_k}{\epsi_k}(t)}   \weakto 
\rmD_z \ene t{z(t)} && \text{ in } L^2(\Omega) \quad \text{ for all } t \in 
[0,T].
\end{align}
\end{subequations}
\end{lemma}
\begin{proof}
The first estimate follows from observing that  for every $t\in (0,T)$ 
\[
\| \pwc z{\tau_k}{\epsi_k} (t) {-} \pwl z{\tau_k}{\epsi_k}  (t)\|_{H^1(\Omega)} 
\leq \int_{\underline{t}_\tau(t)}^{\overline{t}_\tau(t)} \| 
\pwl{z}{\tau_k}{\epsi_k}'(r)\|_{H^1(\Omega)} \dd r \leq \tau_k^{1/2} \| 
\pwl{z}{\tau_k}{\epsi_k}'\|_{L^2(\underline{t}_\tau(t),\overline{t}_\tau(t); 
H^1(\Omega))},
\]
and then \eqref{stability-est}  is a consequence of the a priori estimate \eqref{est-epsi-tau-2}. 
\par
Convergences
\eqref{cvg-1} follow from estimate   \eqref{est-epsi-tau-1}:   %and 
 observe that the sequences $(\pwc z{\tau_k}{\epsi_k} )_k,\, (\pwl 
z{\tau_k}{\epsi_k} )_k$  converge to the same limit, weakly star in $L^\infty 
(0,T;\calZ)$, in view of  the fact that 
\begin{equation}
\label{same-limit-z}
\| \pwc z{\tau_k}{\epsi_k}  {-} \pwl z{\tau_k}{\epsi_k} \|_{L^\infty(0,T;H^1(\Omega))} \to 0 
\end{equation}
as $k\to\infty$ by  \eqref{stability-est}   combined with condition  
\eqref{crucial-seq-params} on the sequences $(\tau_k,\epsi_k)_k$.
\par
It follows from estimate \eqref{est-epsi-tau-3} that the sequences   $(\pwc 
z{\tau_k}{\epsi_k} )_k,\, (\pwl z{\tau_k}{\epsi_k} )_k$  are   bounded in $ 
\BV([0,T];H^1(\Omega))$. Due to
% an infinite-dimensional version of the Helly 
%Theorem, cf.\ e.g. 
the previously mentioned 
\cite[Thm.\ 6.1]{MieThe04RIHM}, up to a subsequence they 
pointwise converge on $[0,T]$, w.r.t.\ the weak $H^1(\Omega)$-topology, to (the 
same, by \eqref{same-limit-z})  function $\tilde z$.  Now, by the 
additional estimate  \eqref{est-epsi-tau-4}, 
$(\pwc z{\tau_k}{\epsi_k} )_k$ is  bounded in $L^\infty (0,T; 
W^{1+\sigma,q}(\Omega))$ for every $0<\sigma <\frac1q$,
%(cf.\ \cite[Thm.\ 2, Rmk.\ 3.5]{Savare98}, \CCC see also
cf.\ Proposition~\ref{prop:Savare98},   
and so is 
$(\pwl z{\tau_k}{\epsi_k} )_k$. Therefore, by compactness the above pointwise 
convergence  to $\tilde z$  improves to a strong convergence in $\calZ$. But 
then, $\pwc z{\tau_k}{\epsi_k},\, \pwl z{\tau_k}{\epsi_k} \to \tilde z$ in 
$L^p(0,T;\calZ)$ for every $1\leq p<\infty$, which allows us to conclude that 
$\tilde z = z$. All in all, we have obtained convergence \eqref{cvg-3}. 
\par
Finally, we address
\eqref{cvg-4}: Observe that $A_q (\pwc z{\tau_k}{\epsi_k}(t)) \to A_q(z(t))$  in $\calZ^*$ as a consequence of the strong convergence  \eqref{cvg-3}. 
A fortiori, by the $L^\infty (0,T;L^2(\Omega))$-bound on 
$(A_q (\pwc z{\tau_k}{\epsi_k}))_k$, we find that  
$A_q (\pwc z{\tau_k}{\epsi_k}(t)) \weakto A_q(z(t))$  in $L^2(\Omega)$. We combine this with \eqref{weak-continuity}, giving that $ \rmD_z \tildene{\overline{t}_{\tau_k}(t)}{\pwc z{\tau_k}{\epsi_k}(t)}  \weakto \rmD_z \tildene t{z(t)} 
$ in $L^2(\Omega)$, 
and arrive at \eqref{cvg-4}.
\end{proof}
\paragraph{\bf Step $1$:  ad the local stability \eqref{eq:65bis}.}  On the one hand, the very same argument leading to the proof of estimate 
\eqref{est-epsi-tau-1-bis} in Corollary \ref{cor:discr-enineq} also shows that 
\begin{equation}
\label{ad-1}
\sup_k \int_0^T  
\calR_{\epsi_k}^* (-\rmD_z
\calI (\overline{t}_{\tau_k}(r),\pwc{z}{\tau_k}{\epsi_k}(r))) \dd r \leq C\,.
\end{equation}
On the other hand,
 $\calR_\epsi^*$ Mosco-converges, w.r.t.\ the $L^2(\Omega)$-topology, to the indicator  functional 
\[
\mathrm{I}_{\partial\calR(0)}:L^2(\Omega)\to [0,+\infty]  \quad \text{ defined  by }   \quad
\mathrm{I}_{\partial\calR(0)}(v):= 
\begin{cases}
0  & \text{ if $v \in \partial\calR_1(0)$,}
\\
 +\infty & \text{ else.}
 \end{cases}
 \]
 Hence
 we have  in view of \eqref{cvg-4} that 
\begin{equation}
\label{ad-2}
\liminf_{k\to\infty} \calR_{\epsi_k}^* (-\rmD_z
\calI (\overline{t}_{\tau_k}(t),\pwc{z}{\tau_k}{\epsi_k}(t)))   \geq 
\mathrm{I}_{\partial\calR(0)} (-\rmD_z \ene t{z(t)}) \qquad \text{for every } t 
\in [0,T].
\end{equation}
Therefore, from \eqref{ad-1} and  \eqref{ad-2} via the  Fatou Lemma we infer that 
\[
\int_0^T  \mathrm{I}_{\partial\calR(0)} (-\rmD_z \ene t{z(t)})  \dd t<+\infty 
\qquad \text{whence} \qquad   \mathrm{I}_{\partial\calR(0)} (-\rmD_z \ene 
t{z(t)}) =0 \quad \foraa\, t \in (0,T).
\]
From this we conclude with an approximation argument  $-\rmD_z \ene t{z(t)} \in 
\partial\calR_1(0)$ for every $t\in [0,T]\setminus\mathrm{J}_z$, and  that 
$-\rmD_z \ene t{z(t_\pm)}  \in \partial \calR_1(0) $ for every $t\in 
\mathrm{J}_z$, i.e.\  \eqref{eq:65bis}.
\medskip

\paragraph{\bf Step $2$: the key lower semicontinuity inequality.}
We aim to prove the following 
\begin{lemma}
For every $0\leq s \leq t \leq T$ there holds
\begin{equation}
\label{key-lsc}
\liminf_{k\to\infty}\int_{\underline{t}_{\tau_k}(s)}^{\overline{t}_{\tau_k}(t)}  
\calR_{\epsi_k}(\pwl z{\tau_k}{\epsi_k}'(r)) \dd r + \calR_{\epsi_k}^* (-\rmD_z
\calI (\overline{t}_{\tau_k}(r),\pwc{z}{\tau_k}{\epsi_k}(r)))  \dd r \geq   
\pVarn {\mathfrak{f}}{\bar\varrho}zst 
%\qquad \text{for all } 0 \leq s \leq t \leq T.
\end{equation}
with $\bar\varrho$ given by 
\[
\begin{aligned}
\bar\varrho: = \sup_k \Big( &  \int_{0}^{T} 
\left(
\calR_{\epsi_k}(\widehat{z}_k'(r)) {+} \calR_{\epsi_k}^* ({-}  \rmD_z 
\ene{\overline{t}_{\tau_k}(r)}{\overline{z}_k(r)}) \right) \dd r + \|  
\widehat{z}_k \|_{L^\infty(0,T;\calZ)  \cap W^{1,1} (0,T; 
H^1(\Omega))}
\\
& \qquad \qquad  + \| \overline{z}_k\|_{L^\infty (0,T;\calZ)} 
+ \|   \rmD_z \ene{\overline{t}_{\tau_k}}{\overline{z}_k})  \|_{L^\infty(0,T;L^2(\Omega))} \Big)
\end{aligned}
\]
\end{lemma}
\begin{proof}
Along the footsteps of  the \cite[proof of Thm.\ 7.3]{MRS16},
 we introduce the non-negative Borel measures on $[0,T]$
\[
\nu_k :=  \left(  \calR_{\epsi_k}(\pwl z{\tau_k}{\epsi_k}')  + \calR_{\epsi_k}^* (-\rmD_z
\calI (\overline{t}_{\tau_k},\pwc{z}{\tau_k}{\epsi_k})) \right) \mathscr{L}^1,
\]
with $\mathscr{L}^1$ the Lebesgue measure. It follows from estimate \eqref{est-diss-epsi-tau} that the sequence $(\nu_k)_k$ is bounded in the space of Radon measures, hence there exists a positive measure $\nu$ such that $\nu_k \weaksto \nu$ as $k\to\infty$. 
Like in the proof of \cite[Thm.\ 7.3]{MRS16}, we   observe that for every 
interval $[a,b]\subset [0,T]$
\[
\begin{aligned}
\nu([a,b]) \geq \limsup_{k\to\infty} \nu_k([a,b])  & \geq \limsup_{k\to\infty} \int_a^b 
\left( \calR_{\epsi_k}(\pwl z{\tau_k}{\epsi_k}'(r)) + \calR_{\epsi_k}^* (-\rmD_z
\calI (\overline{t}_{\tau_k}(r),\pwc{z}{\tau_k}{\epsi_k}(r))) \right) \dd r
\\
 &  \geq \liminf_{k\to\infty}   \int_a^b 
\calR_{\epsi_k}(\pwl z{\tau_k}{\epsi_k}'(r)) \dd 
r 
\\
&
\geq \liminf_{k\to\infty} \Var{\calR_1}{\pwc{z}{\tau_k}{\epsi_k}}ab \stackrel{(1)}{\geq} \Var{\calR_1}zab\stackrel{(2)}{\geq} \mu_{\mathrm{d}}([a,b]),
\end{aligned}
\]
where (1) follows from the pointwise convergence \eqref{cvg-3} and the lower semicontinuity of the variation functional $\Varname{\calR_1}$, and (2) from the definition \eqref{variation-functional} of the measure $\mu$.
We thus conclude that 
\begin{equation}
\label{diffuse-inequality}
\nu \geq \mu_{\mathrm{d}}.
\end{equation} 
\par
We now check 
\begin{equation}
\label{jump-inequality}
\nu(\{ t\})  \geq  \Costn{\mathfrak{f}}{\bar\varrho}t{z(t_-)}{z(t)} + 
\Costn{\mathfrak{f}}{\bar\varrho}t{z(t)}{z(t_+)} 
\geq \mu_{\mathrm{J}} (\{t\})
\qquad \text{for every } t \in \mathrm{J}_z.
\end{equation}
With this aim, for fixed $t\in \mathrm{J}_z$ let us fix two  sequences $\alpha_k 
\uparrow t$ and $\beta_k\downarrow t$ such that 
\[
\begin{cases}
\pwc z{\tau_k}{\epsi_k}(\alpha_k) \to z(t_-),
\\
\pwc z{\tau_k}{\epsi_k}(\beta_k) \to z(t_+)
\end{cases}
\quad \text{
 in $\calZ$  as $k\to\infty$.}
 \]
% It follows from  that
Thus we have
 \[
  \limsup_{k\to\infty} \nu_k ([\alpha_k,\beta_k]) \geq 
\liminf_{k\to\infty} \int_{\alpha_k}^{\beta_k}  \left(  \calR_{\epsi_k}(\pwl 
z{\tau_k}{\epsi_k}'(r)) + \calR_{\epsi_k}^* (-\rmD_z
\calI (\overline{t}_{\tau_k}(r),\pwc{z}{\tau_k}{\epsi_k}(r))) \right) \dd r  
 \stackrel{(1)}{\geq} 
\Costn{\mathfrak{f}}{\bar\varrho}t{z(t_-)}{z(t_+)}\,, 
% \geq    
%\mu_{\mathrm{J}} (\{t\}),
 \] 
 where (1) ensues from 
Proposition \ref{prop:technical}, applying \eqref{thesis:technical} with the choices 
$\overline{z}_k : =\pwc z {\tau_k}{\epsi_k}$, $\widehat{z}_k: = \pwl z {\tau_k}{\epsi_k}$.
With analogous  arguments %as in the proof of  \cite[Thm.\ 7.3]{MRS16}, we also 
check
we check  that 
\begin{equation}
\label{l-r-jumps}
\liminf_{k\to\infty} \nu_k ([\alpha_k,t]) \geq 
\Costn{\mathfrak{f}}{\bar\varrho}t{z(t_-)}{z(t)}, \qquad 
\liminf_{k\to\infty} \nu_k ([t,\beta_k]) \geq 
\Costn{\mathfrak{f}}{\bar\varrho}t{z(t)}{z(t_+)}\,.
\end{equation}
All in all, we have 
\[
\begin{aligned}
\nu(\{ t\})  \stackrel{(1)}{\geq}   \limsup_{k\to\infty} \nu_k ([\alpha_k,\beta_k]) \geq  
\liminf_{k\to\infty} \nu_k ([\alpha_k,t])  +   \liminf_{k\to\infty} \nu_k ([t,\beta_k])  & \geq  
 \Costn{\mathfrak{f}}{\bar\varrho}t{z(t_-)}{z(t)} + 
\Costn{\mathfrak{f}}{\bar\varrho}t{z(t)}{z(t_+)}  \\ & \stackrel{(2)}{\geq} 
  \mu_{\mathrm{J}} (\{t\}),
 \end{aligned}
\]
where  (1) is a property of the weak$^*$-convergence of measures  and (2) ensues 
from \eqref{comparison-costs}. Hence   inequality \eqref{jump-inequality} is 
proved.
\par
Combining \eqref{diffuse-inequality}, \eqref{jump-inequality}, and \eqref{l-r-jumps} and repeating the very same 
calculations as in the proof of \cite[Thm.\ 7.3]{MRS16},
we ultimately conclude \eqref{key-lsc}.
\end{proof}

 \paragraph{\bf Step $3$: ad the energy-dissipation inequality \eqref{one-sided-0T}.}
We now pass to the limit in the discrete energy-dissipation inequality \eqref{discr-enineq}, written for $s=0$ and $t=T$. For the first term on the left-hand side, 
we resort to the lower semicontinuity inequality \eqref{key-lsc} from Step $2$.
It follows from the pointwise convergence \eqref{cvg-3} and the lower semicontinuity \eqref{weak-continuity} of $\calI$ that 
\[
\liminf_{k\to\infty} \ene T{\pwl z{\tau_k}{\epsi_k}(T)} \geq \ene T{z(T)},
\]
whereas by hypothesis we have that $ \ene 0{\pwl z{\tau_k}{\epsi_k}(0)} \to \ene 0{z_0}$. 
Furthermore,  it follows from \eqref{stim3}, \eqref{stim5}, and the Lebesgue Theorem that 
\[
\lim_{k\to\infty} \int_0^T \partial_t  \ene t{\pwl z{\tau_k}{\epsi_k}(t)} \dd t  =\int_0^T\partial_t  \ene t{z(t)} \dd t\,.
\]
Finally, observe that the very last term on the right-hand side of \eqref{discr-enineq} converges to zero by virtue of estimates \eqref{est-epsi-tau} and  convergence \eqref{same-limit-z}.
\par
Thus,  \eqref{one-sided-0T} is proven with  
$\pVarn{\mathfrak{f}}{\bar\varrho}z0T$ 
 and, by virtue of Corollary \ref{cor:this-only}, 
we deduce that $z$ is a Balanced Viscosity solution to the rate-independent damage system \eqref{dndia}.
\par
 Finally,   \eqref{threshold-Var} follows from the following chain of 
inequalities (which in fact holds for every $t\in[0,T]$)
\[
\sup_{\varrho\geq \bar\varrho} \pVarn{\mathfrak{f}}{\varrho}{z}0T\stackrel{(1)}{=}  \pVarn{\mathfrak{f}}{\bar\varrho}{z}0T
\stackrel{(2)}{=}   \calI(0,z(0)) -\ene T{z(T)} +\int_0^T \partial_t \calI(s,z(s)) \dd s \stackrel{(3)}{\leq}  \inf_{\varrho\geq \bar\varrho} \pVarn{\mathfrak{f}}{\varrho}{z}0T,
\]
with (1) due to \eqref{monoton-cost}, (2) to \eqref{eq:84} involving the total  
variation functional $ \pVarn{\mathfrak{f}}{\bar\varrho}{z}0T$, and (3) from the 
chain-rule inequality \eqref{ch-ineq-dam} (observe that $\varrho$ therein is 
arbitrary, provided it fulfills \eqref{admissible-threshold}).  
\medskip

\paragraph{\bf Step $4$:  ad convergences \eqref{convergences}.} The convergences of the energies $(\ene t{\pwc z{\tau_k}{\epsi_k}(t)})_k$  follows from the 
pointwise  convergence \eqref{cvg-1} of $(\pwc z{\tau_k}{\epsi_k}(t))_k$. 
In order to prove the convergence of $(\ene t{\pwl z{\tau_k}{\epsi_k}(t)})_k$
and  of the dissipation integrals in \eqref{cvg-3-b}, we repeat the very same 
arguments as in the proof of \cite[Thm.\ 3.11]{MRS16}. 
 \medskip
 
 \paragraph{\bf Step $5$: ad \eqref{H1-param}.}  We may repeat the proof of
  \cite[Thm.\ 3.22]{MRS16},  to which we refer the reader, relying on 
Proposition \ref{prop:technical}(3).
\par
This concludes the proof of Theorem \ref{thm:van-visc-eps-tau}.
\QED

\begin{appendix}
 \section{Some references on elliptic regularity}

For $d\geq 2$ let $\Omega\subset\R^d$ be a bounded $\rmC^{1,1}$-domain with Dirichlet boundary $\partial\Omega$. 
Let further $\bbC$ satisfy \eqref{elast-tensor}. 

Reference \cite[Theorem 3]{Val78}, see also \cite[Theorem 7.1]{MR03_polyhedral}, yields 
\begin{theorem}
\label{app.reg.thm1}
 For every $p\in (1,\infty)$ the operator $L_\bbC: W_0^{1,p}(\Omega)\to W^{-1,p}(\Omega)$ is a continuous isomorphism.
\end{theorem}
Moreover, 
Theorem 10.5 from \cite{ADN64} 
(there it is assumed that the domain has a $\rmC^2$-boundary, but the coefficients need  
to be continuous, only,  instead of Lipschitz continuous) provides the following a priori estimate:
\begin{theorem}
 For every $p\in(1,\infty)$ there exist constants $c_p,\wt c_p>0$ such 
 that for every $u\in W^{2,p}(\Omega)\cap W_0^{1,p}(\Omega)$ 
 it holds
 \begin{align}
 \label{app.est.1}
  \norm{u}_{W^{2,p}(\Omega)}\leq c_p\left(\norm{L_\bbC u}_{L^p(\Omega)} + \wt c_p \norm{u}_{L^p(\Omega)}\right)\,.
 \end{align}
\end{theorem}
Thanks to Theorem \ref{app.reg.thm1}, for every $p\in (1,\infty)$ the operator
\begin{align}
\label{app.est.2}
 L_\bbC: W^{2,p}(\Omega)\cap W_0^{1,p}(\Omega)\to L^p(\Omega), \quad u\mapsto -\Div( \bbC\varepsilon(u))
\end{align}
is injective, which implies that estimate \eqref{app.est.1} is valid with $\wt 
c_p=0$ and that $L_\bbC$ has a closed range.   
By \cite[Chapter 3.5.5]{Kato84}, %problem 5.27
one finally concludes that   the operator $L_\bbC$ from \eqref{app.est.2} is 
surjective for every $p\in (1,\infty)$. 
This finally results in 
\begin{theorem}
\label{app.reg.thm2}
 For every $p\in(1,\infty)$ the operator in \eqref{app.est.2} is a continuous isomorphism. 
\end{theorem}
\end{appendix}

\subsubsection*{Acknowledgment}
This project was partially supported by the  GNAMPA (INDAM). D.\ Knees  
acknowledges the partial financial support through the  DFG-Priority Program 
SPP 1962 \textit{Non-smooth and Complementarity-based Distributed Parameter 
Systems: Simulation and Hierarchical Optimization}. 
C.\ Zanini and D.\ Knees acknowledge the great hospitality of the University of 
Brescia. 
%%%%

{\small 
%\bibliographystyle{alpha}
%\bibliography{literature_CDR2}

} 
\end{document}